\documentclass[12pt]{amsart}
\usepackage{geometry}                
\usepackage{graphicx}
\usepackage{amssymb}
\usepackage{epstopdf}
\usepackage{amsmath}
\usepackage{color}
\usepackage{hyperref}
\usepackage{pdfsync}
 \usepackage{tikz}
 
\usepackage[all]{xy}
\xyoption{matrix}
\xyoption{arrow}

\usetikzlibrary{arrows,decorations.pathmorphing,backgrounds,positioning,fit,petri}

 \tikzset{help lines/.style={step=#1cm,very thin, color=gray},
help lines/.default=.5} 
\tikzset{thick grid/.style={step=#1cm,thick, color=gray},
thick grid/.default=1} 

\newtheorem{thm}{Theorem}[section]
\newtheorem{lem}[thm]{Lemma}
\newtheorem{cor}[thm]{Corollary}
\newtheorem{prop}[thm]{Proposition}

\theoremstyle{definition}
\newtheorem{defn}[thm]{Definition}
\newtheorem{eg}[thm]{Example}

\theoremstyle{remark}
\newtheorem{rem}[thm]{Remark}
 \newtheorem{formula}[thm]{Formula}

\numberwithin{equation}{section}

\newcommand{\mat}[1]{\ensuremath{
\left[\begin{matrix}#1
\end{matrix}\right]
}}
 
\newcommand{\vs}[1]{\vskip .#1 cm} 

 \newcommand{\onto}{\twoheadrightarrow}

\newcommand{\xxrarrow}[1]{\ensuremath{
\mathop{\longrightarrow}\limits^{#1}}}

 \newcommand{\ot}{\leftarrow}

\DeclareMathOperator{\im}{im}
\DeclareMathOperator{\Hom}{Hom}%
\DeclareMathOperator{\undim}{\underline{dim}}
 
\newcommand{\field}[1]{\mathbb{#1}}
\newcommand{\ZZ}{\ensuremath{{\field{Z}}}}
\newcommand{\CC}{\ensuremath{{\field{C}}}}
\newcommand{\RR}{\ensuremath{{\field{R}}}}

\newcommand{\commentout}[1]{}

\newcommand{\cF}{\ensuremath{{\mathcal{F}}}}

\newcommand{\cG}{\ensuremath{{\mathcal{G}}}}

\newcommand{\cP}{\ensuremath{{\mathcal{P}}}}

\newcommand{\cX}{\ensuremath{{\mathcal{X}}}}
\newcommand{\cY}{\ensuremath{{\mathcal{Y}}}}

\newcommand\vare{\varepsilon}

\title{Generalized Grassmann invariant - redrawn}
\author{Kiyoshi Igusa}
\address{Department of Mathematics, Brandeis University, Waltham, MA 02454}\email{igusa@brandeis.edu}
\thanks{Supported by the Simons Foundation}

\subjclass[2020]{
19J10; 16G20}

\keywords{Pictures, Steinberg group, algebraic K-theory, pseudoisotopy, picture groups, torsion classes, maximal green sequences, ghost modules.}

\begin{document}

\begin{abstract} 
{\color{blue} This is my old unpublished paper called ``The generalized Grassmann invariant''. It shows how ``pictures'' also known as ``Peiffer diagrams'' represent elements of $H_3G$ for any group $G$ and shows that $K_3(\ZZ [G])$ is isomorphic to a group of deformation classes of pictures for the Steinberg group of $\ZZ [G]$. A picture representing an element of order $16$ in $K_3(\ZZ)\cong \ZZ_{48}$ is also constructed. In this updated version of the paper, we modify only the pictures and leave the text more or less unchanged. 

We also added an Appendix to explain the new pictures using representations of quivers and root systems of type $A_n$. Often, some roots are missing in the Morse pictures. We give two ideas to replace these roots. One uses ``ghost handle slides'' to obtain a standard picture. The second idea uses the (real) Cartan subalgebra $H$ to obtain a ``relative'' picture for a torsion class and adds ``ghost modules'' which are directly related to the generalized Grassmann invariant. 

Additions and changes are in blue except the pictures are black with colored ghosts.}
\end{abstract}

\maketitle

\tableofcontents

\section*{Introduction}

{The purpose of this paper is two-fold. First, we give an elementary proof of the existence of an ``exotic'' element of $K_3(\ZZ)$. Second, we give a statement and partial proof of the final correct version of the Hatcher-Wagoner result on $\pi_1\cP(M,\partial M)$ for $\dim M\ge 5$. {\color{blue}We also added an Appendix to explain why the pictures have been redrawn.}

The extra element of $K_3(\ZZ)$ is given explicitly as a 3-cycle in $St(\ZZ)$ which is derived in a natural way from a nonbounding 3-chain in $W(\pm1)$. This 3-cycle determines a homology class which is detected by an essentially geometric invariant $\chi:K_3(\ZZ)\to \ZZ_2$ which we call the \emph{Grassmann invariant}. Since this invariant is zero on the image of $H_3(W(\pm1))$, we have a nontrivial element of $K_3(\ZZ)$.}

{
The Grassmann invariant can be generalized to a homomorphism
\[
    \chi:K_3(\ZZ[\pi])\to K_1(\ZZ [\pi],\ZZ_2[\pi])\cong H_0(\pi;\ZZ_2[\pi]).
\]
The definition of $\chi$ comes from very intrinsic geometric considerations, but unfortunately, the algebraic analogue is rather clumsy. Since the kernel of $\chi$ contains the image of $\Omega_3^{fr}(B\pi)\cong \pi_3^(B\pi\cup pt)$ we can define a map on ``Whitehead groups'' $\chi_{Wh}:Wh_3(\pi)\to Whi_1^+(\pi;\ZZ_2)$ where $Wh_3(\pi)=K_3(\ZZ[\pi])/\Omega_3^{fr}(B\pi)+K_3(\ZZ)$ and 
\[
    Wh_1^+(\pi;\ZZ_2)=K_1^+(\ZZ[\pi],\ZZ_2[\pi])/K_1^+(\ZZ,\ZZ_2)
\]
where $K_1^+(\ZZ,\ZZ_2)=\ZZ_2=\chi(K_3(\ZZ))$. The elements of $Wh_1^+(\pi;\ZZ_2)$ in the image of $\chi_{Wh}$ are those that die in pseudoisotopy. Thus, if $M$ is a compact {\color{blue}smooth} manifold of dimension $\ge5$ and $\pi_1M=\pi,\pi_2M=0$ we have an exact sequence
\[
    	Wh_3(\pi)\to Wh_1^+(\pi;\ZZ_2)\to \pi_0\cP(M,\partial M)\to Wh_2(\pi)\to 0.
\]
}

{\section{The space of pictures for $H_3G$}
Suppose that $G$ is a group with a presentation $G=\left<\cX|\cY\right>$. Then we can construct a free $G$-resolution of $\ZZ$ whose first terms are
\[
\xymatrix{
0 & 
	\ZZ \ar[l] & 
	\ZZ[G] \ar_\varepsilon[l]& 
	\ZZ[G]\left<\cX\right> \ar_{\partial_1}[l]& 
	\ZZ[G]\left<\cY\right>. \ar_{\partial_2}[l]
}
\]
The groups and maps are defined by
\[
\begin{array}{ccl}
 \ZZ[G]\left<\cX\right> &  = & \text{the free $G$-module generated by symbols $[x]$ where $x\in\cX$}  \\
 \partial_1[x]&=&x-1  \\
\partial_2[x_1\cdots x_n]  &  = &  [x_1]+x_1[x_2]+\cdots+x_1x_2\cdots x_{n-1}[x_n] \\
\,[x^{-1}] & = & \text{$-x^{-1}[x]$ if $x\in\cX$}\\
 \varepsilon(g) &  = & 1\text{ for all }g\in G\subset\ZZ[G].  \\
\end{array}
\]

The exactness of the above sequence is well-known and can be derived from the fact that it forms part of the augmented $G$-equivariant chain complex for $\widetilde{BG}$, the universal covering space of $BG$, where $BG$ is constructed in the obvious way with one 0-cell, a 1-cell for every element of $\cX$ and a 2-cell for every element of $\cY$. One must of course add more cells of dimension $\ge3$ but this can be done arbitrarily.

By considering long exact sequences we get $H_3G\cong H_2(G;\ker\vare)\cong H_1(G;\ker\partial_1)$ which fits into the exact sequence
\[
0\to H_3G\to H_0(G;\ker\partial_2)\to H_0(G;\ZZ[G]\left<\cY\right>)\to H_0(G;\ZZ[G]\left<\cX\right>)\to H_0(G;\ker\partial_1)\to 0.
\]
Thus, $H_3G$ is essentially contained in $\ker\partial_2$. We shall show that every element of $\ker\partial_2$ can be represented by a planar graph with certain labels on the edges and vertices. $\ker\partial_1$ can be identified with $R/R'$ where $R$ is the kernel of the obvious map $F\to G$ where $F$ is the free group generated by $\cX$.

We will assume that $\cY$ is a \emph{reduced} set of relations for $G$. That is, $\cY$ and $\cY^{-1}$ are disjoint in $F$. Since $1$ is the only element of $F$ which is its own inverse, every set of relations contains a reduced set of relations.
}

{
\begin{defn}\label{def: 1.1}
Let $P(G)$ be the set of finite planar graphs together with the following additional data.
\begin{enumerate}
\item[a)] Every edge should be oriented and labeled with an element of $\cX$.
\item[b)] At every vertex we get an element of $F$ up to cyclic permutation by reading the labels of the incident edges in a counter-clockwise direction around the vertex, the label should be inverted if the corresponding edge is oriented outward. This word should be an element of $\cY$ or an inverse of an element of $\cY$ up to cyclic permutation of the letters.
\item[c)] In the case where different elements of $\cY$ are cyclic permutations of each other or when an element of $\cY$ is a nontrivial cyclic permutation of itself (e.g. $x_1x_2x_1x_2$) a \emph{base point direction} must be indicated at the vertex to indicate the starting point of the word.
\item[d)] Two graphs are equivalent if there is an orientation preserving self-homeomorphism of the plane which takes one graph to the other and preserves all the data above.
\end{enumerate}
\end{defn}
}

{
\begin{eg}\label{eg: 1.2}
Let $G=\ZZ_2=\left<x|x^2\right>$. The following graph with labels is an element of $P(G)$.
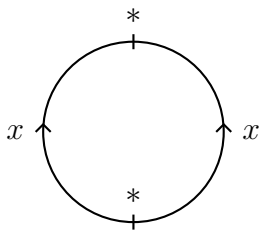
\begin{center}
\begin{tikzpicture}
\draw[thick] (0,0) circle[radius=12mm];
\draw[thick] (0,1.1)--(0,1.3)node[above]{$\ast$};
\draw[thick] (0,-1.3)--(0,-1.1) node[above]{$\ast$};
\draw[thick] (1.1,0)--(1.2,0.1)--(1.3,0) node[right]{$x$};
\draw[thick] (-1.1,0)--(-1.2,0.1)--(-1.3,0) node[left]{$x$};
\end{tikzpicture}
\end{center}
The asterisks indicate the base point directions at the two vertices. The relation at the top vertex is $x^2$ and at the bottom it is $x^{-2}$.
\end{eg}
}

{
$P(G)$ is a commutative monoid where addition is given by disjoint union and the empty graph is the identity. By modding out an equivalence relation we shall make $P(G)$ into a $G$-module with is isomorphic to $\ker\partial_2$.

\begin{defn}\label{def: 1.3}
Let $\overline P(G)$ be the quotient of $P(G)$ by the following relations which we call \emph{deformations}.
\begin{enumerate}
\item[a)] If a graph contains a circular edge with no vertices on it and nothing inside the circle, then this edge can be eliminated.
\item[b)] If two edges (or two portions of one edge) have the same label and opposite orientation they can be connected by a concordance if there is nothing between them.
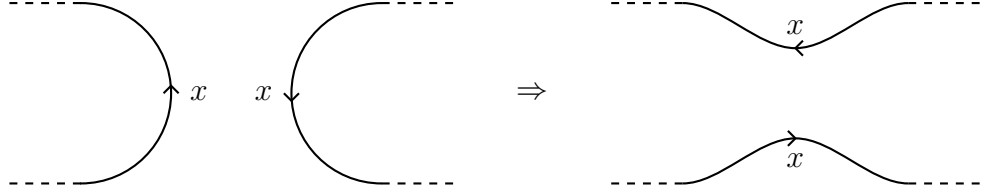
\begin{center}
\begin{tikzpicture}
\begin{scope}[xshift=-4cm]
\draw[thick,dashed] (0,1.2)--(-1,1.2);
\draw[thick,dashed] (0,-1.2)--(-1,-1.2);
\clip (0,-1.3)rectangle(2,1.3);
\draw[thick] (0,0) circle[radius=12mm];
\draw[thick] (1.1,0)--(1.2,0.1)--(1.3,0) node[right]{$x$};
\draw[thick] (-1.1,0)--(-1.2,0.1)--(-1.3,0) node[left]{$x$};
\end{scope}

\begin{scope}
\draw[thick,dashed] (0,1.2)--(1,1.2);
\draw[thick,dashed] (0,-1.2)--(1,-1.2);
\clip (0,-1.3)rectangle(-2,1.3);
\draw[thick] (0,0) circle[radius=12mm];
\draw[thick] (-1.1,0)--(-1.2,-0.1)--(-1.3,0) node[left]{$x$};
\end{scope}
\draw (2,0)node{$\Rightarrow$};
\begin{scope}[xshift=4cm]
\draw[thick,dashed] (0,1.2)--(-1,1.2) (3,1.2)--(4,1.2);
\draw[thick,dashed] (0,-1.2)--(-1,-1.2) (3,-1.2)--(4,-1.2);
\draw[thick] (0,1.2) .. controls (.5,1.2) and (1,.6) .. (1.5,.6);
\draw[thick] (3,1.2) .. controls (2.5,1.2) and (2,.6) .. (1.5,.6);
\draw[thick] (0,-1.2) .. controls (.5,-1.2) and (1,-.6) .. (1.5,-.6);
\draw[thick] (3,-1.2) .. controls (2.5,-1.2) and (2,-.6) .. (1.5,-.6);
\draw[thick] (1.4,-.5)--(1.5,-.6)--(1.4,-.7);
\draw (1.5,-.65) node[below]{$x$};
\draw[thick] (1.6,.5)--(1.5,.6)--(1.6,.7);
\draw (1.5,.65) node[above]{$x$};
\end{scope}
\end{tikzpicture}
\end{center}
\item[c)] Two vertices can be canceled if the associated relations are inverses of each other and if there is a path disjoint from the graph connecting the two base point directions.
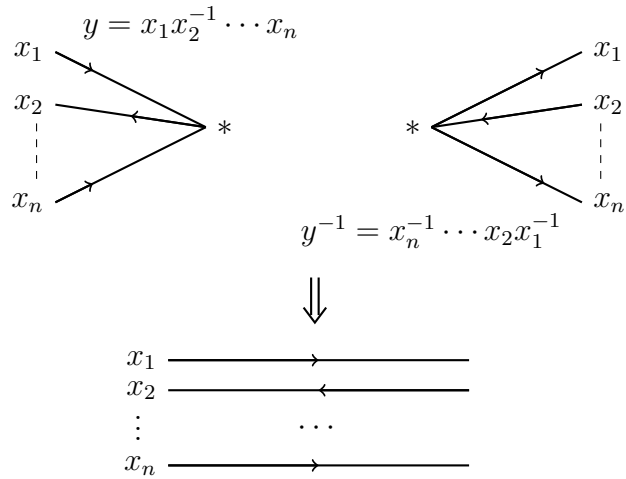
\begin{center}
\begin{tikzpicture}
\begin{scope}
	\draw[thick] (0,0) --(2,1) node[right]{$x_1$};
	\draw[thick,->] (0,0)--(1.5,.75);
	\draw[thick,->] (0,0)--(1.5,-.75);
	\draw[thick] (0,0) --(2,.3)node[right]{$x_2$};
	\draw[thick,->] (2,0.3)--(.66,.1);
	\draw[thick] (0,0) --(2,-1)node[right]{$x_n$};
	\draw[dashed] (2.25,.05)--(2.25,-.7);
	\draw(0,0) node[left]{$\ast$};
	\draw (0,-1)node[below]{$y^{-1}=x_n^{-1}\cdots x_2x_1^{-1}$};
\end{scope}
\begin{scope}[xshift=-3cm]
	\draw[thick] (0,0) --(-2,1) node[left]{$x_1$};
	\draw[thick,->] (-2,1)--(-1.5,.75);
	\draw[thick,->] (-2,-1)--(-1.5,-.75);
	\draw[thick] (0,0) --(-2,.3)node[left]{$x_2$};
	\draw[thick,->] (0,0)--(-1,.15);
	\draw[thick] (0,0) --(-2,-1)node[left]{$x_n$};
	\draw[dashed] (-2.25,.05)--(-2.25,-.7);
	\draw(0,0) node[right]{$\ast$};
	\draw (-.2,1)node[above]{$y=x_1x_2^{-1}\cdots x_n$};
\end{scope}
\begin{scope}[xshift=-1.6cm,yshift=-2.5cm]
\draw[thick] (0,0)--(0,.5);
\draw[thick] (.1,0)--(.1,.5);
\draw[thick] (-.1,.1)--(.05,-.1)--(.2,.1);
\end{scope}
\begin{scope}[xshift=-1.5cm,yshift=-4.5cm]
\draw[thick] (2,0)--(-2,0) node[left]{$x_n$};
\draw[thick,->] (-2,0)--(0,0);
\draw[thick,->] (2,1)--(0,1);
\draw[thick] (2,1)--(-2,1)node[left]{$x_2$};
\draw[thick] (2,1.4)--(-2,1.4)node[left]{$x_1$};
\draw[thick,->] (-2,1.4)--(0,1.4);
\draw (-2.4,.6) node{$\vdots$};
\draw (0,.5) node{$\cdots$};
\end{scope}
\end{tikzpicture}
\end{center}
\end{enumerate}
\end{defn}
}

{
$\overline P(G)$ is an abelian group. The negative of a graph is given by its mirror image with the same labels but opposite orientations on the edges. $F$ acts on $\overline P(G)$ on the left. The action of $x\in\cX$ on a graph is given by enclosing the graph by a circular edge oriented clockwise and labeled with $x$. The same procedure with counterclockwise orientation is the action of $x^{-1}$. The action of $xx^{-1}$ can easily be seen to be trivial by deformations (b),(a).

If $y\in\cY\subset F$ and $P\in \overline P(G)$ then the following deformation shows that $yP=P$.
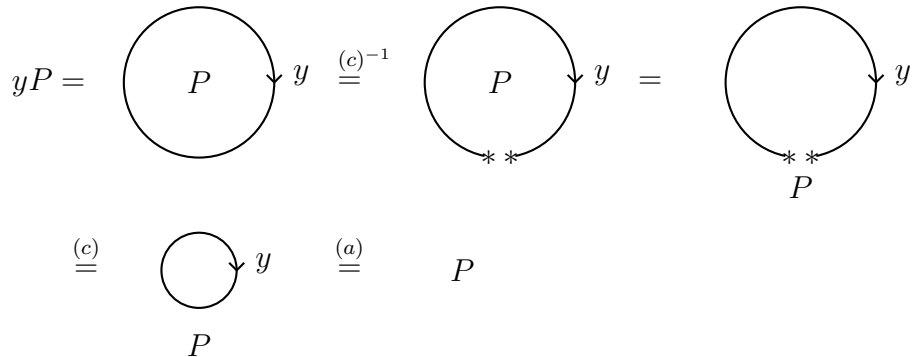
\begin{center}
\begin{tikzpicture}
\begin{scope}
	\draw[thick] (0,0) circle[radius=1cm];
	\draw (0,0) node{$P$};
	\draw[thick] (.9,.1)--(1,0)--(1.1,.1) node[right]{$y$};
	\draw (-2,0) node{$yP=$};
\end{scope}
\begin{scope}[xshift=4cm]
	\draw[thick] (0,0) circle[radius=1cm];
	\draw (0,0) node{$P$};
	\draw[thick] (.9,.1)--(1,0)--(1.1,.1) node[right]{$y$};
	\draw (-2,0) node{$=  $};
	\draw (-1.8,0) node[above]{\tiny$(c)^{-1}$};
	\draw[fill,white](0,-1) circle[radius=2mm];
	\draw (.15,-1) node{$\ast$};
	\draw (-.15,-1) node{$\ast$};
\end{scope}
\begin{scope}[xshift=8cm]
	\draw[thick] (0,0) circle[radius=1cm];
	\draw (0,-1.4) node{$P$};
	\draw[thick] (.9,.1)--(1,0)--(1.1,.1) node[right]{$y$};
	\draw (-2,0) node{$=  $};
	\draw[fill,white](0,-1) circle[radius=2mm];
	\draw (.15,-1) node{$\ast$};
	\draw (-.15,-1) node{$\ast$};
\end{scope}
\begin{scope}[yshift=-2.5cm]
	\draw[thick] (0,0) circle[radius=.5cm];
	\draw (0,-1) node{$P$};
	\draw[thick] (.4,.1)--(.5,0)--(.6,.1) node[right]{$y$};
	\draw (-1.5,0) node{$=  $};
	\draw (-1.5,0) node[above]{\tiny$(c)$};
\end{scope}
\begin{scope}[yshift=-2.5cm,xshift=3.5cm]
	\draw (0,0) node{$P$};
	\draw (-1.5,0) node{$=  $};
	\draw (-1.5,0) node[above]{\tiny$(a)$};
\end{scope}
\end{tikzpicture}
\end{center}
The second deformation is an isotopy which pushes $P$ through the gap between the two base point directions for $y,y^{-1}$. $\overline P(G)$ is thus a left $G$-module.
}

{
\begin{thm}\label{Peiffer theorem}\label{thm: 1.4}
If $\cY$ is a reduced set of relations for $G$ then $\overline P(G)\cong \ker\partial_2$ as $G$-modules.
\end{thm}

{
\color{blue}
This theorem is essentially due to Peiffer \cite{Peiffer}. What we call ``pictures'' are sometimes called ``Peiffer diagrams''. Equation $(\ast)$ below are the ``Peiffer relations''. See \cite{IT14} for more historical comments.
}\vs2

\noindent\emph{Proof}: We shall show that $\overline P(G)$ is the kernel of the natural map $\ZZ[G]\left<\cY\right>\to R/R'$. Then when we show that $R/R'\cong \ker\partial_1=\im \partial_2$ and that the two maps are consistent then we will know that $\overline P(G)\cong \ker\partial_2$.
}

{
\begin{defn}\label{def: 1.5}
Let $Q(G)$ represent the group generated by pairs $(f,y)$ where $f\in F$ and $y\in\cY$ modulo the relations
\[
(f,y)(f',y')(f,y)^{-1}=(fyf^{-1}f',y').\tag{$\ast$}
\]
Let $\varphi:Q(G)\to R$ be the homomorphism given by $\varphi(f,y)=fyf^{-1}$.
\end{defn}

\begin{lem}\label{lem: 1.6}
$\ker \varphi\cong\overline P(G)$ and it is contained in the center of $Q(G)$.
\end{lem}

\begin{proof}
It is clear that $(\ast)$ contralizes the kernel of $\varphi$. Moreover since $R$ is centerless in most cases (i.e. unless $\cX$ has only one element) we have $\ker\varphi=ZQ(G)$. Since $R$ is free, $Q(G)$ is group isomorphic to $R\times \ker\varphi$. Since $\ker\varphi$ is always abelian, the commutator subgroups $Q(G)',R'$ are isomorphic and the isomorphism is induced by $\varphi$.

It is clear from $(\ast)$ that $Q(G)/Q(G)'\cong \ZZ[G]\left<\cY\right>$ and that the induced map $\ZZ[G]\left<\cY\right>\to R/R'$ is the $G$-map which sends $y\in\cY$ to its $R'$ coset $yR'$.

An isomorphism $\psi: \overline P(G)\to \ker\varphi$ is given as follows. Let $P$ be a graph representing an element of $\overline P(G)$. From every vertex of $P$ draw a line from its base point direction to $\infty$. These lines can be chosen to be disjoint from all other vertices and from each other. To each of these lines we will associate a pair $(f,y)$ where $f\in F, y\in \cY\cup \cY^{-1}$. If we use the convention $(f,y^{-1})=(f,y)^{-1}$ we get a generator of $Q(G)$. By multiplying these together we get an element $(f_1,y_1)\cdots(f_n,y_n)\in Q(G)$ if the lines are numbered in a clockwise direction near $\infty$. This will be $\psi(P)$.

The association of the pair $(f,y)$ is given as follows. $y$ is just the relation given by the vertex. $f$ is a product of elements of $\cX\cup \cX^{-1}$ which when read from left to right give the labels from the edges which cross the line as we come from $\infty$. The label is inverted if the edge is oriented from left to right across the line.

To show that $\psi(P)$ is well-defined we must show that it is independent of the choice of the lines and their ordering, and we must show that $\psi(P)$ is invariant under deformation of $P$. An isotopy of the lines does not change $\psi(P)$ because the cyclic ordering of the lines is unchanged and the elements $f_i$ cannot change without going through relations (at the vertices). If the $i$-th line passes through the $i-1$ st vertex, i.e. if it is rechosen to go before the $i-1$ st line, then $f_i$ changes to $f_{i-1}y_{i-1}f_{i-1}^{-1}f_i$ and $\psi(P)$ is changed by the relation $(\ast)$. A sequence of such changes takes any choice to any other choice.

$\psi(P)$ is invariant under cyclic permutation of the generators $(f_i,y_i)$. This is easily seen to be true for any central element of $Q(G)$. To see that $\psi(P)\in\ker\varphi\subset ZQ(G)$ pull each vertex out to $\infty$ along the attached line. The edges will then read $f_1y_1f_1^{-1}\cdots f_ny_nf_n^{-1}$ clockwise near $\infty$. Since the vertices are all gone this element is 1 in $F$. 

$\psi(P)$ is invariant under deformation of $P$. Deformations (a), (b) obviously don't matter. For deformation (c) choose the lines for the two canceling vertices adjacent to each other. Then they will contribute canceling elements $(f,y),(f,y^{-1})$ of $Q(G)$.

If $P$ is the disjoint union of two graphs $P_1,P_2$ we can put the graphs in separate half planes and choose the lines inside the respective half planes. Then we get $\psi(P)=\psi(P_1)\psi(P_2)$.

We shall now define $\psi^{-1}:\ker\varphi\to \overline P(G)$. The graph for $\psi^{-1}((f_1,y_1)\cdots (f_n,y_n)$ should have $n$ vertices. Take any $n$ distinct points in the plane and draw disjoint lines out to $\infty$. Using the lines as base point directions put in the necessary edges with labels and orientations across the $i$-th line. Since $f_1y_1f_1^{-1}\cdots f_ny_nf_n^{-1}=1$ in $F$, the loose ends of the partial edges can be connected together without intersections while respecting the labels and orientations of the edges. The last step involves a choice which is unique up to deformations (a), (b).

The proof that $\psi^{-1}$ is a well-defined homomorphism is analogous to the proof for $\psi$. $\psi^{-1}\psi(P)$ is concordant to $P$, i.e. they are equivalent modulo deformations (a), (b). If $z\in \ker\varphi$, $\psi\psi^{-1}(z)=z$ is obvious from the definition of $\psi^{-1}$.

The $G$-equivariance of $\psi$ is left to the reader.
\end{proof} 
}

{
\begin{lem}\label{lem: 1.7}
$R/R'\cong \ker\partial_1$ as $G$-modules.
\end{lem}

\begin{proof}
Define a homomorphism $\partial:R\to \ZZ[G]\left<\cX\right>$ with the same formula as $\partial_2$, i.e. if $r=x_1\cdots x_n$, $\partial r=[x_1]+x_1[x_2]+\cdots+ x_1\dots x_{n-1}[x_n]$. Then $\partial r_1r_2=\partial r_1+r_1\partial r_2=\partial r_1+\partial r_2$. $\partial$ induces a homomorphism $\partial':R/R'\to \ZZ[G]\left<\cX\right>$ which is easily seen to be $G$-equivariant. One can also see easily that $\partial_2$ is given by the composition $\ZZ[G]\left<\cY\right>\to R/R'\to \ZZ[G]\left<\cX\right>$.

We shall give a ``geometric'' proof that $\partial'$ is an isomophism. $\ker\partial_1$ can be interpreted as the group of $1$-cycles of the universal covering space $\widetilde{BG}$ of $BG$ with the induced equivariant cell structure. Since $BF$ is the 1-skeleton of $BG$, the 1-skeleton of $\widetilde{BG}$ is a classifying space for $R$. Thus $\ker\partial_1\cong H_1R\cong R/R'$. That this isomorphism is given by $\partial'$ is straightforward.
\end{proof}

This concludes the proof of \ref{Peiffer theorem}.
}

{
\begin{cor}\label{cor: 1.8}
The following sequences are exact.
\[
(a)\qquad 0\to H_3G\to \overline P(G)\otimes_G \ZZ\to \ZZ\left<\cY\right>\to R/R'\otimes_G\ZZ\to 0
\]
\[
(b)\qquad 0\to H_2G\to R/R'\otimes_G \ZZ\to \ZZ\left<\cX\right>\to H_1G\to 0\qquad\ \ 
\]
where $\ZZ\left<\cX\right>$ denotes the free abelian group generated by $\cX$.
\end{cor}

\begin{eg}\label{eg: 1.9}
$G=\ZZ_2=\left< x|x^2\right>$. The relevant exact sequence is:
\[
\xymatrix{
0 & \ZZ \ar[l]& 
	\ZZ[\ZZ_2] \ar[l]_\vare&
	\ZZ[\ZZ_2] \ar[l]_{x-1}&
	\ZZ[\ZZ_2] \ar[l]_{x+1}&
	\overline P(\ZZ_2) \ar[l]& 0\ar[l]
	} 
\]
Thus $\overline P(\ZZ_2)$ is the cyclic subgroup of $\ZZ[\ZZ_2]$ generated by $x-1$. $\ZZ_2$ acts by negation. The generator of $\overline P(\ZZ_2)$ can be written as an element of $Q(\ZZ_2)$ by $(x,x^2)(1,x^2)^{-1}$ the corresponding graph for which is given in (\ref{eg: 1.2}). This is the only nontrivial element of $\overline P(\ZZ_2)\otimes_{\ZZ_2}\ZZ$ so it represents the generator of $H_3\ZZ_2$.
\end{eg}
}

{
\section{$K_3(\ZZ[\pi])$ and $\overline P(St(\ZZ[\pi])$}

The inclusion $K_3(\ZZ[\pi])\cong H_3St(\ZZ[\pi])\subset H_0(St(\ZZ[\pi]);\overline P(St(\ZZ[\pi])))$ has a natural section which is given by modding out ``second order Steinberg relations.'' Thus every element of $\overline P(St(\ZZ[\pi]))$ represents a well determined element of $K_3(\ZZ[\pi])$. This can also be done for arbitrary rings but the procedure is more complicated because there are 25 percent more Steinberg relations and 80 percent more second order relations.

\begin{defn}\label{def: 2.1}
$St(\ZZ[\pi])=\left<\cX|\cY\right>$ where 
\[
	\cX=\{e_{ij}^u\,|\, i,j \text{ are distinct natural numbers and } u\in\pi\}
\]
$\cY\cup \cY^{-1}$ consists of the following elements of $F$.
\begin{enumerate}
\item $[e_{ij}^u,e_{k\ell}^v]$ where $i\neq\ell$, $j\neq k$, and $u\neq v$ if $i=k$ and $j=\ell$.
\item {\color{blue} $e_{ij}^ue_{jk}^ve_{ij}^{-u}e_{ik}^{-uv}e_{jk}^{-v}$} if $i\neq k$. $e_{ik}^{-uv}$ represents $\left(e_{ik}^{uv}\right)^{-1}$. {\color{blue} (The original relation $[e_{ij}^u,e_{jk}^v]e_{ik}^{-uv}$ is changed to an equivalent ``good commutator relation'' \cite{IT14}.)}
\item  {\color{blue} $e_{jk}^v e_{ik}^{uv}e_{ij}^u e_{jk}^{-v}e_{ij}^{-u}$} if $i\neq k$.
\end{enumerate}
$\cY$ will represent any maximal reduced subset of $\cY\cup \cY^{-1}$.
\end{defn}

We shall not need an explicit choice of $\cY$ provided we make the following conventions. $\ZZ[St(\ZZ[\pi])]\left<\cY\right>$ will denote the $St(\ZZ[\pi])$-module generated by $\cY\cup \cY^{-1}$ modulo the relation $[y]+[y^{-1}]=0$. And similarly for $\ZZ\left<\cY\right>$.

Since $H_iSt(\ZZ[\pi])=0$ for $i=1,2$ (see \cite{M}) (\ref{cor: 1.8}) produces the exact sequence
\[
\xymatrix{
0 \ar[r]& 
	K_3(\ZZ[\pi] \ar[r]&
	\overline P(St(\ZZ[\pi]))\otimes_{St}\ZZ  \ar[r]^(.65){h}&
	\ZZ\left<\cY\right>  \ar[r]^\partial& 
	\ZZ\left<\cX\right>  \ar[r]& 0.
	}
\]
}

{
\begin{lem}\label{lem: 2.2}
The kernel of $\partial: \ZZ\left<\cY\right>\to \ZZ\left<\cX\right>$ is generated by the images under $h$ of the following elements of $\overline P(St(\ZZ[\pi]))$.
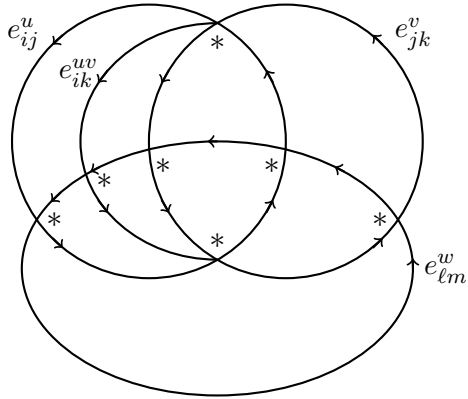
\begin{enumerate}
\item[a)] $m\neq i,j$, $\ell\neq j,k$:
{
\begin{center}
\begin{tikzpicture}[scale=1.3]
\begin{scope}
\clip rectangle (-2,-1.3) rectangle (0,1.5);
\draw[thick] (0,0) ellipse [x radius=1.4cm,y radius=1.21cm];
\end{scope}
\draw (0,1.2) node[below]{$\ast$};
\draw (0,-1.2) node[above]{$\ast$};
\draw (.75,-.25) node[left]{$\ast$};
\draw (-.75,-.25) node[right]{$\ast$};
\draw (1.86,-.8) node[left]{$\ast$};
\draw (-1.86,-.8) node[right]{$\ast$};
\draw (-1.35,-.4) node[right]{$\ast$};
		\draw[thick,<-] (-1.22,.6)--(-1.12,.73); 
		\draw[thick,->] (-1.22,-.6)--(-1.12,-.73);
		\draw[thick] (-1.1,.7) node[left]{\small$e_{ik}^{uv}$};
\begin{scope}[xshift=-.7cm]
	\draw[thick] (0,0) circle[radius=1.4cm];
		\draw[thick,<-] (-.985,.99)--(-.88,1.09); 
		\draw (-1,1.1) node[left]{\small$e_{ij}^u$};
		\draw[thick,->] (1.27,.6)--(1.2,.73);
		\draw[thick,->] (-.985,-.99)--(-.88,-1.09);
		\draw[thick,<-] (1.27,-.6)--(1.2,-.73);
\end{scope}
\begin{scope}[xshift=.7cm]
	\draw[thick] (0,0) circle[radius=1.4cm];
	\draw[thick,->] (.99,.99)--(.89,1.09); 
		\draw[thick,<-] (-1.27,.6)--(-1.195,.73);
		\draw (1,1.1) node[right]{\small$e_{jk}^v$};
	\draw[thick,<-] (.99,-.99)--(.89,-1.09);
		\draw[thick,->] (-1.27,-.6)--(-1.195,-.73);
\end{scope}
\begin{scope}[yshift=-1.3cm]
	\draw[thick] (0,0) ellipse[x radius=2cm, y radius=1.3cm];
	\draw[thick,->] (2,0)--(2,.1); 
	\draw (2,0) node[right]{\small$e_{\ell m}^w$};
	\draw[thick,->] (0,1.3)--(-.1,1.3) ;
	\draw[thick,<-] (1.2,1.04)--(1.3,.99);
	\draw[thick,->] (-1.2,1.04)--(-1.3,.99);
	\draw[thick,->] (-1.6,.78)--(-1.7,.68);
\end{scope}
\end{tikzpicture}
\end{center}
The exceptional case $\ell=i,m=k,w=uv$ is explained in \emph{(\ref{def: 2.3}(0))}.
}
\item[b)]\quad
{\begin{center}
\begin{tikzpicture}[scale=1.5] Figure 2.2 (b)
\begin{scope}
\clip rectangle (-2,-1.3) rectangle (0,1.3);
\draw[thick] (0,0) ellipse [x radius=1.4cm,y radius=1.21cm];
\end{scope}
\draw (0,1.2) node[below]{$\ast$};
\draw (0,-1.2) node[above]{$\ast$};
\draw (.75,-.2) node[left]{$\ast$};
\draw (-.75,-.2) node[right]{$\ast$};
\draw (1.86,-.8) node[left]{$\ast$};
\draw (-1.86,-.8) node[right]{$\ast$};
\draw (-1.35,-.4) node[right]{$\ast$};
\draw (-.64,.54) node[right]{$\ast$};
\draw (.65,.46) node[left]{$\ast$};
		\draw[thick,<-] (-1.22,.6)--(-1.12,.73);
				\draw[thick] (-1.15,.7) node[left]{\small$e_{ik}^{uv}$};
\begin{scope}[xshift=-.7cm]
	\draw[thick] (0,0) circle[radius=1.4cm];
		\draw[thick,<-] (-.985,.99)--(-.88,1.09);
		\draw (-1,1.1) node[left]{$e_{ij}^u$};
\end{scope}
\begin{scope}[xshift=.7cm]
	\draw[thick] (0,0) circle[radius=1.4cm];
	\draw[thick,->] (.99,.99)--(.89,1.09);
		\draw (1,1.1) node[right]{$e_{jk}^v$};
\end{scope}
\begin{scope}[yshift=-1.3cm]
	\draw[thick] (0,0) ellipse[x radius=2cm, y radius=1.3cm];
	\draw[thick,->] (2,0)--(2,.1); 
	\draw (2,0) node[right]{$ e_{k\ell}^w$};
\coordinate (A) at (-.7,1.22);
\coordinate (A1) at (-.6,1.5);
\coordinate (B1) at (0,2);
\coordinate (B) at (.57,1.87);
\coordinate (B2) at (.77,1.87);
\coordinate (C) at (1.85,.5);
\coordinate (C1) at (1.7,1.5);
\coordinate (D) at (-1.34,.96);
\coordinate (E) at (1.4,1.405);
\coordinate (E2) at (1.35,1.5);
\coordinate (F1) at (-1.08,1.5);
\coordinate (F2) at (-1,1.55);
\coordinate (F) at (-1,1.595);
\draw[thick,->] (F)--(F1);
\draw[thick,<-] (1.3,1.5)--(E);
\draw[thick] (A)..controls (A1) and (B1)..(B);
\draw[thick] (B)..controls (B2) and (C1)..(C);
\draw[thick](B)..controls (.2,2.3) and (-1.1,2)..(D);
\draw (E2) node[right]{\small$e_{j\ell}^{vw}$};
\draw (F2) node[above]{\tiny$e_{i\ell}^{uvw}$};
\end{scope}
\end{tikzpicture}
\end{center}
}
We are using the convention that a smooth sequence of edges without corners all have the same label and orientation.
\end{enumerate}
\end{lem}

\begin{proof}
Since $\partial[a,b]=0$, {\color{blue} $\partial(acba^{-1}b^{-1})=c$, and $\partial(aba^{-1}c^{-1}b^{-1})=-[c]$}, the kernel of $\partial: \ZZ\left<\cY\right>\to \ZZ\left<\cX\right>$ is generated be the relations \ref{def: 2.1}(1) and pairs of relations one from \ref{def: 2.1}(2) and one from \ref{def: 2.1}(3) with cancelling boundaries. If we apply $h$ to \ref{lem: 2.2}(a) we get $[[e_{\ell m}^w,e_{ik}^{uv}]]$ which is a generator of $\ker\partial$ of the first kind and if we apply $h$ to \ref{lem: 2.2}(b) we get {\color{blue} $[e_{ij}^u e_{j\ell}^{vw} e_{ij}^{-u} e_{i\ell}^{-uvw} e_{j\ell}^{-vw}]+[e_{k\ell}^w e_{i\ell}^{uvw} e_{ik}^{uv} e_{k\ell}^{-w}e_{ik}^{-uv}]\,+$} more generators of the first kind. This is a generator of the second kind.
\end{proof}

Let $H_0$ be the submodule of $\overline P(St(\ZZ[\pi]))$ generated by elements of the form \ref{lem: 2.2}(a), (b). If is clear from (\ref{lem: 2.2}) that the composition $K_3(\ZZ[\pi])\to \overline P\otimes_{St} \ZZ\to \overline P/H_0\otimes_{St}\ZZ$ is surjective. One can show that it is in fact an isomorphism.
}

{
\begin{defn}\label{def: 2.3}
Let $H$ be the submodule of $\overline P(St(\ZZ[\pi]))$ generated by the following \emph{second order Steinberg relations}.
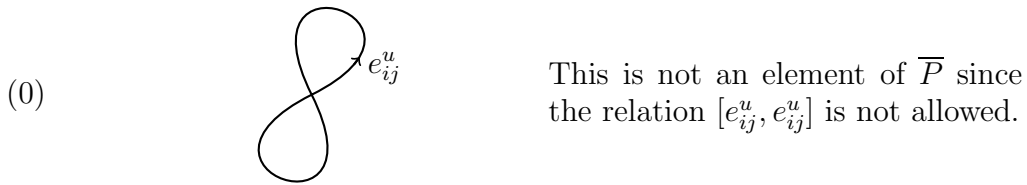
\begin{enumerate}
\item[(0)] \,
%
%
\begin{minipage}{0.4\textwidth}
\begin{center}
\begin{tikzpicture}
\coordinate (A1) at (.56,.4);
\coordinate (A3) at (.6,.35);
\coordinate (A2) at (.64,.5);
\begin{scope}
	\draw[ thick] (0,0)..controls (2,1) and (-1,2)..(0,0);
	\draw[ thick] (0,0)..controls (-2,-1) and (1,-2)..(0,0);
	\draw[ thick,->]  (A1)--(A2);
	\draw (A3) node[right]{$e_{ij}^u$};
\end{scope}
\end{tikzpicture}
\end{center}
\end{minipage}
%
%
%
%
\begin{minipage}{0.4\textwidth}
This is not an element of $\overline P$ since the relation $[e_{ij}^u,e_{ij}^u]$ is not allowed.
\end{minipage}%

This should on the other hand be thought of as a convention which says that, whenever a general formula includes a relation of the form $[x,x]$ in one of its special cases, this vertex should be deleted and the graph modified in the following way.
\begin{center}
\begin{tikzpicture}
\begin{scope}[xshift=-1cm]
\draw[thick,<-] (-1,-1)--(1,1);
\draw[thick,<-] (-1,1)--(1,-1);
\draw[thick] (1,1) node[right]{$x$};
\draw[thick] (1,-1) node[right]{$x$};
\end{scope}

\draw (2,0)node{$\Rightarrow$};
\begin{scope}[xshift=3.5cm]
\draw (1.5,.5) node[above]{$x$};
\draw[thick] (0,1) .. controls (.5,.7) and (1,.4) .. (1.5,.4);
\draw[thick] (3,1) .. controls (2.5,.7) and (2,.4) .. (1.5,.4);
\draw[thick] (1.6,.3)--(1.5,.4)--(1.6,.5);
\draw[thick] (0,-1) .. controls (.5,-.7) and (1,-.4) .. (1.5,-.4);
\draw[thick] (3,-1) .. controls (2.5,-.7) and (2,-.4) .. (1.5,-.4);
\draw (1.5,-.5) node[below]{$x$};
\draw[thick] (1.6,-.3)--(1.5,-.4)--(1.6,-.5);
\end{scope}
\end{tikzpicture}
\end{center}

\item $j\neq k$, etc.

\begin{center}
\begin{tikzpicture}
\begin{scope}[xshift=-.7cm]
	\draw[thick] (0,0) circle[radius=1.4cm];
		\draw[thick,<-] (-.985,.99)--(-.88,1.09); 
		\draw (-1,1.1) node[left]{\tiny$e_{ij}^u$};
\end{scope}
\begin{scope}[xshift=.7cm]
	\draw[thick] (0,0) circle[radius=1.4cm];
	\draw[thick,->] (.99,.99)--(.89,1.09); 
		\draw (1,1.1) node[right]{\tiny$e_{k\ell}^v$};
\end{scope}
\begin{scope}[yshift=-1.3cm]
	\draw[thick] (0,0) circle[radius=1.4cm];
	\draw[thick,->] (1.33,-.43)--(1.37,-.3); 
	\draw (1.37,-.36) node[right]{\tiny$e_{mn}^w$};
\end{scope}
\end{tikzpicture}
\end{center}
\item Same as \ref{lem: 2.2}(a).
\item $i\neq k,\ell$

\begin{center}
\begin{tikzpicture}[scale=1.5] 
%
\begin{scope}
\clip rectangle (-2,-1.3) rectangle (0,1.3);
\draw[thick] (0,0) ellipse [x radius=1.3cm,y radius=1.21cm];
\end{scope}
\begin{scope}[xshift=.87mm]
		\draw[thick,<-] (-1.22,.6)--(-1.12,.73); 
		\draw[thick] (-1.1,.7) node[left]{\tiny$e_{ik}^{uv}$};
\end{scope}
\begin{scope}[xshift=-3mm,yshift=-8mm]
		\draw[thick,<-] (-1.18,.6)--(-1.13,.73); 
		\draw[thick] (-1.1,.7) node[left]{\tiny$e_{i\ell}^{uw}$};
\end{scope}
\begin{scope}[xshift=-3.35mm,yshift=-6.9mm]
\begin{scope}[rotate=297]
\clip rectangle (-2,-1.3) rectangle (0,1.3);
\draw[thick] (0,0) ellipse [x radius=1.3cm,y radius=1.175cm];
\end{scope}
\end{scope}
\begin{scope}[xshift=-.7cm]
	\draw[thick] (0,0) circle[radius=1.4cm];
		\draw[thick,<-] (-.985,.99)--(-.88,1.09); 
		\draw (-1,1.1) node[left]{\small$e_{ij}^u$};
\end{scope}
\begin{scope}[xshift=.7cm]
	\draw[thick] (0,0) circle[radius=1.4cm];
	\draw[thick,->] (.99,.99)--(.89,1.09); 
		\draw (1,1.1) node[right]{\small$e_{jk}^v$};
\end{scope}
\begin{scope}[yshift=-1.35cm]
	\draw[thick] (0,0) circle[radius=1.4cm];
	\draw[thick,->] (1.33,-.43)--(1.37,-.3); 
	\draw (1.37,-.36) node[right]{\small$e_{j\ell}^w$};
\end{scope}
\end{tikzpicture}
\end{center}

\item $k\neq i,\ell$
\begin{center}
\begin{tikzpicture}[scale=1.5] 
%
\begin{scope}
\clip rectangle (-2,-1.3) rectangle (0,1.3);
\draw[thick] (0,0) ellipse [x radius=1.3cm,y radius=1.21cm];
\end{scope}
\begin{scope}[xshift=.87mm]
		\draw[thick,<-] (-1.22,.6)--(-1.12,.73); 
		\draw[thick] (-1.1,.7) node[left]{\tiny$e_{ik}^{uv}$};
\end{scope}
\begin{scope}[xshift=-3mm,yshift=-8mm]
		\draw[thick,->] (0.5,-1.164)--(0.6,-1.164); 
		\draw[thick] (.5,-1.2) node[below]{\tiny$e_{\ell k}^{wv}$};
\end{scope}
\begin{scope}[xshift=3.5mm,yshift=-6.9mm]
\begin{scope}[rotate=62]
\clip rectangle (-2,-1.3) rectangle (0,1.3);
\draw[thick] (0,0) ellipse [x radius=1.3cm,y radius=1.175cm];
\end{scope}
\end{scope}
\begin{scope}[xshift=-.7cm]
	\draw[thick] (0,0) circle[radius=1.4cm];
		\draw[thick,<-] (-.985,.99)--(-.88,1.09); 
		\draw (-1,1.1) node[left]{\small$e_{ij}^u$};
\end{scope}
\begin{scope}[xshift=.7cm]
	\draw[thick] (0,0) circle[radius=1.4cm];
	\draw[thick,->] (.99,.99)--(.89,1.09); 
		\draw (1,1.1) node[right]{\small$e_{jk}^v$};
\end{scope}
\begin{scope}[yshift=-1.35cm]
	\draw[thick] (0,0) circle[radius=1.4cm];
	\draw[thick,->] (1.33,-.43)--(1.37,-.3); 
	\draw (1.37,-.36) node[right]{\small$e_{\ell j}^w$};
\end{scope}
\end{tikzpicture}
\end{center}
\item Same as \ref{lem: 2.2}(b)
\end{enumerate}
\end{defn}
}

{The following theorem and its partial proof can be ignored for the purpose of the remainder of this paper.

\begin{thm}\label{thm: 2.4}
The composition
\[
	K_3(\ZZ[\pi])\to \overline P\otimes_{St}\ZZ \to \overline P/H\otimes_{St}\ZZ
\]
is an isomorphism.
\end{thm}

\noindent\emph{Proof}:
Since $H\supset H_0$ we know that the composition above is surjective. Thus it is sufficient to show that the image of $H\otimes_{St}\ZZ$ in $\overline P\otimes_{St}\ZZ$ is disjoint from the image of $K_3(\ZZ[\pi])$. This can be reworded as follows. We must show that every element of $H$ which goes to zero in $\ZZ\left<\cY\right>$ is already zero in $\overline P\otimes_{St}\ZZ$.

\begin{lem}\label{lem: 2.5}
An element of $\overline P$ of the form \emph{\ref{def: 2.3}(1)} is zero in $\overline P\otimes_{St}\ZZ$.
\end{lem}

\begin{proof}
Let $P$ be an element of the form \ref{def: 2.3}(2) = \ref{lem: 2.2}(a). Let $X$ be a generator of $St(\ZZ[\pi])$ which commutes with the four generators involved in $P$. We shall perform a deformation on $xP-P$.

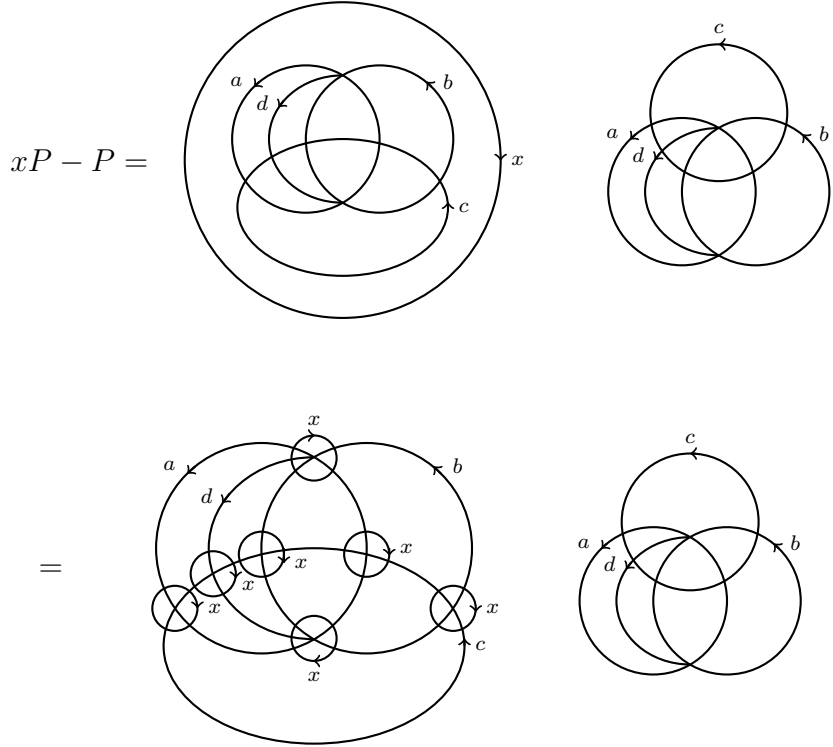
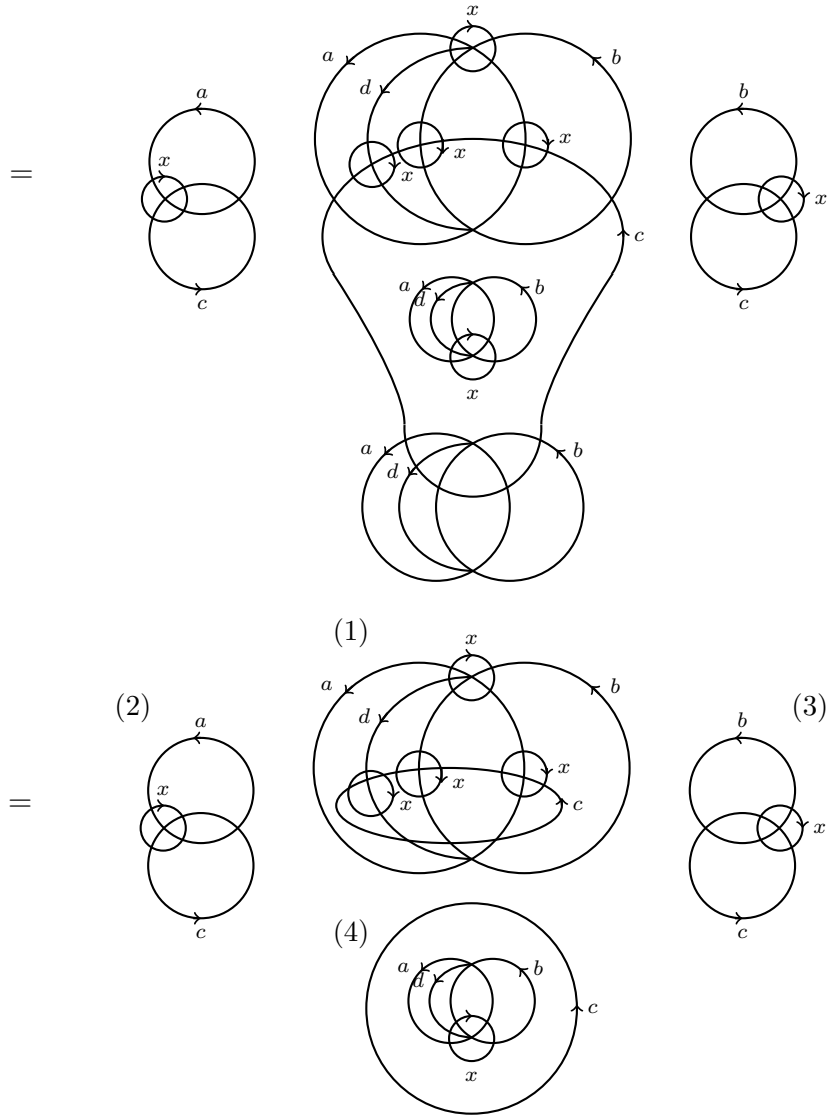
\begin{center}
\begin{tikzpicture}
\draw (-3.5,-.3) node{$xP-P=$};
\begin{scope}[scale=.7]
\draw[thick] (0,-.4) circle[radius=3cm];
\draw[thick,->] (3,-.4)--(3,-.41);
\draw (3,-.4) node[right]{\tiny$x$};
\begin{scope}
\clip rectangle (-2,-1.3) rectangle (0,1.3);
\draw[thick] (0,0) ellipse [x radius=1.4cm,y radius=1.21cm];
\end{scope}
		\draw[thick,<-] (-1.22,.6)--(-1.12,.73); 
		\draw[thick] (-1.2,.7) node[left]{\tiny$d$}; 
\begin{scope}[xshift=-.7cm]
	\draw[thick] (0,0) circle[radius=1.4cm];
		\draw[thick,<-] (-.985,.99)--(-.88,1.09); 
		\draw (-1,1.1) node[left]{\tiny$a$}; 
\end{scope}
\begin{scope}[xshift=.7cm]
	\draw[thick] (0,0) circle[radius=1.4cm];
	\draw[thick,->] (.99,.99)--(.89,1.09); 
		\draw (1,1.1) node[right]{\tiny$b$}; 
\end{scope}
\begin{scope}[yshift=-1.3cm]
	\draw[thick] (0,0) ellipse[x radius=2cm, y radius=1.3cm];
	\draw[thick,->] (2,0)--(2,.1); 
	\draw (2,0) node[right]{\tiny$c$}; 
\end{scope}
\end{scope} 
\begin{scope}[xshift=5cm,scale=.7,yshift=-1cm] 
\begin{scope}
\clip rectangle (-2,-1.3) rectangle (0,1.3);
\draw[thick] (0,0) ellipse [x radius=1.4cm,y radius=1.21cm];
\end{scope}
		\draw[thick,<-] (-1.22,.6)--(-1.12,.73); 
		\draw[thick] (-1.2,.7) node[left]{\tiny$d$}; 
\begin{scope}[xshift=-.7cm]
	\draw[thick] (0,0) circle[radius=1.4cm];
		\draw[thick,<-] (-.985,.99)--(-.88,1.09); 
		\draw (-1,1.1) node[left]{\tiny$a$}; 
\end{scope}
\begin{scope}[xshift=.7cm]
	\draw[thick] (0,0) circle[radius=1.4cm];
	\draw[thick,->] (.99,.99)--(.89,1.09); 
		\draw (1,1.1) node[right]{\tiny$b$}; 
\end{scope}
\begin{scope}[yshift=1.5cm]
	\draw[thick] (0,0) circle[radius=1.3cm]; 
	\draw[thick,<-] (0,1.3)--(.1,1.3); 
	\draw (0,1.3) node[above]{\tiny$c$}; 
\end{scope}
\end{scope} 
\end{tikzpicture}
\end{center}

\begin{center}
\begin{tikzpicture}
\draw (-3.5,-.3) node{$=$};
\begin{scope}
%
\draw[thick] (0,1.2) circle[radius=3mm];
\draw[ thick,->] (0,1.5)--(.01,1.5);
\draw (0,1.5) node[above]{\tiny$x$};
\draw[thick] (0,-1.2) circle[radius=3mm];
\draw[ thick,->] (0,-1.5)--(-.01,-1.5);
\draw (0,-1.5) node[below]{\tiny$x$};
\begin{scope}[yshift=-13mm]
\draw[thick] (1.85,.5) circle[radius=3mm];
\draw[ thick,<-] (2.15,.45)--(2.15,.47);
\draw (2.15,.5) node[right]{\tiny$x$};
\draw[thick] (-.7,1.22) circle[radius=3mm];
\draw[thick,<-] (-.4,1.1)--(-.4,1.3);
\draw (-.4,1.1) node[right]{\tiny$x$};
\draw[thick] (.7,1.22) circle[radius=3mm];
\draw[thick,<-] (1,1.2)--(1,1.3);
\draw (1,1.3) node[right]{\tiny$x$};
\draw[thick] (-1.85,.5) circle[radius=3mm];
\draw[thick,<-] (-1.55,.5)--(-1.55,.55);
\draw (-1.55,.5) node[right]{\tiny$x$};

\draw[thick] (-1.34,.96) circle[radius=3mm];
\draw[thick,<-] (-1.04,.9)--(-1.04,.92);
\draw (-1.1,.8) node[right]{\tiny$x$};

\end{scope}
\begin{scope}
\clip rectangle (-2,-1.3) rectangle (0,1.3);
\draw[thick] (0,0) ellipse [x radius=1.4cm,y radius=1.21cm];
\end{scope}
		\draw[thick,<-] (-1.22,.6)--(-1.12,.73); 
		\draw[thick] (-1.2,.7) node[left]{\tiny$d$}; 
\begin{scope}[xshift=-.7cm]
	\draw[thick] (0,0) circle[radius=1.4cm];
		\draw[thick,<-] (-.985,.99)--(-.88,1.09); 
		\draw (-1,1.1) node[left]{\tiny$a$}; 
\end{scope}
\begin{scope}[xshift=.7cm]
	\draw[thick] (0,0) circle[radius=1.4cm];
	\draw[thick,->] (.99,.99)--(.89,1.09); 
		\draw (1,1.1) node[right]{\tiny$b$}; 
\end{scope}
\begin{scope}[yshift=-1.3cm]
	\draw[thick] (0,0) ellipse[x radius=2cm, y radius=1.3cm];
	\draw[thick,->] (2,0)--(2,.1); 
	\draw (2,0) node[right]{\tiny$c$}; 
\end{scope}
\end{scope} 
\begin{scope}[xshift=5cm,scale=.7,yshift=-1cm] 
\begin{scope}
\clip rectangle (-2,-1.3) rectangle (0,1.3);
\draw[thick] (0,0) ellipse [x radius=1.4cm,y radius=1.21cm];
\end{scope}
		\draw[thick,<-] (-1.22,.6)--(-1.12,.73); 
		\draw[thick] (-1.2,.7) node[left]{\tiny$d$}; 
\begin{scope}[xshift=-.7cm]
	\draw[thick] (0,0) circle[radius=1.4cm];
		\draw[thick,<-] (-.985,.99)--(-.88,1.09); 
		\draw (-1,1.1) node[left]{\tiny$a$}; 
\end{scope}
\begin{scope}[xshift=.7cm]
	\draw[thick] (0,0) circle[radius=1.4cm];
	\draw[thick,->] (.99,.99)--(.89,1.09); 
		\draw (1,1.1) node[right]{\tiny$b$}; 
\end{scope}
\begin{scope}[yshift=1.5cm]
	\draw[thick] (0,0) circle[radius=1.3cm]; 
	\draw[thick,<-] (0,1.3)--(.1,1.3); 
	\draw (0,1.3) node[above]{\tiny$c$}; 
\end{scope}
\end{scope} 

\end{tikzpicture}
\end{center}

{
\begin{center}
\begin{tikzpicture}
\draw (-6,-.5) node{$=$};

\begin{scope}[xshift=-3.6cm,yshift=-.8cm]
\draw[thick] (0,.5) circle[radius=7mm];
	\draw[thick,->] (0,1.2)--(-.1,1.2);
	\draw (0,1.2) node[above]{\tiny$a$};
\draw[thick] (0,-.5) circle[radius=7mm];
	\draw[thick,<-] (0,-1.2)--(-.1,-1.2);
	\draw (0,-1.2) node[below]{\tiny$c$};
\end{scope}
\begin{scope}[xshift=-4.1cm,yshift=-.8cm]
\draw[thick] (0,0) circle[radius=3mm]; 
\draw[ thick,->] (0,.3)--(.01,.3);
\draw (0,.3) node[above]{\tiny$x$};
\end{scope}

\begin{scope}[xshift=3.6cm,yshift=-.8cm]
\draw[thick] (0,.5) circle[radius=7mm];
	\draw[thick,->] (0,1.2)--(-.1,1.2);
	\draw (0,1.2) node[above]{\tiny$b$};
\draw[thick] (0,-.5) circle[radius=7mm];
	\draw[thick,<-] (0,-1.2)--(-.1,-1.2);
	\draw (0,-1.2) node[below]{\tiny$c$};
\end{scope}
\begin{scope}[xshift=4.1cm,yshift=-.8cm]
\draw[thick] (0,0) circle[radius=3mm]; 
\draw[ thick,<-] (.3,.00)--(.3,.02);
\draw (.3,0) node[right]{\tiny$x$};
\end{scope}

\begin{scope}
%
\draw[thick] (0,1.2) circle[radius=3mm]; 
\draw[ thick,->] (0,1.5)--(.01,1.5);
\draw (0,1.5) node[above]{\tiny$x$};
\begin{scope}[yshift=-13mm]
\draw[thick] (-.7,1.22) circle[radius=3mm]; 
\draw[thick,<-] (-.4,1.1)--(-.4,1.3);
\draw (-.4,1.1) node[right]{\tiny$x$};
\draw[thick] (.7,1.22) circle[radius=3mm]; 
\draw[thick,<-] (1,1.2)--(1,1.3);
\draw (1,1.3) node[right]{\tiny$x$};

\draw[thick] (-1.34,.96) circle[radius=3mm]; 
\draw[thick,<-] (-1.04,.9)--(-1.04,.92);
\draw (-1.1,.8) node[right]{\tiny$x$};

\end{scope}
\begin{scope}
\clip rectangle (-2,-1.3) rectangle (0,1.3);
\draw[thick] (0,0) ellipse [x radius=1.4cm,y radius=1.21cm];
\end{scope}
		\draw[thick,<-] (-1.22,.6)--(-1.12,.73); 
		\draw[thick] (-1.2,.7) node[left]{\tiny$d$}; 
\begin{scope}[xshift=-.7cm]
	\draw[thick] (0,0) circle[radius=1.4cm];
		\draw[thick,<-] (-.985,.99)--(-.88,1.09); 
		\draw (-1,1.1) node[left]{\tiny$a$}; 
\end{scope}
\begin{scope}[xshift=.7cm]
	\draw[thick] (0,0) circle[radius=1.4cm];
	\draw[thick,->] (.99,.99)--(.89,1.09); 
		\draw (1,1.1) node[right]{\tiny$b$}; 
\end{scope}
\begin{scope}[yshift=-1.3cm]
	\draw[thick] (0,0) ellipse[x radius=2cm, y radius=1.3cm];
	\draw[thick,->] (2,0)--(2,.1); 
	\draw (2,0) node[right]{\tiny$c$}; 
\end{scope}
\end{scope} 
\begin{scope}[xshift=0cm,scale=.7,yshift=-7cm] 
\begin{scope}
\clip rectangle (-2,-1.3) rectangle (0,1.3);
\draw[thick] (0,0) ellipse [x radius=1.4cm,y radius=1.21cm];
\end{scope}
		\draw[thick,<-] (-1.22,.6)--(-1.12,.73); 
		\draw[thick] (-1.2,.7) node[left]{\tiny$d$}; 
\begin{scope}[xshift=-.7cm]
	\draw[thick] (0,0) circle[radius=1.4cm];
		\draw[thick,<-] (-.985,.99)--(-.88,1.09); 
		\draw (-1,1.1) node[left]{\tiny$a$}; 
\end{scope}
\begin{scope}[xshift=.7cm]
	\draw[thick] (0,0) circle[radius=1.4cm];
	\draw[thick,->] (.99,.99)--(.89,1.09); 
		\draw (1,1.1) node[right]{\tiny$b$}; 
\end{scope}
\begin{scope}[yshift=1.5cm]
	\draw[thick] (0,0) circle[radius=1.3cm]; 
	\draw[thick,<-] (0,1.3)--(.1,1.3); 
	\draw (0,1.3) node[above]{\tiny$c$}; 
\end{scope}
\end{scope} 

\begin{scope}[yshift=-2.8cm]
\draw[fill,white] (-3,-1) rectangle (3,1);
\draw[thick] (-.91,-1)..controls (-.9,-.7) and (-1.2, 0)..(-1.86,1.01);
\draw[thick] (.91,-1)..controls (.9,-.7) and (1.2, 0)..(1.86,1.01);
\end{scope}

\begin{scope}[xshift=0cm,yshift=-2.9cm]
\draw[thick] (0,0) circle[radius=3mm]; 
\draw[ thick,->] (0,.3)--(.01,.3);
\draw (0,-.3) node[below]{\tiny$x$};
\end{scope}


\begin{scope}[xshift=0cm,scale=.4,yshift=-6cm] 
\begin{scope}
\clip rectangle (-2,-1.3) rectangle (0,1.3);
\draw[thick] (0,0) ellipse [x radius=1.4cm,y radius=1.21cm];
\end{scope}
		\draw[thick,<-] (-1.22,.6)--(-1.12,.73); 
		\draw[thick] (-1.2,.7) node[left]{\tiny$d$}; 
\begin{scope}[xshift=-.7cm]
	\draw[thick] (0,0) circle[radius=1.4cm];
		\draw[thick,<-] (-.985,.99)--(-.88,1.09); 
		\draw (-1,1.1) node[left]{\tiny$a$}; 
\end{scope}
\begin{scope}[xshift=.7cm]
	\draw[thick] (0,0) circle[radius=1.4cm];
	\draw[thick,->] (.99,.99)--(.89,1.09); 
		\draw (1,1.1) node[right]{\tiny$b$}; 
\end{scope}

\end{scope} 

\end{tikzpicture}
\end{center}
}

{
\begin{center}
\begin{tikzpicture}
\draw (-6,-.5) node{$=$};

\draw (-1.6,1.8) node{\small$(1)$};
\draw (-4.5,.8) node{\small$(2)$};
\draw (4.5,.8) node{\small$(3)$};
\draw (-1.6,-2.2) node{\small$(4)$};
\begin{scope}[xshift=-3.6cm,yshift=-.8cm]
\draw[thick] (0,.5) circle[radius=7mm];
	\draw[thick,->] (0,1.2)--(-.1,1.2);
	\draw (0,1.2) node[above]{\tiny$a$};
\draw[thick] (0,-.5) circle[radius=7mm];
	\draw[thick,<-] (0,-1.2)--(-.1,-1.2);
	\draw (0,-1.2) node[below]{\tiny$c$};
\end{scope}
\begin{scope}[xshift=-4.1cm,yshift=-.8cm]
\draw[thick] (0,0) circle[radius=3mm]; 
\draw[ thick,->] (0,.3)--(.01,.3);
\draw (0,.3) node[above]{\tiny$x$};
\end{scope}

\begin{scope}[xshift=3.6cm,yshift=-.8cm]
\draw[thick] (0,.5) circle[radius=7mm];
	\draw[thick,->] (0,1.2)--(-.1,1.2);
	\draw (0,1.2) node[above]{\tiny$b$};
\draw[thick] (0,-.5) circle[radius=7mm];
	\draw[thick,<-] (0,-1.2)--(-.1,-1.2);
	\draw (0,-1.2) node[below]{\tiny$c$};
\end{scope}
\begin{scope}[xshift=4.1cm,yshift=-.8cm]
\draw[thick] (0,0) circle[radius=3mm]; 
\draw[ thick,<-] (.3,.00)--(.3,.02);
\draw (.3,0) node[right]{\tiny$x$};
\end{scope}

\begin{scope}
%
\draw[thick] (0,1.2) circle[radius=3mm]; 
\draw[ thick,->] (0,1.5)--(.01,1.5);
\draw (0,1.5) node[above]{\tiny$x$};
\begin{scope}[yshift=-13mm]
\draw[thick] (-.7,1.22) circle[radius=3mm]; 
\draw[thick,<-] (-.4,1.1)--(-.4,1.3);
\draw (-.4,1.1) node[right]{\tiny$x$};
\draw[thick] (.7,1.22) circle[radius=3mm]; 
\draw[thick,<-] (1,1.2)--(1,1.3);
\draw (1,1.3) node[right]{\tiny$x$};

\draw[thick] (-1.34,.96) circle[radius=3mm]; 
\draw[thick,<-] (-1.04,.9)--(-1.04,.92);
\draw (-1.1,.8) node[right]{\tiny$x$};

\end{scope}
\begin{scope}
\clip rectangle (-2,-1.3) rectangle (0,1.3);
\draw[thick] (0,0) ellipse [x radius=1.4cm,y radius=1.21cm];
\end{scope}
		\draw[thick,<-] (-1.22,.6)--(-1.12,.73); 
		\draw[thick] (-1.2,.7) node[left]{\tiny$d$}; 
\begin{scope}[xshift=-.7cm]
	\draw[thick] (0,0) circle[radius=1.4cm];
		\draw[thick,<-] (-.985,.99)--(-.88,1.09); 
		\draw (-1,1.1) node[left]{\tiny$a$}; 
\end{scope}
\begin{scope}[xshift=.7cm]
	\draw[thick] (0,0) circle[radius=1.4cm];
	\draw[thick,->] (.99,.99)--(.89,1.09); 
		\draw (1,1.1) node[right]{\tiny$b$}; 
\end{scope}
\begin{scope}[yshift=-.5cm,xshift=-.3cm]
	\draw[thick] (0,0) ellipse[x radius=1.5cm, y radius=.5cm];
	\draw[thick,->] (1.5,0)--(1.5,.1); 
	\draw (1.5,0) node[right]{\tiny$c$}; 
\end{scope}

\end{scope} 

\begin{scope}[yshift=-.7cm] 
\draw[thick] (0,-2.5) circle[radius=1.4cm];
\draw[thick,->] (1.4,-2.5)--(1.4,-2.45);
\draw (1.4,-2.5) node[right]{\tiny$c$};
\begin{scope}[xshift=0cm,yshift=-2.9cm]
\draw[thick] (0,0) circle[radius=3mm]; 
\draw[ thick,->] (0,.3)--(.01,.3);
\draw (0,-.3) node[below]{\tiny$x$};
\end{scope}
\begin{scope}[xshift=0cm,scale=.4,yshift=-6cm] 
\begin{scope}
\clip rectangle (-2,-1.3) rectangle (0,1.3);
\draw[thick] (0,0) ellipse [x radius=1.4cm,y radius=1.21cm];
\end{scope}
\draw[thick,<-] (-1.22,.6)--(-1.12,.73); 
\draw[thick] (-1.2,.7) node[left]{\tiny$d$}; 
\begin{scope}[xshift=-.7cm]
	\draw[thick] (0,0) circle[radius=1.4cm];
		\draw[thick,<-] (-.985,.99)--(-.88,1.09); 
		\draw (-1,1.1) node[left]{\tiny$a$}; 
\end{scope}
\begin{scope}[xshift=.7cm]
	\draw[thick] (0,0) circle[radius=1.4cm];
	\draw[thick,->] (.99,.99)--(.89,1.09); 
		\draw (1,1.1) node[right]{\tiny$b$}; 
\end{scope}
\end{scope} 
\end{scope}
\end{tikzpicture}
\end{center}
}
The circular edge in (4) can be removed because $c^{-1}g-g$ is zero in $\overline P\otimes_{St}\ZZ$. When (4) is then attached at the top of (1) it cancels the small $x$ circle at the top. After multiplying (2) by $b^{-1}$ it can be attached to (1) on the right to cancel the right hand little $x$ circle. {\color{blue}Similarly $a^{-1}d^{-1}(3)$ will cancel the middle little $x$ circle} and we are left with:

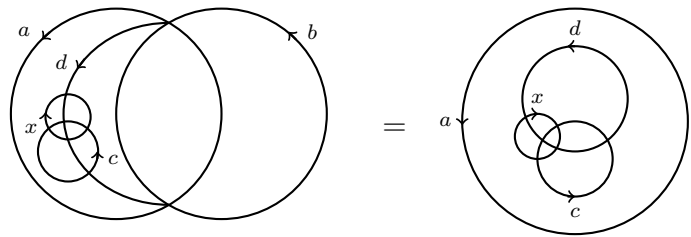
{
\begin{center}
\begin{tikzpicture}
\draw (3,-.2) node{$=$};

\begin{scope}[xshift=9cm,yshift=.5cm]
\draw[thick] (-3.6,-.6) circle[radius=1.5cm];
\draw[thick,->] (-5.1,-.6)--(-5.1,-.65);
\draw (-5.1,-.6) node[left]{\tiny$a$};
\begin{scope}[xshift=-3.6cm,yshift=-.8cm]
\draw[thick] (0,.5) circle[radius=7mm];
	\draw[thick,->] (0,1.2)--(-.1,1.2);
	\draw (0,1.2) node[above]{\tiny$d$};
\draw[thick] (0,-.3) circle[radius=5mm];
	\draw[thick,<-] (0,-.8)--(-.1,-.8);
	\draw (0,-.8) node[below]{\tiny$c$};
\end{scope}
\begin{scope}[xshift=-4.1cm,yshift=-.8cm]
\draw[thick] (0,0) circle[radius=3mm]; 
\draw[ thick,->] (0,.3)--(.01,.3);
\draw (0,.3) node[above]{\tiny$x$};
\end{scope}
\end{scope}

\begin{scope}
%
\begin{scope}[yshift=-10mm]
\draw[thick] (-1.34,.96) circle[radius=3mm]; 
\draw[thick,->] (-1.64,1)--(-1.64,1.02);
\draw (-1.6,.8) node[left]{\tiny$x$};
\end{scope}

\begin{scope}
\clip rectangle (-2,-1.3) rectangle (0,1.3);
\draw[thick] (0,0) ellipse [x radius=1.4cm,y radius=1.21cm];
\end{scope}
		\draw[thick,<-] (-1.22,.6)--(-1.12,.73); 
		\draw[thick] (-1.2,.7) node[left]{\tiny$d$}; 
\begin{scope}[xshift=-.7cm]
	\draw[thick] (0,0) circle[radius=1.4cm];
		\draw[thick,<-] (-.985,.99)--(-.88,1.09); 
		\draw (-1,1.1) node[left]{\tiny$a$}; 
\end{scope}
\begin{scope}[xshift=.7cm]
	\draw[thick] (0,0) circle[radius=1.4cm];
	\draw[thick,->] (.99,.99)--(.89,1.09); 
		\draw (1,1.1) node[right]{\tiny$b$}; 
\end{scope}
\begin{scope}[yshift=-.3cm,xshift=-1.2cm]
	\draw[thick] (-.14,-.2) circle[radius=.4cm];
	\draw[thick,->] (.25,-.3)--(.25,-.2); 
	\draw (.25,-.3) node[right]{\tiny$c$}; 
\end{scope}

\end{scope} 
%
\end{tikzpicture}
\end{center}
}
\noindent Eliminating the {\color{blue}circular edge labeled $a$} produces a typical element of the form \ref{def: 2.3}(1).
\end{proof}

The remainder of the proof {\color{blue}of Theorem \ref{thm: 2.4}} is similar. The details can be found in \cite{I} {\color{blue}See also \cite{IT14}}.
}

{
\section{The Generalized Grassmann Invariant}

We shall define a natural homomorphism
\[
	\chi: K_3(\ZZ[\pi])\to K_1(\ZZ[\pi];\ZZ_2[\pi])\cong H_0(\pi;\ZZ_2[\pi])
\]
where $\pi$ acts on $\ZZ_2[\pi]$ by conjugation.

\begin{defn}\label{def: 3.1}
The \emph{intersection pairing} will be the symmetric biadditive map $\ZZ[\pi]\times\ZZ[\pi]\to \ZZ[\pi]$ given by the formula
\[
	\left<\sum n_iu_i,\sum m_iu_i
	\right> = \sum n_im_i u_i.
\]
We shall also use the same notation for the induced pairing $\ZZ_2[\pi]\times\ZZ_2[\pi]\to \ZZ_2[\pi]$ which can be interpreted as set intersection if we think of $\ZZ_2[\pi]$ as the set of finite subsets of $\pi$.
\end{defn}

This pairing clearly satisfies the following condition. If $u,v\in\pi$ and $x,y\in\ZZ[\pi]$ or $\ZZ_2[\pi]$ then $\left< uxv, uyv\right>=u\left<x,y\right>v$.

Let $\chi_Q:Q(St(\ZZ[\pi]))\to H_0(\pi;\ZZ_2[\pi])$ be defined by the following formula
\[
	\chi_Q(f,[e_{ij}^u,e_{ik}^v])= \sum_p r_{pi}\left<us_{jp},vs_{kp}\right>
\]
where $(r_{pq})$ is the matrix representing the image of $f$ in $GL(\ZZ_2[pi])$ and $(s_{qp})$ is its inverse.

$\chi_Q(f,y)=0$ if $y$ is not of the above form.

Since the range is abelian $\chi_Q$ induces an additive homomorphism
\[
	\chi_c: \ZZ[St(\ZZ[\pi])]\left< \cY\right> \to H_0(\pi;\ZZ_2[\pi]).
\]
By restricting this map to $\ker\partial_2$ we get an additive homomorphism
\[
	\chi_{\overline P}: \overline P(St(\ZZ[\pi])) \to H_0(\pi;\ZZ_2[\pi]).
\]
We shall show that $\chi_{\overline P}$ is a homomorphism of $St(\ZZ[\pi])$ modules and that it is zero when restricted to $H_0$.
}

{
\begin{thm}\label{thm: 3.2}
$\chi_{\overline P}(H_0)=0$.
\end{thm}

\begin{proof}
The following computations can also be carried out for all the second order Steinberg relations to show $\chi_{\overline P}(H)=0$.
\begin{enumerate}
\item[a)] Suppose that $fP$ is an additive generator of $H_0$ where $f\in F$ and $P$ is a graph of the form \ref{lem: 2.2}(a). If $\ell\neq i$ then $\chi_{\overline P}(fP)=0$ because no relevant relations exist. If $\ell=i$ then there are three relevant relations except in the exceptional case when $m=k$ and $w=uv$. This exceptional case is taken care of by c) below.

At the three relevant relations in the case $\ell=i$ we have:
\begin{enumerate}
\item[1)]	$\chi_Q(fe_{im}^{-w}e_{ij}^{-u},[e_{im}^w, e_{ij}^u])= \sum_p r_{pi}\left<ws_{mp},us_{jp}\right>$.
\item[2)] $\chi_Q(fe_{im}^{-w} e_{jk}^{-v} e_{ij}^{-u},[ e_{ij}^u,e_{im}^w])= \sum_p r_{pi}\left<u(s_{jp}+vs_{kp}), ws_{mp}\right>$.
\item[3)] {\color{blue}
$\chi_Q(fe_{im}^{-w}  e_{ij}^{-u} e_{ik}^{-uv},[ e_{im}^w,e_{ik}^{uv}])= \sum_p r_{pi}\left< ws_{mp}, uvs_{kp}\right>$.
}
\end{enumerate}
where $(s_{\ast\ast})^{-1}=(r_{\ast\ast})$ is the image of $f$ in $GL(\ZZ_2[\pi])$. It can easily be seen that $(2)=(1)+(3)$ so the sum of all three is zero.
\item[b)] $\chi_{\overline P}(fP)=0$ if $P$ is a graph of the form \ref{lem: 2.2}(b). There are no relevant Steinberg relations.
\item[c)] Let us pretend for a moment that we allow $[x,x]$ as a Steinberg relation. Then we have
\[
	\chi_Q(f,[e_{ij}^u,e_{ij}^u])= \sum_p r_{pi}\left< us_{jp},us_{jp}\right>= \sum_p r_{pi} us_{jp}=\sum_p  us_{jp}r_{pi}=0
\]
since $i\neq j$. This means that the general computation of a) extends to the exceptional case.
\end{enumerate}
\end{proof}
}

{
\begin{thm}\label{thm: 3.3}
$\chi_{\overline P}:\overline P\to H_0(\pi;\ZZ_2\pi])$ is a homomorphism of $St(\ZZ[\pi])$-modules where the action on $H_0(\pi;\ZZ_2[\pi])$ is trivial.
\end{thm}

\begin{proof}
Let $x=e_{h\ell}^w$ be an arbitrary generator of $St(\ZZ[\pi])$. We shall show that $\chi_c(x-1)$ is an additive coboundary, i.e. it factors through $\ZZ[St]\left<\cX\right>$, and thus $\chi_{\overline P}(x-1)=0$.
\[
	\chi_c(x-1)(f[[e_{ij}^u,e_{ik}^v]])= \chi_c(f[[e_{ij}^u,e_{ik}^v]])+\chi_c(e_{h\ell}^w f[[e_{ij}^u,e_{ik}^v]])
\]
\[
	= \sum_p r_{pi} \left< us_{jp}, vs_{kp}\right> +
	\sum_p r_{pi}' \left< us_{jp}', vs_{kp}'\right>
\]
where 
\begin{eqnarray*}
r_{pi}' &=& r_{pi}\quad \text{if } p\neq h\\
r_{hi}' &=& r_{hi}+wr_{\ell i}\\
s_{qp}' &=& s_{qp} \quad \text{if } p\neq \ell\\
s_{q\ell}' &=& s_{q\ell} +s_{qh}w.
\end{eqnarray*}
Thus the terms in the above sum cancel except when $p=h$ or $\ell$ where we get

$(p=h): wr_{\ell i}\left< us_{jh} , vs_{kh} \right>$

$(p=\ell): r_{\ell i}\left< us_{j\ell}+us_{jh}w, , vs_{k\ell}+ vs_{kh} \right>+
r_{\ell i}\left< us_{j\ell} , vs_{k\ell} \right>$

\noindent since $r_{\ell i}\left< us_{jh}w, ws_{kh}w\right> = r_{\ell i}\left< us_{jh}, ws_{kh}\right>w=wr_{\ell i}\left< us_{jh}, ws_{kh}\right>$, where we have left only two of the four cross terms of $p=\ell$ which are
\begin{formula}\label{eq: 3.4}
$
\chi_c(e_{h\ell}^w-1)(f[[e_{ij}^u,e_{ik}^v]])= 
r_{\ell i} \left< us_{i\ell}, vs_{kh}w\right> 
+ r_{\ell i} \left< us_{jh}w,vs_{k\ell}\right>
$.
\end{formula}

Let $\psi:\ZZ[St(\ZZ[\pi])]\left<\cX\right>\to H_0(\pi;\ZZ_2[\pi])$ be defined by the following formula.
\[
	\psi(f[e_{ij}^u])=r_{\ell i} \left< us_{j\ell}, s_{ih}w+us_{jh}w\right>
\]
if $u\in\pi$. From the convention $[x^{-1}]=-x^{-1}[x]$ we get the formula
\[
	\psi(f[e_{ij}^{-u}]) =r_{\ell i} \left< us_{j\ell}, s_{ih} w\right>.
\]

We shall now show that $\psi \partial_2=\chi_c(x-1)$. Note that $\partial_2$ is equivariant but the other homomorphisms are not.
\begin{enumerate}
\item[a)] $\psi\partial_2 f[[e_{ij}^u,e_{ik}^v]]$ is a sum of four terms.
\begin{enumerate}
\item[1)] $\psi(f[e_{ij}^u])= r_{\ell i}\left< us_{j\ell}, s_{ih}w+ us_{jh}w\right>$
\item[2)] $\psi(fe_{ij}^u[e_{ik}^v])= r_{\ell i} \left<  vs_{k\ell}, s_{ih}w+ us_{jh}w + vs_{kh}w \right>$
\item[3)] $\psi(fe_{ij}^ue_{ik}^v[e_{ij}^{-u}])= r_{\ell i} \left<  us_{j\ell}, s_{ih}w+ vs_{kh}w+us_{jh}w\right>$
\item[4)] $\psi(fe_{ik}^v[e_{ik}^{-v}])= r_{\ell i} \left<  vs_{k\ell}, s_{ih}w+ vs_{kh}w \right>$
\end{enumerate}
All the terms cancel except two which are the same as the terms given in \eqref{eq: 3.4}
\item[b)] $\psi\partial_2f[[e_{ij}^u,e_{km}^v]]=0$ if $i\neq k,m$ and $j\neq k$. This follows from the general formula $\psi(fe_{km}^v[e_{ij}^u] )=\psi(f[e_{ij}^u])$ if $i\neq k,m$ and $j\neq k$.
\item[c)] {\color{blue} $\psi\partial_2f[e_{ij}^u e_{jk}^v e_{ij}^{-u} e_{ik}^{-uv} e_{jk}^{-v} ]$} is the sum of the following five terms.
\begin{enumerate}
\item[1)] $\psi(f[e_{ij}^u])= r_{\ell i}\left< us_{j\ell}, s_{ih}w+ us_{jh}w\right>$
\item[2)] $\psi(fe_{ij}^u[e_{jk}^v])= (r_{\ell j}+r_{\ell i}u)\left< vs_{k\ell}, s_{jh}w+ vs_{kh}w\right>$
\item[3)] $\psi(fe_{ij}^ue_{jk}^v[e_{ij}^{-u}])=\psi(fe_{jk}^ve_{ik}^{uv}[e_{ij}^{u}])= r_{\ell i}\left< us_{j\ell}+uvs_{k\ell}, s_{ih}w+ us_{jh}w\right>$
\item[4)] {
\color{blue}$\psi(fe_{ij}^ue_{jk}^ve_{ij}^{-u}[e_{ik}^{-uv}])=\psi(fe_{jk}^v[e_{ik}^{uv}])= r_{\ell i}\left< uvs_{k\ell}, s_{ih}w+ uvs_{kh}w\right>$
}
\item[5)] {
\color{blue}$\psi(fe_{jk}^v[e_{jk}^{-v}])=\psi(f[e_{jk}^{v}])= r_{\ell j}\left< vs_{k\ell}, s_{jh}w+ vs_{kh}w\right>$
}
\end{enumerate}
\end{enumerate}

One can easily see that all the terms cancel.
\end{proof}

\begin{cor}\label{cor: 3.5}
$\chi_{\overline P}$ is a 3-cocycle representing an element of $H^3(St(\ZZ[\pi]); H_0(\pi;\ZZ_2[\pi]))$. 
\end{cor}
 This can be interpreted as a Postikov invariant for Waldhausen's space $GL(\ZZ[\Omega B\pi])$ \cite{W}. {\color{blue} (See also \cite{A-infty}.)}
}


{
\section{Naturality of $\chi$}
In this section we shall prove that $\chi$ is a natural transformation of functors on the category of groups and homomorphism.

\begin{lem}\label{lem: 4.1}
$\chi$ is natural for injective homomorphisms.
\end{lem}

\begin{proof}
The intersection pairing is natural for injective homomorphisms.
\end{proof}

Every homomorphism $A\to B$ is the composition of a monomorphism $A\to A\times B$ and a projection $A\times B\to B$. Thus we will restrict our attention to the latter. As in the proof of (\ref{thm: 3.3}) we shall show that the difference between the maps
\[
	\ZZ[St(\ZZ[A\times B])]\left<\cY\right> \xrightarrow{\chi_c} H_0(A\times B;\ZZ_2[A\times B])\to H_0(B;\ZZ_2[B])
\]
and
\[
	\ZZ[St(\ZZ[A\times B])]\left<\cY\right> \to\ZZ[St(\ZZ[ B])]\left<\cY\right> \to  \xrightarrow{\chi_c'} H_0(B;\ZZ_2[B])
\]
is an integral coboundary and thus zero on $\overline P(St(\ZZ[A\times B]))$.

This difference map will be denoted $\Delta$.

If $x\in \ZZ_2[A\times B]$ then the image of $x$ in $\ZZ_2[\pi]$ will be denoted by $\overline x$. If we define a symmetric biadditive pairing
\[
	d: \ZZ_2[A\times B]\times \ZZ_2[A\times B]\to \ZZ_2[ B]
\]
by $d(x,y)=\left<\overline x,\overline y\right>+ \overline{\left<x,y\right>}$ then we have
\begin{formula}\label{eq: 4.2}
$\Delta(f[[e_{ij}^u,e_{ik}^v]])=\sum_p \overline r_{pi} d(us_{jp},vs_{kp})$.
\end{formula}
}

{
We shall consider $\ZZ_2[A]$ as the set of finite subsets of $A$. If $x\in\ZZ_2[A]$, $|x|$ will represent the number of elements of $x$. Let $\eta(x)\in \ZZ_2$ be defined by
\[
	\eta(x)= \begin{cases} 0 & \text{if } |x| \equiv 0,1 (4)\\
    1 &\text{if } |x| \equiv 2,3 (4)
    \end{cases}
\]
Then we have the following formula
\begin{enumerate}
\item $\eta(x+y)=\eta(x)+\eta(y)+|x|\cdot|y|+|x\cap y|$ 

 where $x\cap y$ can also be written as $\left<x,y\right>$. If $x\in \ZZ_2[A\times B]$ we can define $|x|=\sum_{b\in B} |x|_b b$ where $|x_b|=|xb^{-1}\cap A|$. Define $\eta(x)\in\ZZ_2[B]$ by $\eta(x)=\sum_{b\in B} \eta_b(x)b$ where $\eta_b(x)=\eta(xb^{-1}\cap A)$. Then formula (1) generalizes to
\item $\eta(x+y)=\eta(x)+\eta(y)+d(x,y)$.

 Now let $A$ be well ordered. If $x\in \ZZ_2[A]$ let $ O^+(x)=(x_1,x_2,\cdots,x_n)$ be the elements of $x$ in increasing order. Let $O^-(x)$ be the same elements in decreasing order. If $x,y$ are disjoint elements of $\ZZ_2[A]$, i.e. if $\left<x,y\right>=0$, then let $n_1(x,y)=0,1\in \ZZ_2$ depending on whether $O^-(x),O^+(x)$ and $O^+(x+y)$ differ by an even or odd permutation. The function $\eta$ originated in the following equation.
\item $n_1(y,x)=n_1(x,y)+\eta(x+y)$

 If $x,y$ are disjoint elements of $\ZZ_2[A\times B]$, $n_1$ can be generated by $n_2(x,y)=\sum_{b\in B} n_1(xb^{-1}\cap A,yb^{-1}\cap A)b$. Then formula (3) becomes
\item $n_2(y,x)=n_2(x,y)+\eta(x+y)$

If $x,y$ are arbitrary elements of $\ZZ_2[A\times B]$ then they determine three mutually disjoint elements
\begin{eqnarray*}
 x\backslash y &=& x+\left<x,y\right> = \left< x,x+y\right>\\
 y\backslash x &=& y+\left<x,y\right> =\left< y,x+y\right>\\
 x\cap y &=& \left< y,x+y\right>
\end{eqnarray*}
and we can define the following generalization of $n_2$.
\[
	n_3(x,y)=n_2(x\backslash y,y\backslash x) +n_2(x\cap y,x\backslash y)+n_2(x\cap y,y\backslash x).
\]
Since $x+y=(x\backslash y)+(y\backslash x)$ formula (4) gives
\item $n_3(y,x)=n_3(x,y)+\eta(x+y)$

Let $\overline d(x,y)=n_3(x,y)+\eta(x)$, then 2) and 5) give
\item $d(x,y)=\tilde d(x,y)+\tilde d(y,x)$.
\end{enumerate}

\begin{lem}\label{lem: 4.3}
\[
	\tilde d:\ZZ_2[A\times B]\times \ZZ_2[A\times B]\to \ZZ_2[ B]
\]
is biadditive.
\end{lem}

\begin{proof}
Let $\tilde d_i(x,y)= n_i(x,y)+\eta(x)$. Then $\tilde d_1(x,y)$ is the parity of the permutation taking $O^+(x)O^+(y)$ to $O^+(x+y)$. For each element of $y$ count the number of elements of $x$ what are larger and add these up for all the elements of $y$ and one gets $\tilde d_1(x,y)$. For this description one easily sees that $\tilde d_1(x,y)$ is biadditive where defined. By the analogous argument at each $b\in B$ one sees that $\tilde d_2$ is biadditive where defined. Using the biadditivity of $\tilde d_2$ and formula (2), the biadditivity of $
\tilde d_3$ is a straighforward computation. Note that because of (6) one need only show additivity in one variable. 
\end{proof}
}

{
In order to force equivariance of $\tilde d_3$ we make the following definition.

$\tilde d_4(x,y,z)=\sum_i \tilde d_3(x_iy,x_iz)$ if $x=\sum_i x_i$ where $x_i\in A\times B$. Then we have the equivariance condition
\begin{enumerate}
\item[7)] $\tilde d_4(xu,y,z)=\tilde d_4(x,uy,uz)$ if $u\in A\times B$.

Equation (6) now becomes
\item[8)] $\overline xd(y,z) =\tilde d_4(x,y,z)+\tilde d_4(x,z,y)$.

Note that $\tilde d_4$ is additive in each variable.
\end{enumerate}

Let $\psi:\ZZ[St(\ZZ[A\times B])]\left<\cX\right> \to H_0(B;\ZZ_2[B])$ be defined by
\[
	\psi(f[e_{ij}^u])= \sum_p \tilde d_4(r_{pi},us_{jp},s_{ip}+us_{jp})
\]
if $u\in A\times B$. We are using the notation $(r_{pq})=(s_{qp})^{-1}=$ image of $f$ in $GL(\ZZ_2[A\times B])$. The convention $[x^{-1}]=-x^{-1}[x]$ necessitates the equation
\[
	\psi(f[e_{ij}^{-u}])= \sum_p \tilde d_4(r_{pi},us_{jp},s_{ip})
\]
}

{
\begin{thm}\label{thm: 4.4}
$\psi \partial_2=\Delta$.
\end{thm}

\begin{proof}
We shall verify the equation on each additive generator of $\ZZ[St(\ZZ[A\times B])]\left<\cY\right>$.
\begin{enumerate}
\item[a)] $\Psi\partial_2(f[[e_{ij}^u,e_{ik}^v]])$ is the sum of four terms.
	\begin{enumerate}
	\item[1)] $\Psi(f[e_{ij}^u])= \sum_p \tilde d_4(r_{pi},us_{jp},s_{ip}+us_{jp})$
	\item[2)] $\Psi(fe_{ij}^u[e_{ik}^v])= \sum_p \tilde d_4(r_{pi},vs_{kp},s_{ip}+us_{jp}+vs_{kp})$
	\item[3)] $\Psi(fe_{ij}^ue_{ik}^v[e_{ij}^{-u}])= \sum_p \tilde d_4(r_{pi},us_{jp},s_{ip}+us_{jp}+vs_{kp})$
	\item[4)] $\Psi(fe_{ik}^v[e_{ik}^{-v}])= \sum_p \tilde d_4(r_{pi},vs_{kp},s_{ip}+vs_{kp}+vs_{kp})$
	\end{enumerate}
The sum is easily seen to be equal to
\[
	\sum_p \overline r_{pi} d(us_{jp},vs_{kp})=\Delta(f[[e_{ij}^u,e_{ik}^v]])
\]
\item[b)] $\Psi\partial_2(f[[e_{ij}^u,e_{k\ell}^v]])=0$ if $i\neq k,\ell$ and $j\neq k$ this is the result of the general formula
\[
	\Psi(fe_{ij}^u[e_{k\ell}^v])= \Psi(f[e_{k\ell}^v])
\] if $i\neq k,\ell$ and $j\neq k$. 
\item[c)] {\color{blue} $\Psi\partial_2f[e_{ij}^u e_{jk}^v e_{ij}^{-u} e_{ik}^{-uv} e_{jk}^{-v} ]$} is the sum of the following five terms.
\begin{enumerate}
\item[1)] $\Psi(f[e_{ij}^u])= \sum_p \tilde d_4( r_{p i}, us_{jp}, s_{ip}+ us_{jp})$
\item[2)] $\Psi(fe_{ij}^u[e_{jk}^v])= \sum_p \tilde d_4(r_{p j}+r_{p i}u, vs_{kp}, s_{jp}+ s_{kp})$
\item[3)] $\Psi(fe_{ij}^ue_{jk}^v[e_{ij}^{-u}])= \sum_p \tilde d_4( r_{p i},us_{jp}+uvs_{kp}, s_{ip}+ us_{jp})$
\item[4)] {
\color{blue}$\Psi(fe_{ij}^ue_{jk}^ve_{ij}^{-u}[e_{ik}^{-uv}])=\Psi(fe_{jk}^v[e_{ik}^{uv}])= \sum_p \tilde d_4(r_{p i}, uvs_{kp}, s_{ip}+ uvs_{kp})$
}
\item[5)] {
\color{blue}$\Psi(fe_{jk}^v[e_{jk}^{-v}])=\Psi(f[e_{jk}^{v}])= \sum_p \tilde d_4(r_{p j}, vs_{kp}, s_{jp}+ vs_{kp})$
}
\end{enumerate}
Using (7) and the triadditivity of $\tilde d_4$ this sum is easily seen to be zero.
\end{enumerate}
\end{proof}
}

{
\section{$\chi(\pi_3^s(B\pi\cup pt))=0$.}

We will assume the reader if familiar with \cite[\S 5]{M}.

Let $M(\pi)$ denote the subgroup of $GL(\ZZ[\pi])$ of monomial matrices with entries in $\pi$.

\begin{prop}\label{prop: 5.1}
The commutator subgroup of $M(\pi)$ is perfect and consists of all even monomials with abelianized determinant equal to 0 in $\pm \pi/\pi'$. Thus $M(\pi)'$ admits a universal central extension $T(\pi)\to M(\pi)'$ and there exists a unique homomorphism $T(\pi)\to St(\ZZ[\pi])$ over the inclusion $M(\pi)'\subset E(\ZZ[\pi])= GL(\ZZ[\pi])'$.
\end{prop}

By an argument analogous to the one in \cite{G} one sees that $H_3T(\pi)\cong \pi_3BM(\pi)^+$ which is isomorphic to $\pi_3^s(B\pi\cup pt)$ by the generalized Kahn-Priddy theorem. We shall show that the image of $H_3T(\pi)$ in $H_3St(\ZZ[\pi])=K_3(\ZZ[\pi])$ is contained in the kernel of $\chi$.

\begin{defn}\label{def: 5.2}
Let $W(\pm\pi)$ be the group generated by symbols $w_{ij}(u)$ where $i,j$ are distinct natural numbers and $u\in\pi$ modulo the reduced set of relations
\begin{enumerate}
\item $[w_{ij}(u),w_{k\ell}(v)]$ if $i,j,k,\ell$ are distinct and $i<k$.
\item $w_{ij}(u)w_{ik}(v)w_{ij}(u)^{-1}w_{jk}(u^{-1}v)$ \quad $j\neq k$
\item $w_{ij}(u)w_{ki}(v)w_{ij}(u)^{-1} w_{kj}(vu)$\quad $ j\neq k$
\item $w_{ij}(u)w_{jk}(v)w_{ij}(u)^{-1} w_{ik}(uv)^{-1}$\quad $i\neq k$
\item $w_{ij}(u)w_{kj}(v)w_{ij}(u)^{-1} w_{ki}(vu^{-1})^{-1}$\quad $i\neq k$
\end{enumerate}
\end{defn}
Let $\varphi:W(\pm\pi)\to M(\pm\pi)$ be the homomorphism given as follows.

$\varphi(w_{ij}(u))$ is the monomial matrix given by taking the identity matrix, multiplying the $i$-th column by $u$ and the $j$-th column by $-u^{-1}$ and transposing the two columns.

\begin{lem}\label{lem: 5.3}
The kernel of $\varphi$ is contained in the center of $W(\pm\pi)$, i.e. $W(\pm\pi)$ is a central extension of $\im\varphi$.
\end{lem}

\begin{proof}
Let $x\in\ker\varphi$. We shall show that $x$ commutes with $w_{ki}(u)$. Let $j$ be a natural number which does not appear as an index in the expansion of $x$ as a project of generators. Then by relation 5) we have $w_{ki}(u)=w_{ij}(1)w_{kj}(u)w_{ij}(1)^{-1}$. Thus it suffices to show that $xw_{kj}(u)x^{-1}=w_{kj}(u)$. However it is clear from the relations that $xw_{kj}(u)x^{-1}=w_{hj}(v)^{\pm 1}$. Since $\varphi (w_{kj}(u))=\varphi (w_{hj}(v)^{\pm 1})$ we must have $w_{kj}(u)=w_{hj}(v)^{\pm1}$.
\end{proof}

\begin{lem}\label{lem: 5.4}
The image of $\varphi$ consists of all monomials in $M(\pm\pi)$ with abelianized determinant $+1\in \pm\pi/\pi'$.
\end{lem}
}

{
\begin{thm}\label{thm: 5.5}
There exists a unique homomorphism $T(\pi)\to W(\pm\pi)$ covering the inclusion
\[
	M(\pi)'\subset M(\pm\pi).
\]
\end{thm}
\begin{proof}
This follows from the universality of $T(\pi)$ and the above two lemmas.
\end{proof}

\begin{lem}\label{lem: 5.6}
There exists a homomorphism $h:W(\pm\pi)\to St(\ZZ[\pi])$\footnote{It is proved in \cite{I} that this map is injective for finitely presented $\pi$ thus justifying the notation.} covering the inclusion $M(\pm\pi)\subset GL(\ZZ[\pi])$.
\end{lem}

\begin{proof}
Let $h(w_{ij}(u))=e_{ij}^u e_{ji}^{u^{-1}} e_{ij}^u$. For a proof that this is a homomorphism see \cite[p.72]{M}, or see \S 6.
\end{proof}

\begin{thm}\label{thm: 5.7}
The image of $H_3T(\pi)$ in $H_3St(\ZZ[\pi])=K_3(\ZZ[\pi])$ is contained in the image of $h_\ast:H_3W(\pm\pi)\to H_3St(\ZZ[\pi])$.
\end{thm}

\begin{proof}
By universality the map $T(\pi)\to St(\ZZ[\pi])$ is equal to the composition $T(\pi)\to W(\pm\pi)\xrightarrow h St(\ZZ[\pi])$.
\end{proof}

We shall consider $\chi$ as a cohomology class and show that $h^\ast(\chi)=0$. To compute $h^\ast(\chi)$ take any equivariant chain map
\[
\xymatrix{
0&\ZZ\ar[d]^1\ar[l]& \ZZ[W]\ar[d]^{h_0}\ar[l] &
\ZZ[W]\left<\cX_W\right>\ar[d]^{h_1}\ar[l] &
\ZZ[W]\left<\cY_W\right>\ar[d]^{h_2}\ar[l] &
\overline P(W)\ar[d]^{h_3}\ar[l] & 0\ar[l]\\
0&\ZZ\ar[l]& \ZZ[St]\ar[l] &
\ZZ[St]\left<\cX_{St}\right>\ar[l] &
\ZZ[St]\left<\cY_{St}\right>\ar[l] &
\overline P(St)\ar[l] & 0\ar[l]
	}
\]
The equivariant 3-cocycle $\chi_{\overline P(St)}$ which represents $\chi$ was defined as the coboundary of an integral (i.e. nonequivariant) 2-cochain
\[
	\chi_c:\ZZ[St]\left<\cY_{St}\right>\to H_0(\pi;\ZZ_2[\pi]).
\]

We shall show that $\chi_ch_2$ is $W(\pm\pi)$-equivariant. This implies $\chi_{\overline P(St)}h_3=\chi_c\partial_3h_3=\chi_ch_2\partial_3$ is an equivariant coboundary and thus represents the trivial cohomology class.

\begin{lem}\label{lem: 5.8}
Let $y\in \cY_{St}$, $f\in St(\ZZ[\pi])$, and $w\in W(\pm\pi)$. Then
\[
	\chi_c(h(w)f[y])=\chi_c(f[y]).
\]
\end{lem}

\begin{proof}
These are both zero by definition unless $y$ is a Steinberg relation of the form $y=[e_{ij}^u,e_{ik}^v]$. In this case we have
\[
	\chi_c(f[y])=\sum_p r_{pi}\left< us_{jp}, vs_{kp}\right>
\]
\[
	\chi_c(h(w)f[y])=\sum_p r_{pi}'\left< us_{jp}', vs_{kp}'\right>
\]
where $(r_{pq})=(s_{qp})^{-1}=$ image of $f$ in $GL(\ZZ_2[\pi])$ and $(r'_{pq})=(s'_{qp})^{-1}=$ image of $h(w)f$ in $GL(\ZZ_2[\pi])$. Every monomial matrix can be written uniquely as the product of a permutation matrix and a diagonal matrix. Thus $\varphi(w)=PD$. Let the entries of the diagonal matrix $D$ be written $d_p$ and let $P$ be the permutation matrix gotten by permuting the rows of the identity matrix by $\sigma^{-1}$. Then
\[
	(r'_{pq})=PD(r_{pq})\Rightarrow r_{pq}'=d_{\sigma(p)} r_{\sigma(p)q}
\]
\[
	(s'_{qp})=(s_{qp})D^{-1}P^{-1}\Rightarrow s_{qp}'=s_{q\sigma(p)} d_{\sigma(p)}^{-1} 
\]
Thus we have 
\[
\chi_c(h(w)f[y])= \sum_p d_{\sigma(p)} r_{\sigma(p)i}\left<us_{j\sigma(p)}d_{\sigma(p)}^{-1}, 
vs_{k\sigma(p)}d_{\sigma(p)}^{-1}\right>
\]
\[
	=\sum_p r_{\sigma(p) i}\left< us_{j\sigma(p)},vs_{k\sigma(p)}\right>
\]
because $d_{\sigma(p)}\in \pm\pi$. It is clear that the last expression is equal to $\chi_c(f[y])$.
\end{proof}

\begin{thm}\label{thm: 5.9}
$h^\ast(\chi)=0$ and thus the kernel of $\chi$ contains the image of $H_3W(\pm\pi)$.
\end{thm}
}


{
\section{$\chi:K_3(\ZZ)\to \ZZ_2$ is surjective}

Lemma \ref{lem: 2.2} implies that the image of $\chi:K_3(\ZZ)\to \ZZ_2$ is the same as the image of $\chi_{\overline P(St)}:\overline P(St(\ZZ))\to \ZZ_2$. We shall show that $\chi_{\overline P(St)}h_3:\overline P(W(\pm\pi))\to \ZZ_2$ is surjective for an appropriate choice of the chain map $h_\ast$.
}

{
Let $h_0:\ZZ[W(\pm1)]\to \ZZ[St(\ZZ)]$ be the unique $h$-equivariant map which sends 1 to 1. Let $h_1:\ZZ[W(\pm1)]\left<\cX_W\right>\to \ZZ[St(\ZZ)]\left<\cX_{St}\right>$ be the $h$-equivariant map given by
\[
	h_1([w_{ij}(1)]=[1_{ij}]+e_{ij}^1[e_{ji}^{-1}]+ e_{ij}^1e_{ji}^{-1}[e_{ij}^1].
\]	
Then it is clear that $\partial h_1=h_0\partial$.

Let $h_1:\ZZ[W(\pm1)]\left<\cY_W\right>\to \ZZ[St(\ZZ)]\left<\cY_{St}\right>$ be defined on $\cY_W$ by the following equations.
\begin{enumerate}
\item If $y=[w_{ij}(1),w_{k\ell}(1)]$ where $i,j,k,\ell$ are distinct, let
\begin{eqnarray*}
		h_2([y])&=& (1+ab^{-1}+cd^{-1}+ab^{-1}cd^{-1}) [[a,c]]\\
		& +&(ab^{-1}+ab^{-1}cd^{-1}) [[c,b]]\\
		& +&(cd^{-1}+ab^{-1}cd^{-1}) [[d,a]]\\
		& +&ab^{-1}cd^{-1}[[c,b]]
\end{eqnarray*}
where $a=e_{ij}^1, b=e_{ji}^1, c=e_{k\ell}^1, d=e_{\ell k}^1$, This can better be understood by examining the following partial graph.
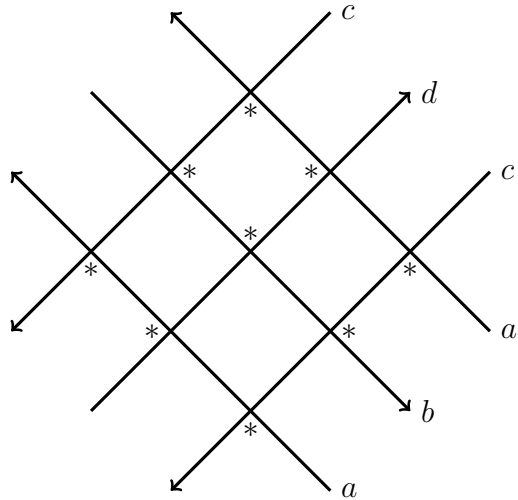
{ 
\begin{center}
\begin{tikzpicture}[rotate=45,scale=1.5]
\begin{scope}
\draw[very thick,<-] (-2,1)--(2,1);
\draw[very thick,->] (-2,0)--(2,0);
\draw[very thick,<-] (-2,-1)--(2,-1);
\draw (2,1) node[right]{$c$};
\draw (2,0) node[right]{$d$};
\draw (2,-1) node[right]{$c$};
\end{scope}

\begin{scope}
\draw[very thick,->] (-1,-2)--(-1,2);
\draw[very thick,<-] (0,-2)--(0,2);
\draw[very thick,->] (1,-2)--(1,2);
\draw (-1,-2) node[right]{$a$};
\draw (0,-2) node[right]{$b$};
\draw (1,-2) node[right]{$a$};
\end{scope}

\foreach \x in{(-1,1),(1,1),(-1,-1),(1,-1)}
\draw \x node[below]{$\ast$};
\draw (-1,0) node[left]{$\ast$};
\draw (1,0) node[left]{$\ast$};
\draw (0,1) node[right]{$\ast$};
\draw (0,-1) node[right]{$\ast$};
\draw (0,0) node[above]{$\ast$};
\end{tikzpicture}
\end{center}
}
From the partial graph it is obvious that $\partial h_2([y])=h_1\partial([y])$ but this can also be checked algebraically.
\item 
{

If $y=w_{ij}(1) w_{ik}(1) w_{ij}(1)^{-1} w_{jk}(1)$ where $i,j,k$ are distinct, let
{\color{blue}
\begin{eqnarray*}
h_2([y])&=& (1+e^{-1}f)(1+ab^{-1}) [[a,c]]\\
&+& (1+e^{-1}f)ab^{-1} [cebc^{-1}b^{-1}]\\
&+& (1+e^{-1}f)e^{-1} [ecae^{-1}a^{-1}]\\
&+& (1+e^{-1}f)ab^{-1}c [[b,e]]\\
&+& (e^{-1}+ab^{-1}cd^{-1}) [[a,f]]\\
&+&ab^{-1}cd^{-1} ([[b,d]]+[dad^{-1}f^{-1}a^{-1}]+ [fbf^{-1}d^{-1}b^{-1}])
\end{eqnarray*}
}
where $a=e_{ij}^1$, $ b=e_{ji}^1$, $ c=e_{ik}^1$, $d=e_{ki}^1$, $e=e_{jk}^1$, $f=e_{kj}^1$.

The corresponding partial graph is:
}
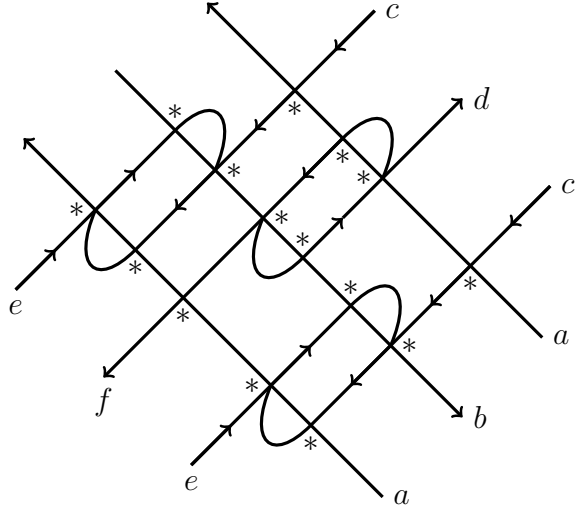
{
\begin{center}
\begin{tikzpicture}[rotate=45,scale=1.5]

\begin{scope}[xshift=1cm] 
\draw[very thick,<-] (-3,0.5)--(0,0.5); 
\draw (-3,.5) node[below]{$f$};
\draw[very thick,<-] (-.5,0.5)--(0,0.5);
\draw[very thick] (0,0.5)..controls (.5,.5) and (.5,.2)..(0,0);
\draw[very thick,->] (-1,0)--(-.5,0); 
\draw[very thick,->] (-.5,0)--(1,0);
\draw (1,0) node[right]{$d$};
\draw[very thick] (-1,0)..controls (-1.5,0) and (-1.5,.3)..(-1,0.5);
\draw (-2,0.5) node[below]{$\ast$};
\draw (-1,0.5) node[right]{$\ast$};
\draw (-1,0) node[above]{$\ast$};
\draw (0,0.5) node[below]{$\ast$};
\draw (0,0) node[left]{$\ast$};
\end{scope}

\begin{scope}[yshift=1.1cm] 
\draw[very thick] (-2,0.5)--(0,0.5); 
\draw[very thick,->] (-2,0.5)--(-1.5,0.5);
\draw[very thick,->] (-1,0.5)--(-.5,0.5);
\draw (-2,.5) node[below]{$e$};
\draw[very thick] (0,0.5)..controls (.5,.5) and (.5,.2)..(0,0);
\draw[very thick] (-1,0)--(2,0); 
\draw[very thick,<-] (-.5,0)--(0,0);
\draw[very thick,<-] (.5,0)--(1,0);
\draw[very thick,<-] (1.5,0)--(2,0);
\draw (2,0) node[right]{$c$};
\draw[very thick] (-1,0)..controls (-1.5,0) and (-1.5,.3)..(-1,0.5);
\draw (-1,0) node[below]{$\ast$};
\draw (-1,0.5) node[left]{$\ast$};
\draw (0,0.5) node[above]{$\ast$};
\draw (1,0) node[below]{$\ast$};
\draw (0,0) node[right]{$\ast$};
\end{scope}

\begin{scope}[yshift=-1.1cm] 
\draw[very thick] (-2,0.5)--(0,0.5); 
\draw[very thick,->] (-2,0.5)--(-1.5,0.5);
\draw[very thick,->] (-1,0.5)--(-.5,0.5);
\draw (-2,.5) node[below]{$e$};
\draw[very thick] (0,0.5)..controls (.5,.5) and (.5,.2)..(0,0);
\draw[very thick] (-1,0)--(2,0); 
\draw[very thick,<-] (-.5,0)--(0,0);
\draw[very thick,<-] (.5,0)--(1,0);
\draw[very thick,<-] (1.5,0)--(2,0);
\draw (2,0) node[right]{$c$};
\draw[very thick] (-1,0)..controls (-1.5,0) and (-1.5,.3)..(-1,0.5);
\draw (-1,0) node[below]{$\ast$};
\draw (-1,0.5) node[left]{$\ast$};
\draw (0,0.5) node[above]{$\ast$};
\draw (1,0) node[below]{$\ast$};
\draw (0,0) node[right]{$\ast$};
\end{scope}

{
\begin{scope}
\draw[very thick,->] (-1,-2)--(-1,2.5);
\draw[very thick,<-] (0,-2)--(0,2.35);
\draw[very thick,->] (1,-2)--(1,2.2);
\draw (-1,-2) node[right]{$a$};
\draw (0,-2) node[right]{$b$};
\draw (1,-2) node[right]{$a$};
\end{scope}
}

%
\end{tikzpicture}
\end{center}
}
\item 
{

If $y=w_{ij}(1) w_{ki}(1) w_{ij}(1)^{-1} w_{kj}(1)$ where $i,j,k$ are distinct, then
{\color{blue}
\begin{eqnarray*}
h_2([y])&=& (1+f^{-1}e)(f^{-1}+ab^{-1}f^{-1}) [[f,a]]\\
&+& (1+f^{-1}e)ab^{-1}f^{-1} ([bdfb^{-1}f^{-1}]+[[d,b]]+[afda^{-1}d^{-1}])\\
&+&  f^{-1}[aea^{-1}c^{-1}e^{-1}]\\
&+&  (f^{-1}e+af^{-1}b^{-1}c^{-1})[[c,a]]\\
&+&  f^{-1}ab^{-1}[[e,b]]\\
&+&  af^{-1}b^{-1}c^{-1}[bcb^{-1}e^{-1}c^{-1}]
\end{eqnarray*}
}
where $a-f$ are the same as in (2).

The corresponding partial graph is:

}
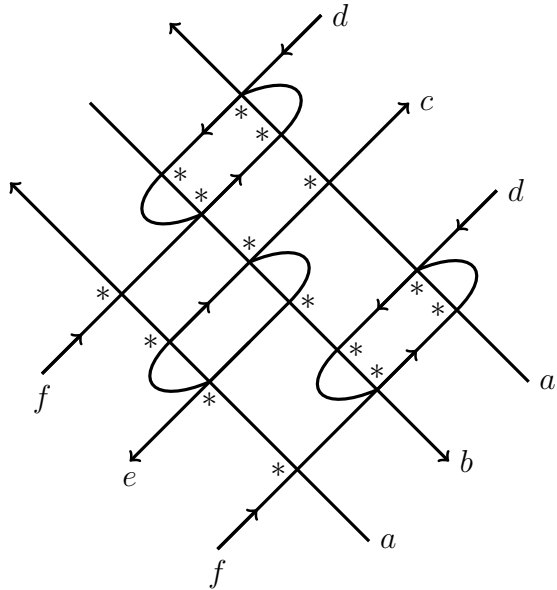
{
\begin{center}
\begin{tikzpicture}[rotate=45,scale=1.5]

\begin{scope}[yshift=1.1cm] 
\draw[very thick] (-2,0)--(-1,0); 
\draw[very thick,->] (-2,0)--(-1.5,0);
\draw (-1,0) node[left]{$\ast$};
\draw (-2,0) node[below]{$f$};
\end{scope}
\begin{scope}[yshift=-1.1cm] 
\draw[very thick] (-2,0)--(-1,0); 
\draw[very thick,->] (-2,0)--(-1.5,0);
\draw (-1,0) node[left]{$\ast$};
\draw (-2,0) node[below]{$f$};
\end{scope}

\begin{scope}[yshift=1.1cm, xshift=1cm] 
\draw[very thick] (-1,0.5)--(1,0.5); 
\draw[very thick,->] (1,0.5)--(.5,0.5); 
\draw[very thick,->] (0,0.5)--(-.5,0.5); 
\draw (1,.5) node[right]{$d$};
\draw[very thick] (0,0.5)..controls (.5,.3) and (.5,0)..(0,0);
\draw[very thick] (-2,0)--(0,0); 
\draw[very thick,->] (-1,0)--(-0.5,0);
\draw[very thick] (-1,0)..controls (-1.5,.2) and (-1.5,.5)..(-1,0.5);
\draw (-1, 0) node[above]{$\ast$};
\draw (-1,0.5) node[right]{$\ast$};
\draw (0,0) node[left]{$\ast$};
\draw (0,0.5) node[below]{$\ast$};
\end{scope}

\begin{scope}[yshift=-1.1cm, xshift=1cm] 
\draw[very thick] (-1,0.5)--(1,0.5); 
\draw[very thick,->] (1,0.5)--(.5,0.5); 
\draw[very thick,->] (0,0.5)--(-.5,0.5); 
\draw (1,.5) node[right]{$d$};
\draw[very thick] (0,0.5)..controls (.5,.3) and (.5,0)..(0,0);
\draw[very thick] (-2,0)--(0,0); 
\draw[very thick,->] (-1,0)--(-0.5,0);
\draw[very thick] (-1,0)..controls (-1.5,.2) and (-1.5,.5)..(-1,0.5);
\draw (-1, 0) node[above]{$\ast$};
\draw (-1,0.5) node[right]{$\ast$};
\draw (0,0) node[left]{$\ast$};
\draw (0,0.5) node[below]{$\ast$};
\end{scope}

\begin{scope}
\draw[very thick,->] (-1,0.5)--(2,0.5); 
\draw (2,.5) node[right]{$c$};
\draw[very thick,->] (-1,0.5)--(-.5,0.5);
\draw[very thick] (0,0.5)..controls (.5,.3) and (.5,0)..(0,0);
\draw[very thick,<-] (-2,0)--(0,0); 
\draw[very thick] (-1,0)--(-0.5,0);
\draw (-2,0) node[below]{$e$};
\draw[very thick] (-1,0)..controls (-1.5,.2) and (-1.5,.5)..(-1,0.5);
\draw (0,0) node[right]{$\ast$};
\draw (-1,0.5) node[left]{$\ast$};
\draw (-1,0) node[below]{$\ast$};
\draw (1,0.5) node[left]{$\ast$};
\draw (0,0.5) node[above]{$\ast$};
\end{scope}

\begin{scope}
\draw[very thick,->] (-1,-2)--(-1,2.5);
\draw[very thick,<-] (0,-2)--(0,2.5);
\draw[very thick,->] (1,-2)--(1,2.5);
\draw (-1,-2) node[right]{$a$};
\draw (0,-2) node[right]{$b$};
\draw (1,-2) node[right]{$a$};
\end{scope}

%
\end{tikzpicture}
\end{center}
}
\item  

{\color{blue}
If $y=w_{ij}(1) w_{jk}(1) w_{ij}(1)^{-1} w_{ik}(1)^{-1}$ where $i,j,k$ are distinct, then let
\begin{eqnarray*}
h_2([y])&=&  (1+cd^{-1})(1+ab^{-1}) [[a,c]]\\
&+& (1+cd^{-1})ab^{-1} ([cebc^{-1}b^{-1}]+[[c,e]])\\
&+& cd^{-1}[dad^{-1}f^{-1}a^{-1}]\\
&+& cd^{-1}a[[f,d]]\\
&+& cd^{-1}ab^{-1}([bdfb^{-1}f^{-1}]+ [[f,a]])
\end{eqnarray*}
}

where $a-f$ are the same as in (2). 

The corresponding partial diagram is:

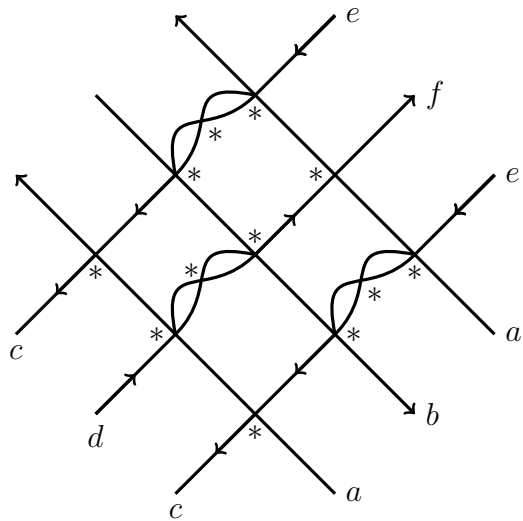
{
\begin{center}
\begin{tikzpicture}[rotate=45,scale=1.5]

\begin{scope}[yshift=1cm] 
\draw[very thick] (-2,0)--(-1,0); 
\draw[very thick,->] (-1,0)--(-1.5,0);
\draw (-1,0) node[below]{$\ast$};
\draw (-2,0) node[below]{$c$};
\end{scope}

\begin{scope}[yshift=-1cm] 
\draw[very thick] (-2,0)--(-1,0); 
\draw[very thick,->] (-1,0)--(-1.5,0);
\draw (-1,0) node[below]{$\ast$};
\draw (-2,0) node[below]{$c$};
\end{scope}

\begin{scope}[yshift=1cm, xshift=1cm] 
\draw (0,0) node[below]{$\ast$};
\draw (-1,0) node[right]{$\ast$};
\draw (-.5,0) node{$\ast$};
\draw (1,0) node[right]{$e$};
\draw[very thick,->] (-1,0)--(-1.5,0);
\draw[very thick,->] (1,0)--(.5,0);
\draw[very thick,<-] (.5,0)--(1,0);
\draw[very thick] (-2,0)--(-1,0);
\draw[very thick] (0,0)--(1,0);
\draw[very thick] (-1,0)..controls (-.4,0) and (-.4,.6)..(0,0);
\draw[very thick] (0,0)..controls (-.6,0) and (-.6,.6)..(-1,0);
\end{scope}
\begin{scope}[yshift=-1cm, xshift=1cm] 
\draw (0,0) node[below]{$\ast$};
\draw (-1,0) node[right]{$\ast$};
\draw (-.5,0) node{$\ast$};
\draw (1,0) node[right]{$e$};
\draw[very thick,->] (-1,0)--(-1.5,0);
\draw[very thick,->] (1,0)--(.5,0);
\draw[very thick,<-] (.5,0)--(1,0);
\draw[very thick] (-2,0)--(-1,0);
\draw[very thick] (0,0)--(1,0);
\draw[very thick] (-1,0)..controls (-.4,0) and (-.4,.6)..(0,0);
\draw[very thick] (0,0)..controls (-.6,0) and (-.6,.6)..(-1,0);
\end{scope}

\begin{scope}
\draw (0,0) node[above]{$\ast$};
\draw (-1,0) node[left]{$\ast$};
\draw (-2,0) node[below]{$d$};
\draw[very thick,->] (-2,0)--(-1.5,0);
\draw[very thick,->] (0,0)--(.5,0);
\draw (-.5,.3) node{$\ast$};
\draw[very thick,->] (1,0)--(2,0);
\draw (2,0) node[right]{$f$};
\draw (1,0) node[left]{$\ast$};
\draw[very thick] (-2,0)--(-1,0);
\draw[very thick] (0,0)--(1,0);
\draw[very thick] (-1,0)..controls (-.4,0) and (-.4,.6)..(0,0);
\draw[very thick] (0,0)..controls (-.6,0) and (-.6,.6)..(-1,0);
\end{scope}

\begin{scope}
\draw[very thick,->] (-1,-2)--(-1,2);
\draw[very thick,<-] (0,-2)--(0,2);
\draw[very thick,->] (1,-2)--(1,2);
\draw (-1,-2) node[right]{$a$};
\draw (0,-2) node[right]{$b$};
\draw (1,-2) node[right]{$a$};
\end{scope}

%
\end{tikzpicture}
\end{center}
}
\item 
{
If $y=w_{ij}(1) w_{kj}(1) w_{ij}(1)^{-1} w_{ki}(1)^{-1}$ where $i,j,k$ are distinct, then let
{\color{blue}
\begin{eqnarray*}
h_2([y])&=&  (1+dc^{-1}) [afda^{-1}d^{-1}]\\
&+& (1+dc^{-1})a [[d,f]]\\
&+& (1+dc^{-1})ab^{-1} ([fbf^{-1}d^{-1}b^{-1}]+[[a,f]])
\end{eqnarray*}
}
where $a-f$ are the same as in (2).

The corresponding partial graph is:
}
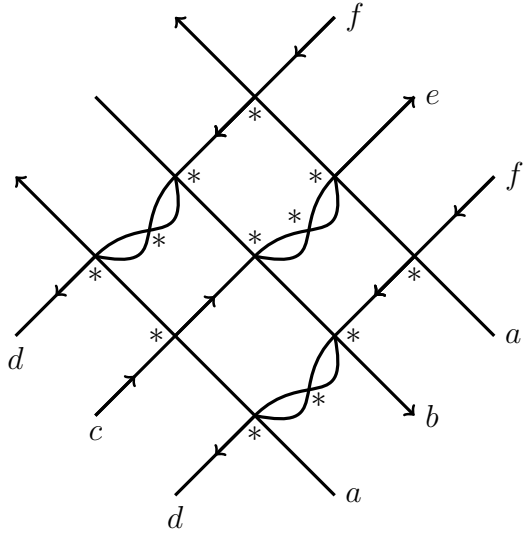
{
\begin{center}
\begin{tikzpicture}[rotate=45,scale=1.5]

\begin{scope}[yshift=1cm] 
\draw[very thick] (2,0)--(1,0); 
\draw[very thick,->] (1,0)--(.5,0);
\draw (1,0) node[below]{$\ast$};
\draw (-1,0) node[below]{$\ast$};
\draw (0,0) node[right]{$\ast$};
\draw (-.5,-.3) node{$\ast$};
\draw (2,0) node[right]{$f$};
\draw (-2,0) node[below]{$d$};
\end{scope}

\begin{scope}[yshift=-1cm] 
\draw[very thick] (2,0)--(1,0); 
\draw[very thick,->] (1,0)--(.5,0);
\draw (1,0) node[below]{$\ast$};
\draw (-1,0) node[below]{$\ast$};
\draw (0,0) node[right]{$\ast$};
\draw (-.5,-.3) node{$\ast$};
\draw (2,0) node[right]{$f$};
\draw (-2,0) node[below]{$d$};
\end{scope}

\draw[very thick] (-2,0)--(-1,0);
\draw[very thick,->] (-2,0)--(-1.5,0);
\draw (-2,0) node[below]{$c$};
\draw (0,0) node[above]{$\ast$};
\draw (1,0) node[left]{$\ast$};
\draw (-1,0) node[left]{$\ast$};
\draw (.5,0) node{$\ast$};

\begin{scope}[yshift=-1cm, xshift=0cm] 
\draw[very thick,->] (2,0)--(1.5,0);
\draw[very thick,->] (-1,0)--(-1.5,0);
\draw[very thick,->] (1,0)--(.5,0);
\draw[very thick,<-] (.5,0)--(1,0);
\draw[very thick] (-2,0)--(-1,0);
\draw[very thick] (0,0)--(1,0);
\draw[very thick] (-1,0)..controls (-.4,0) and (-.4,-.6)..(0,0);
\draw[very thick] (0,0)..controls (-.6,0) and (-.6,-.6)..(-1,0);
\end{scope}
\begin{scope}[yshift=1cm, xshift=0cm] 
\draw[very thick,->] (2,0)--(1.5,0);
\draw[very thick,->] (-1,0)--(-1.5,0);
\draw[very thick,->] (1,0)--(.5,0);
\draw[very thick,<-] (.5,0)--(1,0);
\draw[very thick] (-2,0)--(-1,0);
\draw[very thick] (0,0)--(1,0);
\draw[very thick] (-1,0)..controls (-.4,0) and (-.4,-.6)..(0,0);
\draw[very thick] (0,0)..controls (-.6,0) and (-.6,-.6)..(-1,0);
\end{scope}

\begin{scope}[yshift=0cm, xshift=1cm] 
\draw[very thick,->] (-2,0)--(-1.5,0);
\draw[very thick,->] (0,0)--(1,0);
\draw (1,0) node[right]{$e$};
\draw[very thick] (-2,0)--(-1,0);
\draw[very thick] (0,0)--(1,0);
\draw[very thick] (-1,0)..controls (-.4,0) and (-.4,-.6)..(0,0);
\draw[very thick] (0,0)..controls (-.6,0) and (-.6,-.6)..(-1,0);
\end{scope}

\begin{scope}
\draw[very thick,->] (-1,-2)--(-1,2);
\draw[very thick,<-] (0,-2)--(0,2);
\draw[very thick,->] (1,-2)--(1,2);
\draw (-1,-2) node[right]{$a$};
\draw (0,-2) node[right]{$b$};
\draw (1,-2) node[right]{$a$};
\end{scope}

%
\end{tikzpicture}
\end{center}
}
\end{enumerate}
We must now compute $\chi_ch_2$.
}

{
\begin{thm}\label{thm: 6.1}
$\chi_ch_2([y])=1$ if $y$ is a relation of type 3 and $\chi_ch_2([y])=0$ if $y$ is a relation of type 1,2,4,5.
\end{thm}

\begin{proof}
If $y$ is a type 1 relation for $W(\pm1)$ then $\chi_ch_2([y])=0$ because $h_2([y])$ contains no relevant relation. For the other elements of $\cY_W$, the only relevant Steinberg relations that occur are $[a,c]$, $[b,e]$, $[d,f]$ and their inverses.

If $g\in St(\ZZ)$ and $(r_{pq})=(s_{qp})^{-1}$ is the image of $g$ in $GL(\ZZ_2)$ then
\[
	\chi_c(g[[e_{ij}^1,e_{ik}^1]])=\sum_p s_{jp}s_{kp}r_{pi}.
\]
Using this formula one can easily calculate $\chi_c$ for all the relevant terms that occur in all the $h_2([y])$'s. The result is that $\chi_c$ is zero on all the terms except the last term of (3) where it is 1. In fact, if {\color{blue}$g=f^{-1}ab$ then
\[
	(r_{pq})= \begin{matrix} i\\j\\k\\ \ \end{matrix}\begin{matrix}\mat{0 & 1 & 0\\1 & 1 & 0\\ 1 & 1 &1}\\
	i\quad j\quad k\end{matrix}
	\qquad\qquad (s_{qp})= \begin{matrix} i\\j\\k\\ \ \end{matrix}\begin{matrix}\mat{1 & 1 & 0\\1 & 0 & 0\\ 0 & 1 &1}\\
	i\quad j\quad k\end{matrix}
\]
}
and $\chi_c(g[[e,b]])=\sum_p s_{jp}s_{kp}r_{pi}=1$.
\end{proof}
}

{
\begin{thm}\label{thm: 6.2}
There exists a graph $P\in \overline P(W(\pm1))$ with the property that $\chi_c h_2\partial_3(P)=1$.
\end{thm}

\begin{proof}
All we have to do is to find a graph which has an odd number of type 3 relations. Here it is.
{
\begin{center}
\begin{tikzpicture}[scale=.7]

\draw[very thick] (0,0) circle[radius=4.5cm];
\draw[very thick] (-3.758,-1.368)--(2.294,3.276);
\draw[very thick] (2.294,3.276)..controls (4.294,4.803) and (9,5)..(9,0);
\draw[very thick] (2.294,-3.276)..controls (4.294,-4.803) and (9,-5)..(9,0);
\draw[very thick] (3.758,1.368)--(-2.294,-3.276);
\draw[very thick] (3.758,1.368)..controls (4.758,2.135) and (6.5,2)..(6.5,0);
\draw[very thick] (3.758,-1.368)..controls (4.758,-2.135) and (6.5,-2)..(6.5,0);
\draw[very thick] (3.758,-1.368)--(-2.294,3.276);
\draw[very thick] (-2.294,3.276)..controls (-4.294,4.803) and (-9,5)..(-9,0);
\draw[very thick] (-2.294,-3.276)..controls (-4.294,-4.803) and (-9,-5)..(-9,0);

\draw[very thick] (-3.758,1.368)--(2.294,-3.276);
\draw[very thick] (-3.758,1.368)..controls (-4.758,2.135) and (-6.5,2)..(-6.5,0);
\draw[very thick] (-3.758,-1.368)..controls (-4.758,-2.135) and (-6.5,-2)..(-6.5,0);
\draw[very thick,<-] (0,4.5)--(.2,4.5); 
\draw[very thick,<-] (0,-4.5)--(.2,-4.5); 
\draw (0,4.5) node[above]{$w_{13}$};
\draw (0,-4.5) node[below]{$w_{13}$};
\draw[very thick,->] (9,0)--(9,.2); 
\draw[very thick,->] (-9,0)--(-9,.2); 
\draw (9,0) node[right]{$w_{12}$};
\draw (-9,0) node[left]{$w_{32}$};
\draw[very thick,->] (6.5,0)--(6.5,.2); 
\draw[very thick,->] (-6.5,0)--(-6.5,.2); 
\draw (6.5,0) node[right]{$w_{12}$};
\draw (-6.5,0) node[left]{$w_{32}$};
\draw[very thick,->] (4.5,0)--(4.5,.2); 
\draw[very thick,<-] (-4.5,0)--(-4.5,.2); 
\draw (4.5,0) node[right]{$w_{13}$};
\draw (-4.5,0) node[left]{$w_{13}$};
\draw[very thick,<-] (3.37,2.97)--(3.47,2.87); 
\draw[very thick,<-] (-3.57,2.77)--(-3.47,2.87);
\draw (3.47,2.97) node[right]{$w_{13}$};
\draw (-3.47,2.97) node[left]{$w_{13}$};

\draw[very thick,->] (3.37,-2.97)--(3.47,-2.87); 
\draw[very thick,<-] (-3.57,-2.77)--(-3.47,-2.87);
\draw (3.47,-2.97) node[right]{$w_{23}$};
\draw (-3.47,-2.97) node[left]{$w_{12}$};
\draw[very thick,->] (-.93,.8) -- (-1.03,.725); 
\draw (-1.4,.825) node[above]{$w_{31}$};
\draw (1.4,.8) node[above]{$w_{12}$};
\draw[very thick,->] (.93,.8) -- (1.03,.725);
\draw[very thick,<-] (-.93,-.8) -- (-1.03,-.725);
\draw[very thick,<-] (.93,-.8) -- (1.03,-.725);
\draw (-1.4,-.825) node[below]{$w_{12}$};
\draw (1.4,-.8) node[below]{$w_{31}$};
\draw[very thick,<-] (1.3,2.52)--(1.1,2.36); 
\draw[very thick,->] (-2.294,3.276)--(-1.3,2.5);
\draw (1,2.56) node[above]{$w_{32}$};
\draw (-1,2.56) node[above]{$w_{12}$};
\draw[very thick,<-] (1.3,-2.52)--(1.1,-2.36); 
\draw[very thick,<-] (-1.3,-2.52)--(-1.1,-2.36);
\draw (1,-2.56) node[below]{$w_{12}$};
\draw (-1,-2.56) node[below]{$w_{32}$};
\draw[very thick,->] (-3.358,1.06)--(-2.958,0.75); 
\draw (-2.958,0.9)node[above]{$w_{12}$};
\draw[very thick,<-] (-3.358,-1.06)--(-2.958,-0.75);
\draw (-2.8,-1)node[below]{$w_{32}$};
\draw[very thick,<-] (3.358,1.06)--(2.958,0.75); 
\draw (2.958,1)node[above]{$w_{32}$};
\draw[very thick,<-] (3.358,-1.06)--(2.958,-0.75);
\draw (2.8,-.95)node[below]{$w_{12}$};
\draw (2.7,3.6) node[below]{\tiny 2};
\draw (2.7,-3.6) node[above]{\tiny 2};
\draw (-2.6,3.6) node[right]{\tiny 4};
\draw (-2.7,-3.6) node[below]{\tiny 5};
\draw (0,1.5) node[above]{\small 3};
\draw (0,-1.5) node[above]{\tiny 5};

\draw (4.1,1.6) node[below]{\tiny 2};
\draw (4.1,-1.6) node[above]{\tiny 4};
\draw (-3.85,1.4) node[above]{\tiny 4};
\draw (-3.8,-1.4) node[below]{\small 3};
\draw (2,0)node[above]{\small 3};
\draw (-2,0)node[above]{\tiny 5};
\end{tikzpicture}
\end{center}
}
The numbers at the vertices indicate the base point direction and the type of relation that occurs at the vertex. Each type of relation occurs an odd number of times.
\end{proof}
}

{
\begin{cor}\label{cor: 6.3}
$\chi: K_3(\ZZ)\to \ZZ_2$ is surjective and thus $K_3(\ZZ)$ has at least 48 elements.
\end{cor}

\begin{proof}
It is well known (see \cite{Q}) that $\pi_3^s\cong H_3T(1)\to H_3St(\ZZ)\cong K_3(\ZZ)$ is injective, any by Theorem \ref{thm: 5.9} its image is contained in the kernel of $\chi$.
\end{proof}

The example given in \eqref{thm: 6.2} was originally discovered by the author by multiplying together the nontrivial elements of $K_1(\ZZ)$ and $K_2(\ZZ)$ using Loday's formula (see \cite{L}), by displaying the element as a graph, and by deforming the graph (adding and subtracting second order Steinberg relations) until it was the sum of 8 equal pieces.
}


{
\section{Application to Pseudoisotopy}

Let $M$ be a compact smooth manifold. A \emph{pseudoisotopy} of $M$ modulo $\partial M$ is a self-diffeomorphism of $M\times I$ which keeps $M\times \{0\}\cup (\partial M)\times I$ pointwise fixed. Let $\cP(M,\partial M)$ denote the space of all pseudoisotopies of $M$ modulo $\partial M$ with the $C^\infty$-topology.

\begin{thm}\label{thm: 7.1}
If $\dim M\ge 5$, $\pi_1M=\pi$ and $\pi_2M=0$, then there is an exact sequence
\[
	Wh_3(\pi)\xrightarrow{\chi_{Wh}} Wh_1(\pi;\ZZ_2)\to \pi_0\cP(M,\partial M)\to Wh_2(\pi)\to 0
\]
where the ``Whitehead groups'' are defined as follows.
\begin{enumerate}
\item[a)] $Wh_2(\pi)$ is the cokernel of the following map induced by inclusion.
\[
	H_2M(\pi)'\to H_2GL(\ZZ[\pi])'\cong K_3(\ZZ[\pi])
\]
\item[b)] $Wh_1(\pi;\ZZ_2)=H_0(\pi;\ZZ_2[\pi])/\ZZ_2$ where $\ZZ_2$ represents the image of $H_0(1;\ZZ_2[1])$.
\item[c)] $Wh_3(\pi)$ is the quotient of $K_3(\ZZ[\pi])\cong H_3St(\ZZ[\pi])$ by the sum of $K_3(\ZZ)$ and the image of $H_3T(\pi)$.
\item[d)] $\chi_{Wh}:Wh_3(\pi)\to Wh_1(\pi;\ZZ_2)$ is the map induced by $\chi:K_3(\ZZ[\pi])\to H_0(\pi;\ZZ_2[\pi])$.
\end{enumerate}
\end{thm}
}
{\color{blue}
See \cite{HW}, \cite{H2}, \cite{Wh1a}, \cite{Wh1c}, \cite{Wh1b} for more about this theorem. Also, see \cite{What happens} for what happens when $\pi_2M\neq0$..
}

{
We shall assume that the reader is familiar with the Hatcher-Wagoner techniques for studying $\pi_0\cP(M,\partial M)$. Suppose that we have a generic one-parameter family of functions $f_t$ on $M\times I$ with critical points of index $i$ and $i+1$ and with only $i+1/i+1$ handle additions. Suppose also that we have a generic gradient-like vector field $v_t$ of $f_t$. Then we associate to the pair $(f_t,v_t)$ an element of $Wh_1(\pi;\ZZ_2)$ in the following way.
\begin{enumerate}
\item Choose a framing for the tangent bundle of the unstable manifold at each Morse point. This framing should vary smoothly with $t$.
\item At each birth-death point and at each $i+1/i+1$ handle addition the framing should be deformed so that they agree up to the sign of the last vector.
\item The framing for the index $i$ Morse points give a diffeomorphism of each lower index unstable sphere with the standard sphere ($\times I$).
\item Choose a path from each birth point to a base point $\ast$ of $M\times I\times I$.
\item Number the birth points $1,2,3$, etc.
\item We can associate to each $i+1/1+1$ handle addition an elementary operation $e_{jk}^{\pm u}$ where $u\in\pi$.
\item Let $J$ be the subspace of the set of lower index unstable spheres $n-i \cup_p S_p^{n-i+1}\times I$ given by $i+1/i$ intersections. ($n=\dim M$) Then $J$ is a framed 1-complex, that is, the generic points of $J$ are equipped with smoothly varying normal framings and these framings agree at singular points except for the last vector.
\item If $x\in J$, then $x$ determines a parameter value $t(x)$. This determines an invertible matrix $(r_{pq}(x))\in GL(\ZZ[\pi])$ which is the product of the elementary operations associated with the handle additions which occur before time $t(x)$. Let $(s_{qp}(x))=(r_{pq}(x))^{-1}$.
\item If $x\in J$ is a generic point, and if $x$ is an intersection between the $p$-th lower index unstable sphere with the $q$-th upper index stable sphere, then $x$ determines an element $\sigma(x)\in \pi$ by $\sigma(x)=[\lambda^{-1}_p\lambda(x)\lambda_q]$, i.e. this is the path given by going from $\ast$ to the $q$-th birth point, following the $q$-th upper index Morse line to time $t(x)$, going down the integral curve determined by $x$ until one reaches the $p$-th lower index Morse point at time $t(x)$, following this Morse line back to the $p$-th birth point and going back to $\ast$.

Let $c(x)=s_{qp}(x)\in \ZZ[\pi]$. Let $\overline c(x)$ be the image of $c(x)$ in $\ZZ_2[\pi]$.
\item If $u\in\pi$, let $J_u$ be the closure in $J$ of the set of generic points $x$ such that $\left< u,\sigma(x)\overline c(x)\right>=u$. Then $J_u$ is a closed 1-manifold with ``corners'' unless $u=1$.
\item The corners can be straightened out in a canonical way and we get a closed framed 1-manifold in $\bigcup_p S_p^{n-i+1}\times I$ for every nontrivial element of $\pi$. This determines an element of $\ZZ_2[\pi]/\ZZ_2[1]$ by adding up the elements $u\in\pi$ for which $J_u$ is nontrivially framed. $J_u$ is empty except for a finite number of $u\in\pi$ so there are only finitely many such elements.
\item Because of the choices made the element of $\ZZ_2[\pi]/\ZZ_2[1]$ is not well defined but its image in $Wh_1(\pi;\ZZ_2)$ is well defined and will be denoted $k(f_t,v_t)$.
\end{enumerate}

Let $c$ denote the choices made in (1), (2), (4) and (5). Then to the triple $(f_t,v_t,c)$ we can associate a sequence of elementary operations $g(f_t,v_t,c)\in F$ where $F$ is the free group generated by symbols $e_{ij}^u, i\neq j, u\in \pi$. Note that the image of $g(f_t,v_t,c)$ in $GL(\ZZ[\pi])$ must be a monomial matrix with entries from $\pm\pi$.
}

{
\begin{lem}\label{lem: 7.2}
Suppose that $(f_t,v_t,c)$ is deformed by commuting two consecutive letters in $g(f_t,v_t,c)$ of the form $e_{jk}^u,e_{j\ell}^v$. That is, we have a new triple $(f_t',v_t',c')$ with $g(f_t',v_t',c')=f[e_{jk}^u,e_{j\ell}^v]f^{-1} g(f_t,v_t,c)$. Then $k(f_t',v_t')=k(f_t,v_t)+\chi (f[[e_{jk}^u,e_{j\ell}^v]])$.
\end{lem}

\begin{proof}
The above deformation results in the following deformation of $J\cap S_p^{n-i+1}\times I$.
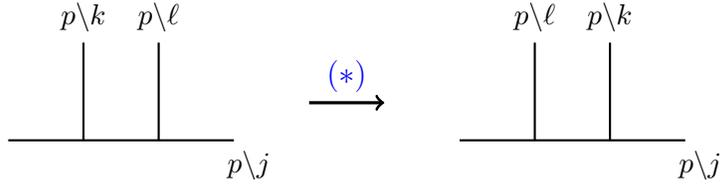
{
\begin{center}
\begin{tikzpicture}
\begin{scope}
	\draw[thick] (-1.5,0)--(1.5,0);
	\draw[thick] (-.5,0)--(-.5,1.3);
	\draw[thick] (.5,0)--(.5,1.3);
	\draw (-.5,1.3) node[above]{\small$p\backslash k$};
	\draw (.5,1.3) node[above]{\small$p\backslash \ell$};
	\draw (1.7,0) node[below]{\small$p\backslash j$};
\end{scope}
\draw[very thick,->] (2.5,.5)--(3.5,.5);
\draw[blue] (3,.5) node[above]{$(\ast)$};
\begin{scope}[xshift=6cm]
	\draw[thick] (-1.5,0)--(1.5,0);
	\draw[thick] (-.5,0)--(-.5,1.3);
	\draw[thick] (.5,0)--(.5,1.3);
	\draw (.5,1.3) node[above]{\small$p\backslash k$};
	\draw (-.5,1.3) node[above]{\small$p\backslash \ell$};
	\draw (1.7,0) node[below]{\small$p\backslash j$};
\end{scope}
\end{tikzpicture}
\end{center}
}
\noindent If the $p\backslash k$ and $p\backslash \ell$ ($\backslash =$ ``under'') segments both belong to $J_w$ then there are two possibilities for $J_w$.

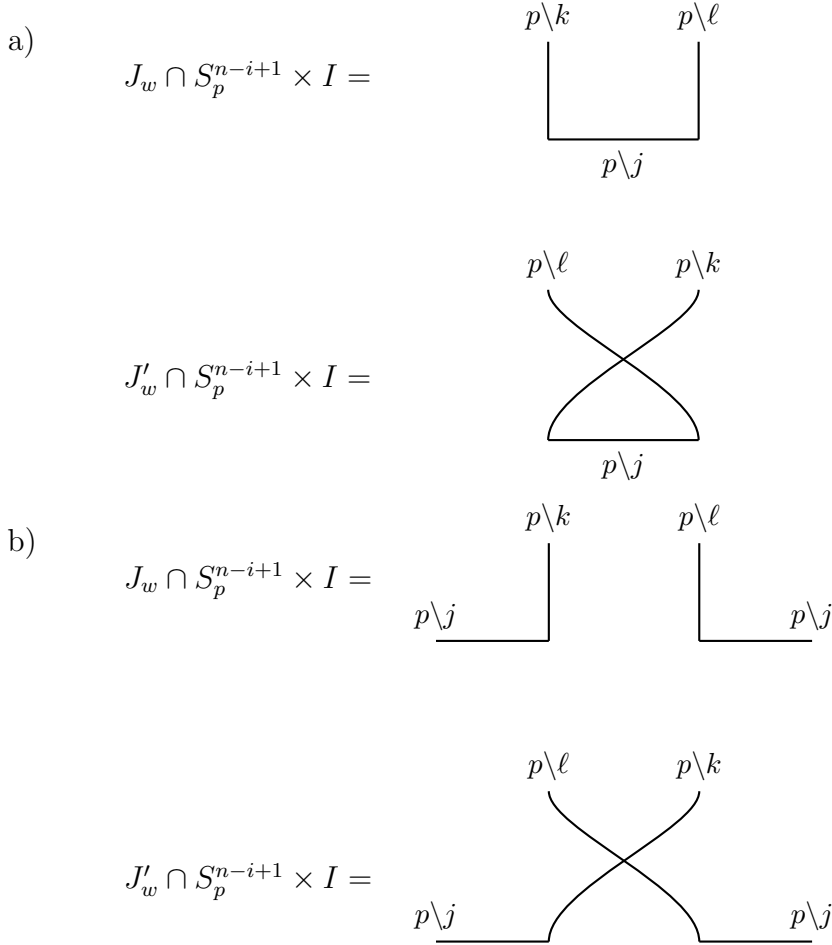
{
%
\begin{tikzpicture}
\begin{scope}
	\draw[thick] (-1,0)--(1,0);
	\draw[thick] (-1,0)--(-1,1.3);
	\draw[thick] (1,0)--(1,1.3);
	\draw (-1,1.3) node[above]{\small$p\backslash k$};
	\draw (1,1.3) node[above]{\small$p\backslash \ell$};
	\draw (0,0) node[below]{\small$p\backslash j$};
\draw (-5,.8) node{$J_w\cap S_p^{n-i+1}\times I=$};
\draw (-8,1.3) node{a)};
\end{scope}
\begin{scope}[yshift=-4cm]
	\draw[thick] (-1,0)--(1,0);
	\draw[thick] (-1,0)..controls (-1,.7) and (1,1.5)..(1,2);
	\draw[thick] (1,0)..controls (1,.7) and (-1,1.5)..(-1,2);
	\draw (1,2) node[above]{\small$p\backslash k$};
	\draw (-1,2) node[above]{\small$p\backslash \ell$};
	\draw (0,0) node[below]{\small$p\backslash j$};
\draw (-5,.8) node{$J'_w\cap S_p^{n-i+1}\times I=$};
\end{scope}
\end{tikzpicture}
%
}

{
%
\begin{tikzpicture}
\begin{scope}
	\draw[thick] (-2.5,0)--(-1,0);
	\draw[thick] (2.5,0)--(1,0);
	\draw[thick] (-1,0)--(-1,1.3);
	\draw[thick] (1,0)--(1,1.3);
	\draw (-1,1.3) node[above]{\small$p\backslash k$};
	\draw (1,1.3) node[above]{\small$p\backslash \ell$};
	\draw (-2.5,0) node[above]{\small$p\backslash j$};
	\draw (2.5,0) node[above]{\small$p\backslash j$};
\draw (-5,.8) node{$J_w\cap S_p^{n-i+1}\times I=$};
\draw (-8,1.3) node{b)};
\end{scope}
\begin{scope}[yshift=-4cm]
	\draw[thick] (-2.5,0)--(-1,0);
	\draw[thick] (2.5,0)--(1,0);
	\draw[thick] (-1,0)..controls (-1,.7) and (1,1.5)..(1,2);
	\draw[thick] (1,0)..controls (1,.7) and (-1,1.5)..(-1,2);
	\draw (1,2) node[above]{\small$p\backslash k$};
	\draw (-1,2) node[above]{\small$p\backslash \ell$};
	\draw (-2.5,0) node[above]{\small$p\backslash j$};
	\draw (2.5,0) node[above]{\small$p\backslash j$};
\draw (-5,.8) node{$J'_w\cap S_p^{n-i+1}\times I=$};
\end{scope}
\end{tikzpicture}
%
}\vs2

\noindent In both cases the deformation changes the framed bordism class of $J_w\cap S_p^{n-i+1}\times I$. The condition that $p\backslash k$ and $p\backslash \ell$ both belong to $J_w$ is expressed algebraically (see (10) above) by $\left<w,\sigma(p\backslash k)\overline c(p\backslash k)\right>=\left<w,\sigma(p\backslash \ell)\overline c(p\backslash \ell)\right>=w$, where

	$\overline c(p\backslash k) = s_{kp}$
	
	$\overline c(p\backslash \ell) = s_{\ell p}$
	
	$ \sigma(p\backslash k) = \sigma(p\backslash j) u$

	$ \sigma(p\backslash \ell) = \sigma(p\backslash j) v$

\noindent and $(r_{pq})=(s_{qp})^{-1}$ is the image of $f$ in $GL(\ZZ_2[\pi])$. However the deformation $(\ast)$ occurs in many places in $J\cap S_p^{n-i+1}\times I$ depending on the number of geometric $p\backslash j$ intersections which occur. If we add these up we see that the framed bordism class of $J_w\cap S_p^{n-i+1}\times I$ changes if and only if $\left<w,\sigma u s_{kp}\right>=\left<w,\sigma v s_{\ell p}\right>=w$ is true for an odd number of $\sigma$ in $r_{pj}\in \ZZ_2[\pi]$.

This can also be expressed by the formula
\[
	\left< w, r_{pj} \left<us_{kp},vs_{\ell p}\right>\right>=w.
\]
Adding these up for all $p$ we get that the framed bordism class of $J_w$ changes if and only if $\left<w,\chi_c(f[[e_{jk}^u, e_{j\ell}^v]])\right>=w$. Thus $\chi_c(f[[e_{jk}^u, e_{j\ell}^v]])$ is the set of all $w$'s for which $J_w$ changes.
\end{proof}
}

{
\begin{rem}\label{rem: 7.3}
One consequence of \eqref{lem: 7.2} is that even if $\chi_{Wh}$ is trivial and we get an exact sequence
\[
	0\to Wh_1(\pi;\ZZ_2)\to \pi_0\cP(M,\partial M)\to Wh_2(\pi)\to 0
\]
the splitting result for this sequence fails because $\chi_c$ is surjective.
\end{rem}

\begin{lem}\label{lem: 7.4}
If $(f_t,v_t,c)$ is deformed by changing $g(f_t,v_t,c)$ by an ``irrelevant'' Steinberg relation or by cancelling two handle additions (by definition this doesn't change $g$) $k(f_t,v_t)$ is unchanged.
\end{lem}

\begin{proof}
If $g(f_t,v_t,c)$ is changed be a relation $[e_{jk}^u,e_{\ell m}^v]$ where $j\neq \ell,m$ and $k\neq \ell$ then nothing happens because $J$ is changed by an isotopy. If two handle additions are cancelled or created, $J$ changes by a concordance:

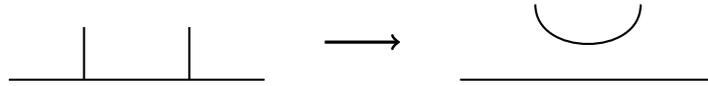
{
\begin{center}
\begin{tikzpicture}
\begin{scope}
	\draw[thick] (-1.7,0)--(1.7,0);
	\draw[thick] (-.7,0)--(-.7,.7);
	\draw[thick] (.7,0)--(.7,.7);
\end{scope}
\draw[very thick,->] (2.5,.5)--(3.5,.5);
\begin{scope}[xshift=6cm]
	\draw[thick] (-1.7,0)--(1.7,0);
	\draw[thick] (-.7,1)..controls (-.7,.3) and (.7,.3)..(.7,1);
\end{scope}
\end{tikzpicture}
\end{center}
}

\noindent or

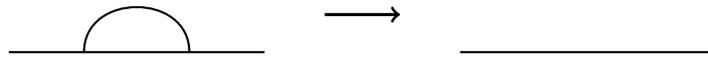
{
\begin{center}
\begin{tikzpicture}
\begin{scope}
	\draw[thick] (-1.7,0)--(1.7,0);
	\draw[thick] (-.7,0)..controls (-.7,.8) and (.7,.8)..(.7,0);
\end{scope}
\draw[very thick,->] (2.5,.5)--(3.5,.5);
\begin{scope}[xshift=6cm]
	\draw[thick] (-1.7,0)--(1.7,0);
\end{scope}
\end{tikzpicture}
\end{center}
}
\noindent If $g(f_t,v_t,c)$ is changed by a relation of the form {\color{blue}$e_{jk}^ue_{k\ell}^v e_{jk}^{-u} e_{j\ell}^{-uv} e_{k\ell}^{-v} $} then $J\cap S_p^{n-i+1}\times I$ changes as follows.
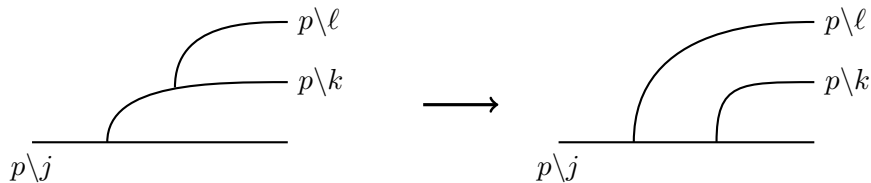
{
\begin{center}
\begin{tikzpicture}
\begin{scope}
	\draw[thick] (-1.7,0)--(1.7,0);
	\draw[thick] (-.7,0)..controls (-.7,.8) and (.7,.8)..(1.7,.8);
	\draw[thick] (.2,.73)..controls (.2,1.6) and (1.2,1.6)..(1.7,1.6);
	\draw (1.7,.8) node[right]{\small$p\backslash k$};
	\draw (1.7,1.6) node[right]{\small$p\backslash \ell$};
	\draw (-1.7,0) node[below]{\small$p\backslash j$};
\end{scope}
\draw[very thick,->] (3.5,.5)--(4.5,.5);
\begin{scope}[xshift=7cm]
	\draw[thick] (-1.7,0)--(1.7,0);
	\draw[thick] (.4,0)..controls (.4,.8) and (.7,.8)..(1.7,.8);
	\draw[thick] (-.7,0)..controls (-.7,1.6) and (1.2,1.6)..(1.7,1.6);
	\draw (1.7,.8) node[right]{\small$p\backslash k$};
	\draw (1.7,1.6) node[right]{\small$p\backslash \ell$};
	\draw (-1.7,0) node[below]{\small$p\backslash j$};
\end{scope}
\end{tikzpicture}
\end{center}
}
\noindent If we examine all eight possibilities for $J_w\cap S_p^{n-i+1}\times I$ we see that nothing happens. 
\end{proof}
}

{
To prove Theorem \ref{thm: 7.1} it is sufficient to show

\begin{thm}\label{thm: 7.5}
If $(f_t,v_t)$ is a lens-shaped one-parameter family of function on $M\times I$ with no handle additions then $(f_t,v_t)$ can be deformed to a one-parameter family with no singularities if and only if $k(f_t,v_t)$ is in the image of $\chi_{Wh}$.
\end{thm}

\begin{proof}
Suppose that $k(f_t,v_t)$ is in the image of $\chi_{Wh}$. Then there is an element $P$ of $\overline P(St(\ZZ[\pi]))$ which maps to $k(f_t,v_t)$. We can construct a thimble-shaped two parameter family of functions on $M\times I$ whose handle addition pattern is given by $P$. By Lemmas \ref{lem: 7.2}, \ref{lem: 7.4} this family is a null-deformation of a lens-shaped family with no handle additions with $Wh_1(\pi;\ZZ_2)$ invariant equal to $k(f_t,v_t)$. If is well known that two such families with the same $Wh_1(\pi;\ZZ_2)$ can be deformed into each other. 

The converse is not so easy. Suppose that there is a null-deformation of $(f_t,v_t)$. This produces a two-parameter family with ``boundary'' $(f_t,v_v)$. By a complicated procedure one can deform this two-parameter family fixing the boundary so that it is thimble-shaped in the same two indices of $f_t$. The exchange points can then be eliminated and we can look at the handle additions and we can read off a graph $P$ in $\overline P(St(\ZZ[\pi]))$ whose image in $Wh_1(\pi;\ZZ_2)$ is $k(f_t,v_t)$. The details can be found in \cite{I}. {\color{blue}See also \cite{Wh1b}.}
\end{proof}
}

{
\section{Appendix: Pictures as stability diagrams}\color{blue}

The ``pictures'' in this paper have been redrawn to match the ``stability diagrams'' or ``scattering diagrams'' which are recent developments in representation theory \cite{BST}, \cite{IT14}, \cite{ITW}. In two examples, the pictures from Morse theory and representation theory do not match: There are arcs missing from the Morse theory pictures. Our attempt to fix this purely esthetic problem has lead us to two new ideas. One is to introduce ``ghost handle slides'' to fill in the missing pieces on the Morse theory side.

The second uses the idea that pictures are intrinsically embedded in the Cartan subalgebra $H$ of the corresponding real Lie algebra and this interpretation works very nicely with Morse pictures. We insert missing pieces on the module theoretic side giving what we call ``ghost modules''. On the Morse theoretic side, in the first picture we get another interpretation of the generalized Grassmann invariant! The second picture shows the dual of the generalized Grassmann invariant. We will see that the dual has a better algebraic description than the original generalized Grassmann invariant.

\subsection{Picture group}

\begin{defn}\label{def: 8.1}
Given a picture $P$ for a group $G=\left<\cX\,|\,\cY\right>$ we define the \emph{picture group} $G(P)$ of $P$ to be the group $\left<\cX_0\,\cY_0\right>$ where $\cX_0$ is the subset of $\cX$ consisting of all labels $x\in\cX$ of the edges in $P$ and $\cY_0$ is the set of all $y\in\cY$ labeling the vertices of $P$. In particular $P$ will be a picture for $G(P)$.
\end{defn}

For example, let $P$ be a picture for the Steinberg group given by three circles as in Figure (1) in Definition \ref{def: 2.3}:
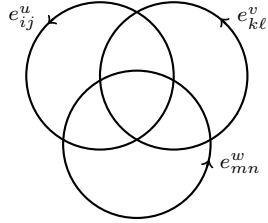
\begin{center}
\begin{tikzpicture}[scale=.7,black] 
%
\begin{scope}[xshift=-.7cm]
	\draw[thick] (0,0) circle[radius=1.4cm];
		\draw[thick,<-] (-.985,.99)--(-.88,1.09); 
		\draw (-1,1.1) node[left]{\tiny$e_{ij}^u$};
\end{scope}
\begin{scope}[xshift=.7cm]
	\draw[thick] (0,0) circle[radius=1.4cm];
	\draw[thick,->] (.99,.99)--(.89,1.09); 
		\draw (1,1.1) node[right]{\tiny$e_{k\ell}^v$};
\end{scope}
\begin{scope}[yshift=-1.3cm]
	\draw[thick] (0,0) circle[radius=1.4cm];
	\draw[thick,->] (1.33,-.43)--(1.37,-.3); 
	\draw (1.37,-.36) node[right]{\tiny$e_{mn}^w$};
\end{scope}
\end{tikzpicture}
\end{center}
This picture has edges with labels $a=e_{ij}^u, b=e_{k\ell}^v,c=e_{mn}^w$ with 6 vertices giving the relations that $a,b,c$ commute. Thus the picture group is $\ZZ^3$, the free abelian group on 3 generators.

Any relabeling of a picture gives a representation of the abstractly defined picture group. For example, the labels on the above picture give a linear representation $\rho$ of the picture group $\ZZ^3$ sending the generators $a,b,c$ to the commuting matrices $e_{ij}^u, e_{k\ell}^v,e_{mn}^w$.
}

{\color{blue}
\begin{eg}\label{eg: 7.6}\label{eg: 8.2}
Another example is the following picture which, as an unoriented picture is equivalent to the picture in Lemma \ref{lem: 2.2} (a). We use the convention (opposite to the one we have been using so far) that all edges are oriented clockwise instead of counterclockwise. They are oriented so that they are always curving to the right.
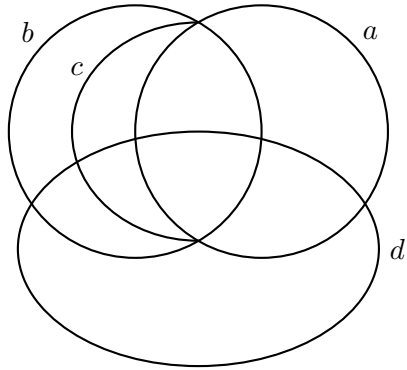
\begin{center}
\begin{tikzpicture}[scale=1.2,black]
\begin{scope}
\clip rectangle (-2,-1.3) rectangle (0,1.3);
\draw[thick] (0,0) ellipse [x radius=1.4cm,y radius=1.21cm];
\end{scope}
		\draw[thick] (-1.15,.7) node[left]{\small$c$};
\begin{scope}[xshift=-.7cm]
	\draw[thick] (0,0) circle[radius=1.4cm];
		\draw (-1,1.1) node[left]{\small$b$};
\end{scope}
\begin{scope}[xshift=.7cm]
	\draw[thick] (0,0) circle[radius=1.4cm];
		\draw (1,1.1) node[right]{\small$a$};
\end{scope}
\begin{scope}[yshift=-1.3cm]
	\draw[thick] (0,0) ellipse[x radius=2cm, y radius=1.3cm];
	\draw (2,0) node[right]{\small$d$};
\end{scope}
\end{tikzpicture}
\end{center}
We see the relation $c=[a,b]$ and its inverse at the two 5-valent vertices and the five 4-valent vertices give us the relation that $d$ commutes with $a,b,c$. Thus the picture group is $G(P)\cong F_2\times\ZZ$, the product of $\ZZ$ with the free group on 2 generators. The figure in Lemma \ref{lem: 2.2} (a) gives a linear representation of this group by sending $a,b,c,d$ to $e_{jk}^{-1},e_{ij}^{-1},e_{ik}^{-uv},e_{\ell m}^{-w}$, respectively. Generators are inverted since we reversed the orientations.
\end{eg}

The figure in Example \ref{eg: 7.6} and the other figures in \ref{lem: 2.2} and \ref{def: 2.3} have been redrawn to match the ``stability diagrams'' or ``wall-and-chamber'' structures for Dynkin quivers given by ``stability conditions'' which we now review.
}

{
\color{blue}
\subsection{Stability conditions}

There are several slightly different notions of stability and semi-stability for representations of a finite dimensional algebra $\Lambda$. Some refer to single modules and others refer to sequences of modules. To avoid confusion, we define and discuss stability of single $\Lambda$-modules in the sense of King \cite{King} and refer to the other notions of Reineke \cite{Reineke}, Bridgeland \cite{Bridgeland}, Keller \cite{Keller} in terms of the resulting sequences of modules which are called ``maximal green sequences'' or, equivalently, Harder-Narasimhan (HN) stratifications of $mod\text-\Lambda$. We take right $\Lambda$-modules.

Given a finite dimensional algebra $\Lambda$, its \emph{rank} $n$ is defined to be the number of nonisomorphic simple right $\Lambda$-modules. For any finitely generated $\Lambda$-module $M$, we define its \emph{dimension vector} $\undim M$ to be the vector in $\ZZ^n$ whose $i$-th coordinate is the number of times that the $i$th simple $S_i$ occurs in the composition series of $M$. One of the basic examples is $\Lambda=KQ$ the \emph{path algebra} of a quiver $Q$ without oriented cycles. This is the algebra over a field $K$ generated by all paths in the quiver including constant paths. Multiplication is by composition of paths from left to right. For example, if $Q$ is the quiver of type $A_2$:
\[
	1\xxrarrow\alpha 2
\]
there are 3 paths: the two constant paths $e_1,e_2$ and one path $\alpha$ of length 1. The path algebra $\Lambda=KQ$ is 3-dimensional with basis $e_1,e_2,\alpha$ and there are three indecomposable modules $S_1$, $S_2$, $P_1$ with dimension vectors $(1,0), (0,1), (1,1)$.

More generally, we consider \emph{Dynkin quivers} $Q$ which are oriented Dynkin diagrams. It is a classical result of Gabriel \cite{Gabriel} that there are only finitely many isomorphism classes of indecomposable $KQ$ modules and the dimension vectors of the indecomposable modules are the positive roots of the root system of the Dynkin diagram. For example, for quivers of type $A_3$ with any orientation, we expect to get 6 indecomposable modules with dimension vectors $(1,0,0), (0,1,0), (0,0,1), (1,1,0), (0,1,1), (1,1,1)$. For any positive root $\alpha$, we denote by $M_\alpha$ the corresponding indecomposable module. Thus $\undim M_\alpha=\alpha$.
}

{\color{blue}
We recall King's notion of semi-stability.
\begin{defn}\label{def: 6.8}\label{def: 8.3}\cite{King}
Let $\theta$ be a linear map $\RR^n\to \RR$, let $\Lambda$ be a finite dimensional algebra of rank $n$, e.g., $\Lambda=KQ$ where $Q$ is an acyclic quiver with $n$ vertices. Then a finitely generated $\Lambda$-module $M$ is said to be \emph{$\theta$-semistable} if $\theta(\undim M)=0$ and $\theta(\undim M')\le0$ for all $M'\subset M$. We call sometimes call $\theta$ a \emph{stability condition}.
\end{defn}

For Dynkin quivers $Q$, the indecomposable modules are $M_\alpha$, $\alpha\in\Phi^+$ with $\undim M_\alpha=\alpha$. We say that a positive root $\beta$ is a \emph{subroot} of $\alpha$ and we write $\beta\subset \alpha$ if $M_\beta\subset M_\alpha$. We say that $\alpha\in\Phi^+$ is $\theta$-semistable if $M_\alpha$ is $\theta$-semistable. I.e., 
\[
\theta(\alpha)=0\text{ and } \theta(\alpha')\le0\  \forall \alpha'\subset\alpha
\]
Recall that roots are linear functions on the Cartan subalgebra $H$ of the corresponding Lie algebra $L$. If $L$ is a real Lie algebra, roots are linear functions $\alpha:H\to \RR$, $\alpha\in H^\ast$. Since $(H^\ast)^\ast=H$, there is a 1-1 correspondence between stability conditions $\theta\in (H^\ast)^\ast$ and elements $h\in H$ so that, for all $\alpha\in\Phi^+$,
\[
	\theta(M_\alpha)=\alpha(h).
\]
In other words, $\theta$ is evaluation at $h$ and we sometimes write $\theta=\theta_h$.
\begin{eg}\label{eg: 7.8}\label{eg: 8.4}
Take $Q$ to be any Dynkin quiver of type $A_n$. The corresponding real Lie algebra is $L=sl(n+1,\RR)$, the algebra of real $(n+1)\times(n+1)$ matrices with trace $0$. The Cartan subalgebra is the set of diagonal matrices $h$ with diagonal entries $h_1,\cdots,h_{n+1}$ adding up to 0. The simple roots are the linear functions $\alpha_i:H\to \RR$ given by
\[
	\alpha_i(h)=h_i-h_{i+1}.
\]
Thus, the simple module $S_i$ with $\undim S_i=\alpha_i$ is $\theta_h$-semistable if and only if 
\[
	h_i=h_{i+1}.
\]
The other positive roots, which we denote by $\alpha_{ij}$ for $1\le i<j\le n+1$, are given by
\[
	\alpha_{ij}(h)=h_i-h_j.
\]
The corresponding module $M_{ij}$ with $\undim M_{ij}=\alpha_{ij}$ will be $\theta_h$-semistable if $h_i=h_j$ plus some other conditions depending on the orientation of the quiver $Q$.
\end{eg}
}

{\color{blue}
\subsection{Walls $D(M)$}

For a finite dimensional $\Lambda$-module $M$ we denote by $D(M)$ to be the set of all linear maps $\theta:\RR^n\to \RR$ so that $M$ is $\theta$-semistable. Thus
\[
	D(M)=\{\theta\in (\RR^n)^\ast\,|\, \theta(\undim M)=0 \text{ and } \theta(\undim M')\le0 \ \forall M'\subset M\}.
\]
We sometimes call $D(M)$ the \emph{semi-invariant domain} of $M$ since it is the set of ``weights'' of the nonzero semi-invariants which are defined on $M$. This is the original concept introduced by King \cite{King}.

When the sets $D(M)$ are drawn in the dual of $\RR^n$, we get what is called a ``stability diagram'' (also called ``scattering diagram''). It is only necessary to draw $D(M)$ for indecomposable $M$ since
\[
	D(M_1\oplus M_2)=D(M_1)\cap D(M_2).
\]
We will restrict to the case when $Q$ is a Dynkin quiver. Then $D(M)$ is embedded as a closed subspace of the Cartan subalgebra $H\cong\RR^n$. The Cartan subalgebra $H$ has an intrinsic metric given by the Killing form. But we prefer the Euclidean metric given by taking, as orthonormal basis, the dual of the basis of $H^\ast$ given by the simple roots.

\begin{eg}\label{eg: 7.9}\label{eg: 8.5}
The easiest example to draw is for the quiver $Q:1\to 2$. In this case there are only three indecomposable modules: the simple modules $S_1,S_2$ with dimension vectors $\alpha_1=(1,0)$ and $\alpha_2=(0,1)$ and the projective module $P_1$ with $\undim P_1=\beta=(1,1)$. ($P_2=S_2$ is also projective.) See Figure \ref{Figure00}. Since $S_1,S_2$ are simple, we have: 
\[
D(S_1)=\ker \alpha_1=\{h\in H\,|\, h_1=h_2\}
\]
\[D(S_2)=\ker \alpha_2=\{h\in H\,|\, h_2=h_3\}
\]
For $P_1$, with $\undim P_1=\alpha_1+\alpha_2=\alpha_{13}$, the only nontrivial submodule is $S_2$. So, the condition is that $\alpha_{13}(h)=0$ (or $h_1=h_3$) and $\alpha_{23}(h)\le 0$ or $h_2\le h_3$. (Equivalently, $\alpha_{12}(h)\ge0$.) So,
\[
D(P_1)=\{h\in H\,|\, h_1=h_3\ge h_2\}.
\]
The picture group $G(A_2)$ has generators $x(S_1),x(S_2),x(P_1)$ with the single relation 
\[
x(S_1)x(P_1)x(S_2)=x(S_2)x(S_1).
\]
Thus $G(A_2)\cong F_2$ is the free group on two generators.
Geometrically, $D(P_1)$ is in the portion of the hyperplane $\ker \beta=\{h\,|\,h_1=h_3\}$ on the negative side of $D(S_2)$ (since $S_2\subset P_1$) and on the positive side of $D(S_1)$ (since $S_1$ is a quotient of $P_1$). See Figure \ref{Figure00}.
\end{eg}
}

{%
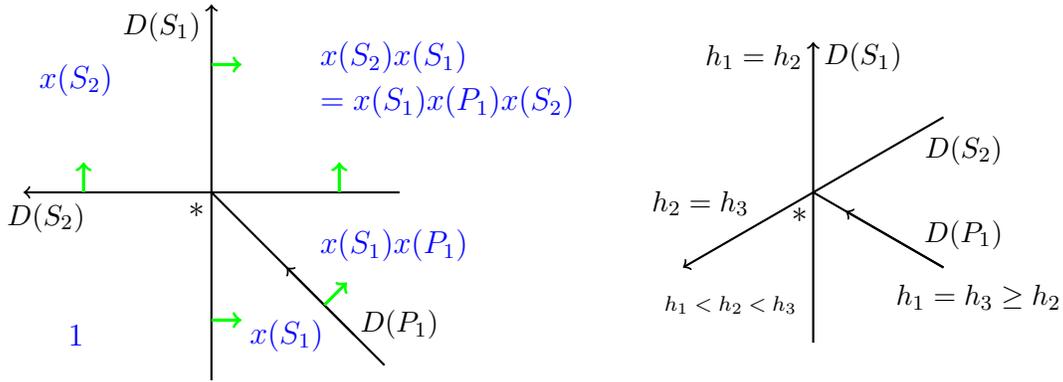
\begin{figure}[htbp]
\begin{center}
\begin{tikzpicture}[black]
{ 
\begin{scope}
\draw[thick,->] (0,-2.5)--(0,2.5);
\draw[thick,->] (2.5,0)--(-2.5,0);
\draw[thick] (2.3,-2.3)--(0,0);
\draw[thick,->] (1.9,-1.9)--(1,-1);
\draw (0,2.2) node[left]{\small$D(S_1)$};
\draw (-2.2,0) node[below]{\small$D(S_2)$};
\draw (2.5,-2.1) node[above]{\small$D(P_1)$};
\draw (-.2,-.23) node{$\ast$};
\draw[very thick,green,->] (0,1.7)--(.4,1.7);
\draw[very thick,green,->] (0,-1.7)--(.4,-1.7);
\draw[very thick,green,->] (-1.7,0)--(-1.7,.4);
\draw[very thick,green,->] (1.7,0)--(1.7,.4);
\draw[very thick,green,->] (1.5,-1.5)--(1.8,-1.2);
\draw[blue] (-1.8,-1.9) node{$1$};
\draw[blue] (-1.8,1.5) node{$x(S_2)$};
\draw[blue] (1,-1.9) node{$x(S_1)$};
\draw[blue] (1.3,-.7) node[right]{$x(S_1)x(P_1)$};
\draw[blue] (1.3,1.8) node[right]{$x(S_2)x(S_1)$};
\draw[blue] (1.3,1.2) node[right]{$=x(S_1)x(P_1)x(S_2)$};
\end{scope}
}
{
\begin{scope}[xshift=8cm]
\draw[thick,->] (0,-2)--(0,2);
\draw[thick,<-] (-1.732,-1)--(1.732,1);
\draw[thick] (0,0)--(1.732,-1);
\draw[thick,<-] (.433,-.25)--(1.732,-1);
\draw (0,1.8) node[left]{\small$h_1=h_2$};
\draw (0,1.8) node[right]{\small$D(S_1)$};
\draw (-1.5,-.45) node[above]{\small$h_2=h_3$};
\draw (2,.9) node[below]{\small$D(S_2)$};
\draw (2,-.9) node[above]{\small$D(P_1)$};
\draw (2.2,-1.7) node[above]{\small$h_1=h_3\ge h_2$};
\draw (-.18,-.3) node{$\ast$};
\draw (-1.1,-1.5) node{\tiny$h_1<h_2<h_3$};
\end{scope}
}
\end{tikzpicture}
\color{blue}
\caption{Picture for $1\to 2$ has 3 indecomposable representations $S_1,S_2,P_1$. The walls $D(S_1)$, $D(S_2)$ are the $y$ and $x$ axes respectively since they are $(1,0)^\perp$ and $(0,1)^\perp$. The wall $D(P_1)$ is the set of all $x$ in $(1,1)^\perp$ so that $x\cdot \undim S_2\le0$ since $S_2\subset P_1$. The picture group has generators $x(M)$ corresponding to each indecomposable module $M$.
On the right is the geometrically correct picture in the Cartan subalgebra $H$. The open region where $h_1<h_2<h_3$ is the subset of $H$ on which $\alpha_1,\alpha_2$ and $\beta$ are negative.}
\label{Figure00}
\end{center}
\end{figure}
}

{\color{blue}
In the special case $n=3$, the sets $D(M)$ will be drawn in the plane by taking their intersections with the unit sphere $S^2$ in $(\RR^3)^\ast$ and taking stereographic projection away from any fixed point on the negative side of all $D(S_i)$. In Euclidean coordinates, this region is the negative octant. For type $A_3$, this subset of $H$ is given by the condition $h_1<h_2<h_3<h_4$. We call the resulting planar diagram the ``picture'' for the algebra $\Lambda$. We give two examples, one of type $A_3$ and the second of type $A_2\times A_1$.
}

{\color{blue} 
{
\begin{eg}\label{eg: 7.10}\label{eg: 8.6}
Let $Q$ be the quiver of type $A_3$ given by
\[
	Q:\qquad1\to 2\to 3
\]
This has six indecomposable representations: 

Three simples: $S_1,S_2,S_3$ with dimension vectors $\alpha_{12},\alpha_{23},\alpha_{34}$.

Two of length 2: $I_2,P_2$ with $\undim I_2=(1,1,0)=\alpha_{13}$ and $\undim P_2=(0,1,1)=\alpha_{24}$. 

One of length 3: $P_1$ with $\undim P_1=(1,1,1)=\alpha_{14}$.

\noindent Since the arrows of $Q$ all go to the right, the subroots of $\alpha_{ij}$ are $\alpha_{kj}$ for $i<k<j$. So, the corresponding wall $D(M_{ij})$ is the set of all $h\in H$ where $\alpha_{ij}(h)=0$ and $\alpha_{kj}(h)\le0$ for all $i<k<j$.
\[
	D(M_{ij})=\{h\in H\,|\, h_i=h_j\text{ and } h_k\le h_j\ \forall i<k<j\}
\]
More precisely,
\[
	D(S_1)=\{h\in H\,|\, h_1=h_2\}
\]
\[
	D(S_2)=\{h\in H\,|\, h_2=h_3\}
\]
\[
	D(S_3)=\{h\in H\,|\, h_3=h_4\}
\]
\[
	D(I_2)=\{h\in H\,|\, h_1=h_3\ge h_2\}
\]
\[
	D(P_2)=\{h\in H\,|\, h_2=h_4\ge h_3\}
\]
\[
	D(P_1)=\{h\in H\,|\, h_1=h_4\ge h_2,h_3\}
\]
The picture group $G(A_3)$ has 6 generators $x(M_{ij})$ for $1\le i<j\le 4$ with 6 relations:
\begin{enumerate}
    \item $[x(M_{ij}),x(M_{k\ell})]=1$ if $i,j,k,\ell$ are distinct and ``noncrossing'' which means an even permutation will put these letter in increasing order.
    \item $[x(M_{ij}),x(M_{jk}]=x(M_{ik})$ when $i<j<k$.
\end{enumerate}
This group was first considered by Loday \cite{Loday} who called it the ``Stasheff group.'' Figure \ref{lem: 2.2}(a) indicates a linear representation $\rho$ of this group into $T_n(\ZZ[\pi])$ given by
\[
    \rho(x(S_1))=e_{ij}^{-u}, \rho(x(S_2))=e_{jk}^{-v}, \rho(x(S_3))=e_{k\ell}^{-w},\]
    \[\rho(x(I_2))=e_{ik}^{-uv}, \rho(x(P_2))=e_{j\ell}^{-vw}, \rho(x(P_1))=e_{i\ell}^{-uvw}.
\] This is compatible with the embedding into $H$, the Cartan subalgebra, if we replace indices $i,j,k,\ell$ with 1,2,3,4.
See Figure \ref{Figure01}
\end{eg}
}

{
\color{blue}
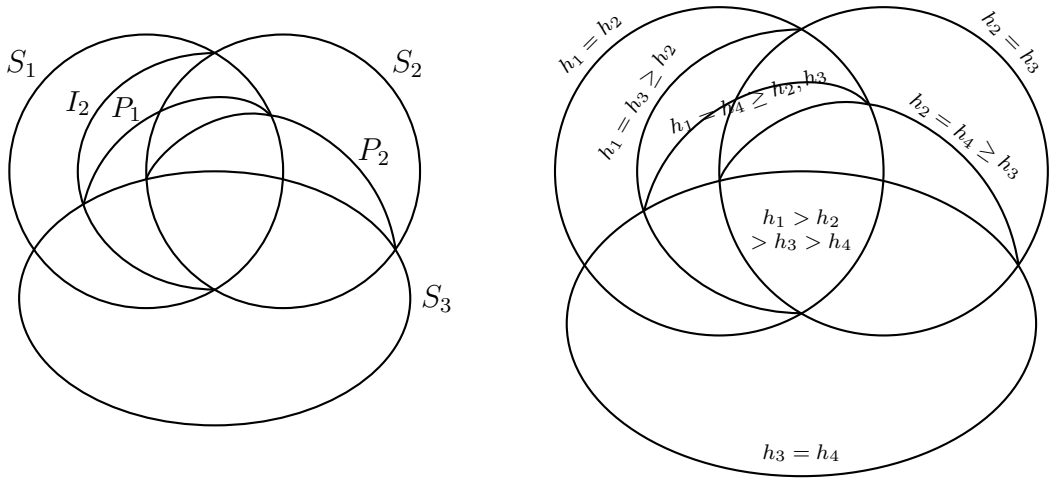
\begin{figure}[htbp]
\begin{center}
\begin{tikzpicture}[scale=1.3,black]
{
\begin{scope}
\clip rectangle (-2,-1.3) rectangle (0,1.3);
\draw[thick] (0,0) ellipse [x radius=1.4cm,y radius=1.21cm];
\end{scope}
%
				\draw[thick] (-1.15,.7) node[left]{\small$I_2$};
\begin{scope}[xshift=-.7cm]
	\draw[thick] (0,0) circle[radius=1.4cm];
		\draw (-1,1.1) node[left]{$S_1$};
\end{scope}
\begin{scope}[xshift=.7cm]
	\draw[thick] (0,0) circle[radius=1.4cm];
		\draw (1,1.1) node[right]{$S_2$};
\end{scope}
\begin{scope}[yshift=-1.3cm]
	\draw[thick] (0,0) ellipse[x radius=2cm, y radius=1.3cm];
	\draw (2,0) node[right]{$S_3$};
\coordinate (A) at (-.7,1.22);
\coordinate (A1) at (-.6,1.5);
\coordinate (B1) at (0,2);
\coordinate (B) at (.57,1.87);
\coordinate (B2) at (.77,1.87);
\coordinate (C) at (1.85,.5);
\coordinate (C1) at (1.7,1.5);
\coordinate (D) at (-1.34,.96);
\coordinate (E) at (1.4,1.405);
\coordinate (E2) at (1.35,1.5);
\coordinate (F1) at (-1.08,1.5);
\coordinate (F2) at (-0.9,1.7);
\coordinate (F) at (-1,1.595);
\draw[thick] (A)..controls (A1) and (B1)..(B);
\draw[thick] (B)..controls (B2) and (C1)..(C);
\draw[thick](B)..controls (.2,2.3) and (-1.1,2)..(D);
\draw (E2) node[right]{$P_2$};
\draw (F2) node[above]{\small$P_1$};
\end{scope}
}

{
\begin{scope}[xshift=6cm,scale=1.2]
\begin{scope}
\clip rectangle (-2,-1.3) rectangle (0,1.3);
\draw[thick] (0,0) ellipse [x radius=1.4cm,y radius=1.21cm];
\end{scope}
%
\begin{scope}
	\draw[thick] (-1.4,.6) node[rotate=60]{\tiny$h_1=h_3\ge h_2$};
\end{scope}
\begin{scope}[xshift=-.7cm]
	\draw[thick] (0,0) circle[radius=1.4cm];
		\draw (-1.1,1.1) node[rotate=45]{\tiny$h_1=h_2$};
\end{scope}
\begin{scope}[xshift=.7cm]
	\draw[thick] (0,0) circle[radius=1.4cm];
		\draw (1.1,1.1) node[rotate=-45]{\tiny$h_2=h_3$};
\end{scope}
\begin{scope}[yshift=-1.3cm]
	\draw[thick] (0,0) ellipse[x radius=2cm, y radius=1.3cm];
	\draw (0,-1.1) node{\tiny$h_3=h_4$};
\draw[thick] (-.7,1.22)..controls (-.6,1.5) and (0,2)..(.57,1.87);
\draw[thick] (.57,1.87)..controls (.77,1.87) and (1.7,1.5)..(1.85,.5);
\draw[thick] (.57,1.87)..controls (.2,2.3) and (-1.1,2)..(-1.34,.96);
\draw (1.4,1.6) node[rotate=-35]{\tiny$h_2=h_4\ge h_3$};
\draw (-0.45,1.9) node[rotate=20]{\tiny$h_1=h_4\ge h_2,h_3$};
\end{scope}
\draw (0,-.4) node{\tiny$h_1>h_2$};
\draw (0,-.6) node{\tiny$>h_3>h_4$};
\end{scope}
}
\end{tikzpicture}
\color{blue}
\caption{Picture for $1\to 2\to  3$ with labels $D(M)$, written simply as $M$, on the left and with corresponding subsets of $H$ on the right. Coordinates $h_i$ are decreasing in the central region and $h_i<h_{i+1}$ at points outside the $h_i=h_{i+1}$ circle. Compare this with Figure \ref{lem: 2.2}(b) where $e_{ij}$ labels are on the sets where $h_i=h_j$ using $i,j,k,\ell=1,2,3,4$.}
\label{Figure01}
\end{center}
\end{figure}
}
}

{ 
{ \color{blue}
\begin{eg}\label{eg: 7.11}\label{eg: 8.7}
Another example is given by the quiver
\[
	Q':\qquad 1\to 2\quad\quad  3
\]
This is a quiver of type $A_2\times A_1$ with corresponding real Lie algebra $sl(3;\RR)\times \sl(2;\RR)$. In this case, the Cartan subalgebra consists of $5\times 5$ diagonal matrices whose entries satisfy $h_1+h_2+h_3=0$ and $h_4+h_5=0$. There are only 4 positive roots: the three simple roots $\alpha_{12}, \alpha_{23},\alpha_{45}$ which are the dimension vectors of the three simple modules $S_1,S_2,S_3$ respectively, and the longer root $\alpha_{13}=\undim P_1$. The domains $D(S_i)$ will be the three ovals given by $h_1=h_2$, $h_2=h_3$ and $h_4=h_5$, respectively. See Figure \ref{Figure02}.
\end{eg}
}

{
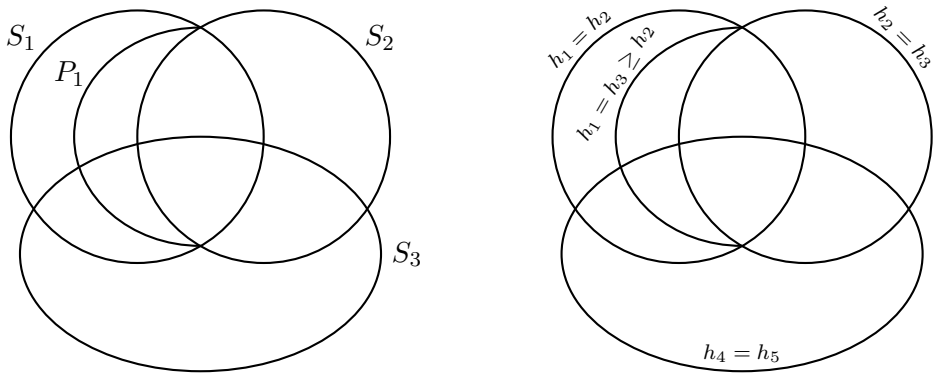
\begin{figure}[htbp]
\begin{center}
\begin{tikzpicture}[scale=1.2,black]
{\begin{scope} 
\begin{scope}
\clip rectangle (-2,-1.3) rectangle (0,1.3);
\draw[thick] (0,0) ellipse [x radius=1.4cm,y radius=1.21cm];
\end{scope}
		\draw[thick] (-1.15,.7) node[left]{\small$P_1$};
\begin{scope}[xshift=-.7cm]
	\draw[thick] (0,0) circle[radius=1.4cm];
		\draw (-1,1.1) node[left]{\small$S_1$};
\end{scope}
\begin{scope}[xshift=.7cm]
	\draw[thick] (0,0) circle[radius=1.4cm];
		\draw (1,1.1) node[right]{\small$S_2$};
\end{scope}
\begin{scope}[yshift=-1.3cm]
	\draw[thick] (0,0) ellipse[x radius=2cm, y radius=1.3cm];
	\draw (2,0) node[right]{\small$S_3$};
\end{scope}
\end{scope}
}
{\begin{scope}[xshift=6cm] 
\begin{scope}
\clip rectangle (-2,-1.3) rectangle (0,1.3);
\draw[thick] (0,0) ellipse [x radius=1.4cm,y radius=1.21cm];
\end{scope}
	\draw[thick] (-1.4,.6) node[rotate=60]{\tiny$h_1=h_3\ge h_2$};
\begin{scope}[xshift=-.7cm]
	\draw[thick] (0,0) circle[radius=1.4cm];
		\draw (-1.1,1.1) node[rotate=45]{\tiny$h_1=h_2$};
\end{scope}
\begin{scope}[xshift=.7cm]
	\draw[thick] (0,0) circle[radius=1.4cm];
		\draw (1.1,1.1) node[rotate=-45]{\tiny$h_2=h_3$};
\end{scope}
\begin{scope}[yshift=-1.3cm]
	\draw[thick] (0,0) ellipse[x radius=2cm, y radius=1.3cm];
	\draw (0,-1.1) node{\tiny$h_4=h_5$};
\end{scope}
\end{scope}
}
\end{tikzpicture}
\color{blue}
\caption{Picture for $1\to 2\quad  3$ which has only 4 indecomposable modules: $S_1,S_2,S_3$ and $P_1$ with $\undim P_1=(1,1,0)=\alpha_{13}$. The walls $D(S_1),D(S_2),D(S_3)$, shown on the left, are given by the equations $h_1=h_2$, $h_2=h_3$, $h_4=h_5$ and $D(P_1)$ is given, as before, by $h_1=h_3\ge h_2$ which means $D(P_1)$ is outside the $S_2$ circle and inside the $S_1$ circle. Coordinates in the Cartan subalgebra are shown on the right. This should be compared with Figure \ref{lem: 2.2}(a) with indices $i,j,k,\ell,m$ replaced with $1,2,3,4,5$.
}
\label{Figure02}
\end{center}
\end{figure}
}
}

{\color{blue}
Figures \ref{Figure01}, \ref{Figure02} are called  ``pictures'' for $\Lambda=KQ$. They are also called ``scattering diagrams'' or ``stability diagrams''. Stability diagrams are also defined in higher and lower dimensions. For example, Figure \ref{Figure00} is a stability diagram for $n=2$.
}

{
\color{blue}
\subsection{Ghost handle slides}
We have seen a nice correspondence between certain pictures of type $A_3$, embeddings of these into the Cartan subalgebra $H$ and Morse pictures which we saw earlier, especially if we replace indices $i,j,k,\ell$ with 1,2,3,4. However, this correspondence seems to break down in the cases of Figures \ref{def: 2.3}(3) and \ref{def: 2.3}(4) even after replacing indices with 1,2,3,4. The stability diagrams for the corresponding Dynkin quivers $A_3^+: 1\ot 2\to 3$ and $A_3^-:1\to 2\ot 3$ do not quite match these figures. But they are close and we will attempt to fix the discrepancies.
}

{
\color{blue}
\begin{eg}\label{eg: 8.8} Consider the quiver $1\ot 2\to 3$ which we refer to as $A_3^+$. This quiver has 6 indecomposable modules: the three simple modules $S_1=M_{12}$, $S_2=M_{23}$, $S_3=M_{34}$ and the longer modules $I_1=M_{13}$, $I_3=M_{24}$ and $P_2=M_{14}$. Containment relations are:
\[
	S_1\subset I_1,\quad S_3\subset I_3, \quad S_1,S_3\subset P_2.
\]
And these can be used to obtain the picture for this quiver which is shown on the left side of Figure \ref{Figure03}. On the right side of Figure \ref{Figure03} we have reproduced Figure \ref{def: 2.3}(3) with indices replaced with 1,2,3,4. There is a missing edge on the right hand side that we have indicated in blue.

Comparison of these two pictures implies that there is a homomorphism from the picture group $G(A_3^+)$ to $T_4(\ZZ[\pi])$, the $4\times 4$ upper triangular matrix group with coefficients in $\ZZ[\pi]$. More precisely, let $T_{34}$ denote the subgroup of $T_4(\ZZ[\pi])$ of all matrices with $(3,4)$ entry equal to $0$. We have a homomorphism $\rho_{uvw}:G(A_3^+)\to T_{34}$ given by
\[
\rho(x(S_2))=e_{12}^{-u}
\]
\[
\rho(x(S_1))=e_{23}^{-v}
\]
\[
\rho(x(S_3))=e_{24}^{-w}
\]
This implies $\rho(x(I_1))=e_{13}^{-uv}$, $\rho(x(I_3))=e_{14}^{-uw}$. Also, $\rho(x(P_2))=1$ since $x(P_2)=[x(S_1),x(I_3)]=[e_{23}^{-v},e_{14}^{-uw}]=1$. So, $x(P_2)$ is in the kernel of $\rho_{uvw}$ which is why that edge disappears in the Morse theory picture. We will fix this by lifting the homomorphism $\rho_{uvw}$ to a central extension of the subgroup $T_{34}$ of  $T_4(\ZZ[\pi])$.
\end{eg}
}

{
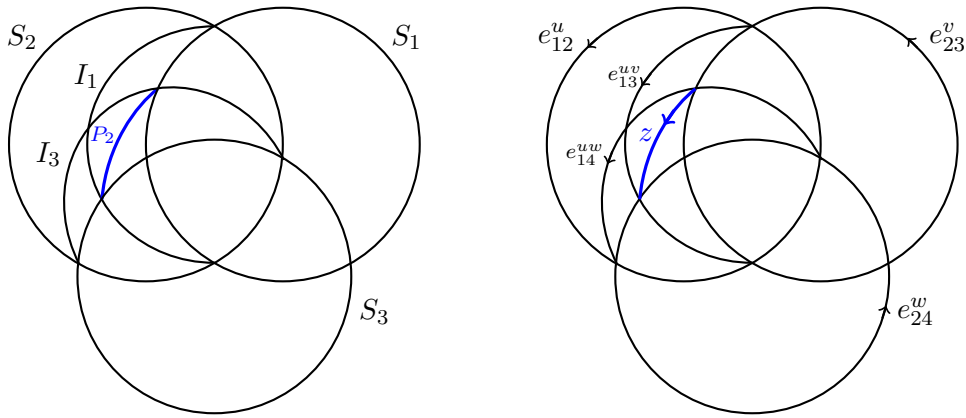
\begin{figure}[htbp]
\begin{center}

\begin{tikzpicture}[scale=1.3] 
%
{
\begin{scope} 
\begin{scope}
\clip rectangle (-2,-1.3) rectangle (0,1.3);
\draw[thick] (0,0) ellipse [x radius=1.3cm,y radius=1.21cm];
\end{scope}
\begin{scope}[xshift=.87mm]
		\draw[thick] (-1.15,.7) node[left]{\small$I_1$};
\end{scope}
\draw[very thick, blue] (-1.15,-.56)..controls (-1.1,0) and (-.8, .4)..(-.58,.57);
\draw[blue] (-.9,0.1) node[left]{\tiny$P_2$};
\begin{scope}[xshift=-3mm,yshift=-8mm]
		\draw[thick] (-1.15,.7) node[left]{\small$I_3$};
\end{scope}
\begin{scope}[xshift=-3.35mm,yshift=-6.9mm]
\begin{scope}[rotate=297]
\clip rectangle (-2,-1.3) rectangle (0,1.3);
\draw[thick] (0,0) ellipse [x radius=1.3cm,y radius=1.175cm];
\end{scope}
\end{scope}
\begin{scope}[xshift=-.7cm]
	\draw[thick] (0,0) circle[radius=1.4cm];
		\draw (-1,1.1) node[left]{\small$S_2$};
\end{scope}
\begin{scope}[xshift=.7cm]
	\draw[thick] (0,0) circle[radius=1.4cm];
		\draw (1,1.1) node[right]{\small$S_1$};
\end{scope}
\begin{scope}[yshift=-1.35cm]
	\draw[thick] (0,0) circle[radius=1.4cm];
	\draw (1.37,-.36) node[right]{\small$S_3$};
\end{scope}
\end{scope} 
}
{
\begin{scope}[xshift=5.5cm]
%
\begin{scope}
\clip rectangle (-2,-1.3) rectangle (0,1.3);
\draw[thick] (0,0) ellipse [x radius=1.3cm,y radius=1.21cm];
\end{scope}
\begin{scope}[xshift=.87mm]
		\draw[thick,<-] (-1.22,.6)--(-1.12,.73); 
		\draw[thick] (-1.1,.7) node[left]{\tiny$e_{13}^{uv}$};
\end{scope}
\begin{scope}[xshift=-3mm,yshift=-8mm]
		\draw[thick,<-] (-1.18,.6)--(-1.13,.73); 
		\draw[thick] (-1.1,.7) node[left]{\tiny$e_{14}^{uw}$};
\end{scope}
\begin{scope}[xshift=-3.35mm,yshift=-6.9mm]
\begin{scope}[rotate=297]
\clip rectangle (-2,-1.3) rectangle (0,1.3);
\draw[thick] (0,0) ellipse [x radius=1.3cm,y radius=1.175cm];
\end{scope}
\end{scope}
\begin{scope}[xshift=-.7cm]
	\draw[thick] (0,0) circle[radius=1.4cm];
		\draw[thick,<-] (-.985,.99)--(-.88,1.09); 
		\draw (-1,1.1) node[left]{\small$e_{12}^u$};
\end{scope}
\begin{scope}[xshift=.7cm]
	\draw[thick] (0,0) circle[radius=1.4cm];
	\draw[thick,->] (.99,.99)--(.89,1.09); 
		\draw (1,1.1) node[right]{\small$e_{23}^v$};
\end{scope}
\begin{scope}[yshift=-1.35cm]
	\draw[thick] (0,0) circle[radius=1.4cm];
	\draw[thick,->] (1.33,-.43)--(1.37,-.3); 
	\draw (1.37,-.36) node[right]{\small$e_{24}^w$};
\end{scope}
\draw[very thick, blue] (-1.15,-.56)..controls (-1.1,0) and (-.8, .4)..(-.58,.57);
\draw[very thick, blue,->]  (-.8, .35)--(-.9,.2) ;
\draw[blue] (-.9,0.1) node[left]{\small$z$};

\end{scope}
} 
\end{tikzpicture}
\color{blue}
\caption{Picture for $1\ot 2\to  3$ which has 6 walls, but the wall $D(P_2)$ vanishes in the Morse theory diagram since $x(P_2)$ is in the kernel of the representation $\rho_{uvw}:G(A_3^+)\to T_{34}$. We will fix this by introducing the ``ghost handle slide'' $z=z_{34}^{\overline{uvw}}$.}
\label{Figure03}
\end{center}
\end{figure}
}

{
\color{blue}
Continuing with our effort to replace the missing are on the right side of Figure \ref{Figure03} we construct a central extension $\widetilde T_{34}$ of $T_{34}$ with kernel the commutative ring $\ZZ[\pi/\pi']$ where $\pi/\pi'$ is the abelianization of $\pi$. This will be given by a cohomology class in $H^2(T_{34}(\ZZ[\pi]);\ZZ[\pi/\pi'])$ given by the factor set
\[
	f(X,Y)=-x_{23}y_{14}-x_{23}x_{12}y_{24}+x_{13}y_{24}\in \ZZ[\pi/\pi']
\]
where $x_{ij}\in\ZZ[\pi]$ is the $ij$ entry of $X$ and similarly for $Y$. It is an easy exercise to show that this satisfies the cocycle condition
\[
	f(Y,Z)-f(XY,Z)+f(XY,Z)-f(X,Y)=0
\]
assuming the variables $x_{ij},y_{k\ell}$ commute which is why we pass to $\ZZ[\pi/\pi']$. For consistency of notation we denote by $z_{34}^{\overline a}$ the central element of $\widetilde T_{34}$ corresponding to the image $\overline a\in \ZZ[\pi/\pi']$ of $a\in\ZZ[\pi]$. Thus
\[
	[e_{14}^{uw},e_{23}^v]=z_{34}^{\overline{vuw}}=z_{34}^{\overline{uvw}}=[e_{13}^{uv},e_{24}^w].
\]
Here we take any lifting of the generators $e_{12}^u, e_{23}^v$ and $e_{24}^w$ of $T_{34}$ to $\widetilde T_{34}$ and get induced liftings of $e_{13}^{uv}$ and $e_{14}^{uw}$. The commutators are given by the factor set:
\[
	f(e_{14}^{uw},e_{23}^v)-f(e_{23}^v,e_{14}^{uw})=\overline{vuw}
\]
\[
	f(e_{13}^{uv},e_{24}^w)-f(e_{24}^w,e_{13}^{uv})=\overline{uvw}
\]
This calculation shows the following.

\begin{prop}\label{prop: 8.9}
    The homomorphism $\rho_{uvw}:G(A_3^+)\to T_{34}(\ZZ[\pi])$ lifts to a homomorphism $\widetilde \rho_{uvw}:G(A_3^+)\to \widetilde T_{34}$ and any such lifting will send $x(P_2)$ to $z_{34}^{-\overline{vuw}}=z_{34}^{-\overline{uvw}}$ where $\overline{uvw}$ is the image of $uvw$ in $\pi/\pi'$.
\end{prop} 

The result is that the picture on the right hand side of Figure \ref{Figure03}, with the blue edge added and labeled $z$, is a picture for $\widetilde T_{34}(\ZZ[\pi])$. The existence of this picture is equivalent to a homomorphism $G(A_3^+)\to \widetilde T_{34}(\ZZ[\pi])$. This homomorphism sends $x(P_2)$ to $z^{-1}$ since the blue edge on the right is oriented counterclockwise, while all edges on the left are oriented clockwise.
}

{ 
\color{blue}
\begin{eg}\label{eg: 8.10}
Consider the quiver $1\to 2\ot 3$. We call this $A_3^-$. This has three simple modules and 3 more indecomposable modules $M_{13}=P_1$, $M_{24}=P_3$ and $M_{14}=I_2$. The picture for $A_3^-$ is shown on the left side of Figure \ref{Figure04}. The figure on the right side of Figure \ref{Figure04}, without the blue edge, is Figure \ref{def: 2.3}(4) with generic indices $i,j,k,\ell$ specialized to $1,3,4,2$. The existence of this figure is equivalent to the existence of a homomorphism $\rho:G(A_3^-)\to T_4(\ZZ[\pi])$ sending $x(S_1),x(S_2),x(S_3),x(P_1),x(P_3)$ to the inverses of $e_{13}^u,e_{34}^v, e_{23}^w,e_{14}^{uv},e_{24}^{wv}$ respectively and furthermore implies that this homomorphism has $x(I_2)$ in its kernel. As in the previous example, we will lift this homomorphism to a central extension of a subgroup of $T_4(\ZZ[\pi])$ so that $x(I_2)$ is sent to the central element, thus completing the picture.
\end{eg}
}

{

\begin{figure}[htbp]
\begin{center}

\begin{tikzpicture}[scale=1.3,black] 
%
{
\begin{scope}
\begin{scope}
\clip rectangle (-2,-1.3) rectangle (0,1.3);
\draw[thick] (0,0) ellipse [x radius=1.3cm,y radius=1.21cm];
\end{scope}
\begin{scope}[xshift=.87mm]
		\draw[thick] (-1.1,.7) node[left]{\small$P_1$};
\end{scope}
\begin{scope}[xshift=-3mm,yshift=-8mm]
		\draw[thick] (.5,-1.13) node[below]{\small$P_3$};
\end{scope}
\begin{scope}[xshift=3.5mm,yshift=-6.9mm]
\begin{scope}[rotate=62]
\clip rectangle (-2,-1.3) rectangle (0,1.3);
\draw[thick] (0,0) ellipse [x radius=1.3cm,y radius=1.175cm];
\end{scope}
\end{scope}
\begin{scope}[xshift=-.7cm]
	\draw[thick] (0,0) circle[radius=1.4cm];
		\draw (-1,1.1) node[left]{\small$S_1$};
\end{scope}
\begin{scope}[xshift=.7cm]
	\draw[thick] (0,0) circle[radius=1.4cm];
		\draw (1,1.1) node[right]{\small$S_2$};
\end{scope}
\begin{scope}[yshift=-1.35cm]
	\draw[thick] (0,0) circle[radius=1.4cm];
	\draw (1.37,-.36) node[right]{\small$S_3$};
\end{scope}
\begin{scope}[yshift=-1cm]
\draw[very thick, blue] (-1.15,.45)..controls (-1.1,0) and (-.9, -.3)..(-.7,-.4);
\draw[blue] (-.95,-0.1) node[left]{\tiny$I_2$};
\end{scope}
\end{scope} 
}

{
\begin{scope}[xshift=5.5cm] 
\begin{scope}
\clip rectangle (-2,-1.3) rectangle (0,1.3);
\draw[thick] (0,0) ellipse [x radius=1.3cm,y radius=1.21cm];
\end{scope}
\begin{scope}[xshift=.87mm]
		\draw[thick,<-] (-1.22,.6)--(-1.12,.73); 
		\draw[thick] (-1.1,.7) node[left]{\tiny$e_{14}^{uv}$};
\end{scope}
\begin{scope}[xshift=-3mm,yshift=-8mm]
		\draw[thick,->] (0.5,-1.164)--(0.6,-1.164); 
		\draw[thick] (.5,-1.2) node[below]{\tiny$e_{24}^{wv}$};
\end{scope}
\begin{scope}[xshift=3.5mm,yshift=-6.9mm]
\begin{scope}[rotate=62]
\clip rectangle (-2,-1.3) rectangle (0,1.3);
\draw[thick] (0,0) ellipse [x radius=1.3cm,y radius=1.175cm];
\end{scope}
\end{scope}
\begin{scope}[xshift=-.7cm]
	\draw[thick] (0,0) circle[radius=1.4cm];
		\draw[thick,<-] (-.985,.99)--(-.88,1.09); 
		\draw (-1,1.1) node[left]{\small$e_{13}^u$};
\end{scope}
\begin{scope}[xshift=.7cm]
	\draw[thick] (0,0) circle[radius=1.4cm];
	\draw[thick,->] (.99,.99)--(.89,1.09); 
		\draw (1,1.1) node[right]{\small$e_{34}^v$};
\end{scope}
\begin{scope}[yshift=-1.35cm]
	\draw[thick] (0,0) circle[radius=1.4cm];
	\draw[thick,->] (1.33,-.43)--(1.37,-.3); 
	\draw (1.37,-.36) node[right]{\small$e_{23}^w$};
\end{scope}
\begin{scope}[yshift=-1cm]
\draw[very thick, blue] (-1.15,.45)..controls (-1.1,0) and (-.9, -.3)..(-.7,-.4);
\draw[blue] (-.95,-0.1) node[left]{\small$z$};
\draw[very thick, blue,->] (-1,-.07)--(-.94,-.17) ;
\end{scope}
\end{scope} 
}
\end{tikzpicture}
\color{blue}
\caption{Picture for $1\to 2\ot  3$ has 6 walls, but the wall $D(I_2)$ vanishes in the Morse theory diagram since $x(I_2)$ is in the kernel of the corresponding representation $\rho_{uvw}:G(A_3^-)\to T_{12}$ where $T_{12}\subset T_4(\ZZ[\pi])$. We will insert a ``ghost'' element $z=z_{12}^{\overline{uvw}}\in \widetilde T_{12}$ to fill in this gap.}
\label{Figure04}
\end{center}
\end{figure}
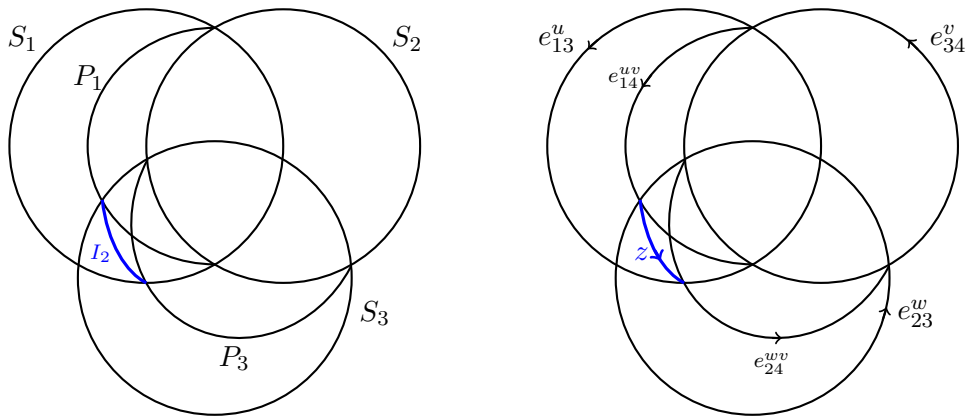
}

{
\color{blue}
Let $T_{12}(\ZZ[\pi])$ be the subgroup of $T_4(\ZZ[\pi])$ of upper triangular matrices with $x_{12}=0$. As in Example \ref{eg: 8.8}, we have a homomorphism
\[
\rho'_{uvw}:G(A_3^-)\to T_{12}
\]
given by $\rho'(x(S_1))=e_{13}^{-u}$, $\rho'(x(S_2))=e_{34}^{-v}$, $\rho'(x(S_3))=e_{23}^{-w}$. $x(I_2)$ is in the kernel of $\rho'_{uvw}$, so the corresponding arc is missing in the right hand figure of \ref{Figure04}. As before, take the central extension
\[
	0\to \ZZ[\pi/\pi']\to \widetilde T_{12}\to T_{12}(\ZZ[\pi])\to 0
\]
given by the factor set
\[
	f(X,Y)=x_{23}y_{14}+x_{13}y_{24}+x_{13}x_{23}y_{34}\in\ZZ[\pi/\pi']
\]
Then, for suitable liftings of the generators of $T_{12}$ to $\widetilde T_{12}$ we get
\[
	[e_{23}^{w},e_{14}^{uv}]=z_{12}^{\overline{wuv}}=z_{12}^{\overline{uwv}}=[e_{13}^{u},e_{24}^{wv}]
\]
where $z_{12}^{\overline s}$ is multiplicative notation for the element $\overline s\in \ZZ[\pi/\pi']$ where $\overline s$ denotes the image of $s\in\pi$ in $\pi/\pi'$.
\begin{prop}\label{prop: 8.11}
 The homomorphism $\rho'_{uvw}:G(A_3^-)\to T_{12}(\ZZ[\pi])$ lifts to a homomorphism $\widetilde \rho'_{uvw}:G(A_3^-)\to \widetilde T_{12}$ and any such lifting will send $x(I_2)$ to $z_{12}^{-\overline{wuv}}=z_{12}^{-\overline{uwv}}$.
 \end{prop}
 Thus the picture on the right hand side of Figure \ref{Figure04}, with the blue edge added and labeled $z$, is a picture for $\widetilde T_{12}(\ZZ[\pi])$. The homomorphism $G(A_3^-)\to \widetilde T_{12}(\ZZ[\pi])$ sends $x(I_2)$ to $z^{-1}$ since the blue edges on the two sides of Figure \ref{Figure04} have opposite orientation.
}

{
\color{blue}
In both examples \ref{eg: 8.8} and \ref{eg: 8.10}, we obtain ``ghost handle slides'' $z^u$ which will appear in the handle slide pattern. We believe that this generalized to larger numbers of handles sliding over each other. Basically, we believe that when two or more critical points are at the same critical level, there will be ghost handle slides over and under them. I will explain more about this in another paper.}

{ 
\subsection{Cartan subalgebra and generalized Grassmann invariant}\color{blue}
By considering an accurate embedding into the Cartan subalgebra $H$ of the Morse pictures for $A_3^+$ we are lead to a graphical method to keep track of the generalized Grassmann invariant. In the next subsection, we consider a similar construction for the Morse picture for $A_3^-$ which gives a graphical interpretation of the dual of the generalized Grassmann invariant which we expect to be equivalent to the original generalized Grassmann invariant: They probably differ by an involution of $K_3(\ZZ[\pi])$ and thus should have the same image.

In the previous section, we modified the group of upper triangular matrices to match the representation theory of certain quivers of type $A_3$ which we called $A_3^+$ and $A_3^-$. Here we will do the opposite: We modify the representation theory to match the pictures that come from Morse theory. The representation theory is rather involved, using Bridgeland stability conditions, Harder-Narasimhan stratifications of torsion classes and torsion-free classes and the introduction of ``ghost modules'' which we explain only in the two examples that we have. Since this is a topology paper and this section is mainly heuristic, we will skim over the details and hopefully come back to this in another paper.

Let $\Lambda$ be the path algebra of the $A_3$ quiver with straight orientation: $1\to 2\to3$. Recall that a \emph{torsion class} is a full subcategory of $mod\text-\Lambda$ which is closed under extensions and quotient modules \cite{Apostolos-Idun}. Let $\cG$ be the torsion class generated by $P_1$ and $P_2$. Equivalently, $\cG$ consists of all $\Lambda$-modules $M$ so that $\Hom(M,S_3)=0$.
}

{\color{blue}
\begin{defn}\label{def: 8.12}
    We define the \emph{relative stability diagram} or \emph{relative picture} of a torsion class $\cG$ as the union of \emph{relative walls} $D_\cG(M)\subset H$ for any $M\in\cG$ defined as follows. $D_\cG(M)$ is the set of all stability conditions $\theta$ so that $\theta(\undim M)=0$ and $\theta(\undim M')\le 0$ for all $M'\subset M$ which also lie in $\cG$. For example, $D_\cG(P_2)$ is the entire hyperplane in $H$ given by $h_2=h_4$ since $P_2$ has no subobjects which lie in the torsion class $\cG$.
\end{defn}
}

{\color{blue}
Figure \ref{fig: relative picture} shows what this construction does to the example at hand. This algebraic construction has the desired geometric property that the locus of the elementary matrices $e_{ij}^u$ is equal to a relative wall which is a subset of the hyperplane in $H$ given by $h_i=h_j$. This property was missing in Figures \ref{Figure03} and \ref{Figure04}.
}

{\color{blue}
The theory of maximal green sequences and Bridgeland stability are well-known. See for example \cite{Linearity}, \cite{Keller}. We use some simple examples to show how these concepts should be modified for torsion classes (but we only consider one torsion class: $\cG=\,^\perp S_3$).
}

{\color{blue}
\begin{eg}\label{eg: MGS1}
The following are maximal green sequences for the torsion class $\cG$.
\begin{enumerate}
\item $P_2,S_2,S_1$
\item $S_2,P_2,S_1$
\item $S_1,P_1,I_2,S_2,P_2.$
\end{enumerate}
These are indicated by the dashed green paths in Figure \ref{fig: relative picture b}. The first two are ``linear'' since the paths are straight lines. In higher dimensions, the pictures are impossible to draw. So, we use algebraic criteria: the Harder-Narasimhan stratification, a Hom-orthogonality condition, Bridgeland stability and the picture monoid which we review in these examples.
\end{eg}
}

{\color{blue}
\begin{defn}\label{def: Hom-orthogonal}
    A sequence of indecomposable objects $M_1,M_2,\cdots$ in a torsion class $\cG$ is said to have the \emph{relative forward Hom-orthogonality condition} if it satisfies:
\begin{enumerate}
    \item[(a)] For any $i\le j$ there are no forbidden morphisms $M_i\to M_j$. (A morphism is \emph{forbidden} if it is nonzero with kernel in the torsion class $\cG$.)
    \item[(b)] The sequence is maximal in the sense that it is not a subsequence of a longer sequence satisfying the same condition.
\end{enumerate}
\end{defn}

We verify that our three examples satisfy these conditions.
\begin{enumerate}
\item[(a)] The only nonzero forward morphisms are $P_2\to S_2$ in (1) and $P_1\to I_2$ in (3). But these are both allowed since both have kernel $S_3\notin\cG$.
\item[(b)] Example (3) is certainly maximal since it includes all 5 modules. The first two examples are missing $P_1$ and $I_2$. But $P_1$ must come before $P_2$ and after $S_1$. So, $P_1$ cannot be added to either sequence. Similarly, $I_2$ must come before $S_2$ and after $S_1$. So, $I_2$ cannot be inserted. So, all examples satisfy (b).
\end{enumerate}
}

{\color{blue}
\begin{defn}\label{def: HN for torsion classes}
A sequence of indecomposable objects $M_1,M_2,\cdots$ in a torsion class $\cG$ forms a \emph{Harder-Narasimhan stratification} of $\cG$ if any other object $X$ in $\cG$ has a filtration
\[
	0=X_0\subset X_1\subset X_2\subset\cdots
\]
so that each $X_i/X_{i-1}$ is a direct sum of copies of $M_i$. Furthermore, the sequence $\{M_i\}$ should be minimal, i.e., no $M_i$ has a filtration as above using the other terms in the sequence.
\end{defn}

We verify that our three examples satisfy these conditions. We note that the forward Hom-orthogonality condition implies minimality since a filtration of any $M_i$ using the others would give either a subobject of $M_i$ which comes before it or a quotient object which comes after it and this would violate Hom-orthogonality which we have already verified in the three examples. So, we only need to verify the first condition. For example (3) this is obvious since all 5 objects are there. For (1) there are two objects missing: $P_1,I_2$. For $P_1$ we have the extension: $0\to P_2\to P_1\to S_1\to 0$. Since $P_2$ comes before $S_1$ in (1), we have the required filtration of $P_1$. Similarly, the extension $0\to S_2\to I_2\to S_1\to 0$ gives the required filtration for $I_2$. Case (2) is similar.
}

{\color{blue}
We need one more concept: the ``picture monoid''. When we write down the relations defining the picture group, we see that they define a monoid which we call the \emph{picture monoid}. For the picture for the torsion class $\cG$ shown on the left side of Figure \ref{fig: relative picture}, we get the following.

\begin{defn}\label{def: picture monoid of G}
The \emph{picture monoid} $M(\cG)$ for $\cG$ has generators $S_1,S_2,P_1,P_2,I_2$ (avoiding the correct notation $x(S_1),x(S_2),\cdots$ for clarity and for applications) and the following relations (which come from the vertices of the picture).
\begin{enumerate}
\item $P_1I_2=I_2P_1$
\item $P_2S_2=S_2P_2$
\item $S_2S_1=S_1I_2S_2$
\item $P_2S_1=S_1P_1P_2$
\item $P_1S_2=S_2P_1$
\item $I_2P_2=P_2I_2$
\end{enumerate}
There is a special element $c=P_2S_2S_1$ which we call the \emph{Coxeter element} for $\cG$.
\end{defn}
}

{\color{blue}
\begin{rem}
It is clear that any MGS gives a word in the generators of the picture monoid whose product is equal to the Coxeter element. This is because a MGS is given by the wall-crossings of a path going from the unbounded region of the picture to the central triangle, but only crossing the walls in the inward direction. We call this a \emph{green path}. (For example, the path $\gamma_3$ in Figure \ref{fig: relative picture b} crosses the ghost wall $A$ in the wrong direction. So, $\gamma_3$ is green for $\cG$ but not green for the augmented picture with ghosts added.) Any two green paths can be deformed into each other and when the path crosses a vertex, the word will change by a relation in the picture monoid. We believe the converse also holds: any word in the generators whose product is the Coxeter element should be a MGS. This holds for picture groups of Dynkin quivers by \cite{IT14}.
\end{rem}
}

{\color{blue}
For our three examples, the picture monoid statement holds. (1) is the Coxeter element. By relation (2), $c=P_2S_2S_1=S_2P_2S_1$ giving (2). For (3), we compute:
\[
    S_1P_1I_2S_2P_2=_{(2),(6)} S_1P_1P_2I_2S_2=_{(4)} P_2S_1I_2S_2=_{(3)} P_2S_2S_1=c.
\]

}

{\color{blue}
We now come to the ghosts. There is only one object missing from $\cG$. It is $S_3$. This object can return as a ``ghost''. But it can have two ghosts! We call them $A$ and $B$. But if both ghosts appear, they must be next to each other in the order $AB$.

\begin{defn}\label{def: ghost}
Given a $M_1,M_2,\cdots$ a MGS for the torsion class $\cG$, the \emph{ghost} $B$ can appear between $M_i$ and $M_{i+1}$ if the sequence
\[
	M_1,M_2,\cdots,M_i,S_3,M_{i+1},\cdots
\]
is a MGS for $mod\text-\Lambda$.
\end{defn}

Example (1) cannot have any ghosts since $P_2$ cannot come before $S_2$ in any MGS for $mod\text-\Lambda$: That would violate the Hom-orthogonality condition. We can also see this from Figure \ref{fig: relative picture b} the path $\gamma_0$ representing example (1) does not cross either of the ``ghost walls'' shown in red and blue. Example (2) can have ghosts in two places and we get two sequences with ghosts:
\begin{enumerate}
\item $S_2,P_2,S_1,A,B$
\item $S_2,P_2,B,S_1$.
\end{enumerate}
These are given by the green paths $\gamma_1$ and $\gamma_2$ in Figure \ref{fig: relative picture b}. Ghost $B$ is explained by Definition \ref{def: ghost}. To explain ghost $A$ we will use the ``augmented picture monoid'' $M(\widetilde\cG)$, the monoid of the augmented picture $\widetilde\cG$ shown in Figure \ref{fig: relative picture b}.

\begin{rem}\label{rem:ghost A in next paper}
An explanation for ghost $A$ using Bridgeland stability conditions will be given in our next paper \cite{MoreGhosts}. The statement is: Every $X\notin\cG$ has one ghost for every object of $\cG$ which contains it. $B$ is the ghost of $S_3$ corresponding to $P_2\supset S_3$ and $A$ is the ghost of $S_3$ corresponding to $P_1$.
\end{rem}
}

{\color{blue}
\begin{defn}\label{def: augmented picture monoid}
The \emph{augmented picture monoid} $M(\widetilde\cG)$ is the monoid with generators $S_1,S_2,P_1,P_2,I_2,A,B$ and the following relations which come from the vertices of the left hand picture in Figure \ref{fig: relative picture b}.
\begin{enumerate}
\item $P_1I_2=I_2P_1A$
\item $P_2S_2=S_2P_2B$
\item $S_2S_1=S_1I_2S_2$
\item $P_2S_1=S_1P_1P_2$
\item $P_1S_2=S_2P_1$
\item $I_2P_2=P_2I_2$
\item $AP_2S_2=S_2P_2AB$
\item $BS_1=S_1AB$
\end{enumerate}
The special element $c=P_2S_2S_1$ is call the \emph{Coxeter element} for $\widetilde\cG$.
\end{defn}

{
We interpret these relations as indicating exact sequence involving ghost modules. E.g., we interpret (1) as a short exact sequence $0\to A\to P_1\to I_2\to 0$ which indicates that $A$ is a ``ghost'' of $S_3$. Similarly we imagine that (8) indicates an extension:
\[
	0\to B\to A\to S_1\to 0.
\]
Combining relations (2) and (7) we obtain the relation $AS_2P_2B=S_2P_2AB$.
In the augmented picture group $G(\widetilde\cG)$ we can cancel the $B$ to get $AS_2P_2=S_2P_2A$. Thus $A$ commutes with $S_2P_2$ in the augmented picture group.
}

If we set the ghosts $A,B$ equal to 1, we get $M(\cG)$. Therefore, we have an epimorphism of monoids $M(\widetilde\cG)\to M(\cG)$. We also note that, in the corresponding picture group $G(\widetilde \cG)$, relation (8) is redundant: it follows from the other relations. Also, if we mod out $A$, the resulting group $G(\widetilde\cG)/(A)$ is the pull back in the following diagram.
\[
\xymatrix{
G(\widetilde\cG)/(A)\ar[d]\ar[r] &
	F_2\ar[d]\\
G(\cG) \ar[r]^\varphi& 
	Z^2
	}
\]
where $F_2=\left<x,y\right>$ and $\varphi$ is given by $\varphi(P_2)=x,\varphi(S_2)=y$ and $\varphi(S_1)=\varphi(I_2)=\varphi(P_1)=1$. The element $B\in G(\widetilde\cG)/(A)$ comes from the commutator $x^{-1}y^{-1}xy\in F_2$. We looked for a similar description of $G(\widetilde\cG)$ but we couldn't find it.
}

{\color{blue}
We recall the standard definition of Bridgeland stability, reinterpret it in terms of the Cartan subalgebra and consider what happens when we restrict to a torsion class. 

A \emph{Bridgeland stability condition} is a linear map
\[
	\sigma:K_0(mod\text-\Lambda)\to \CC
\]
having the property that, for any module $M$, $\sigma(M)\in \CC$ is above the real axis, i.e., its imaginary part is positive. Each module has a \emph{slope} $\phi(M)$ given by $\cot \phi(M)=a/b$ if $\sigma(M)=a+bi$. $M$ is \emph{$\sigma$-semistable} if $\phi(M')\le \phi(M)$ for all $M'\subset M$ or, equivalently, $\cot \phi(M')\ge \cot\phi(M)$. The corresponding HN-stratification of $mod\text-\Lambda$ is given by taking $\sigma$-semistable modules in decreasing order of slope, or increasing order of $\cot\phi(M)$.
}

{\color{blue}
We can reinterpret this in terms of the Cartan subalgebra $H_\CC$ of the complex semisimple Lie algebra. We view elements of $H_\CC$ as given by $a+bi$ where $a,b\in H$. Since roots are elements of the dual of $H_\CC$, $\alpha:H_\CC\to \CC$, we can view Bridgeland's stability $\sigma$ as an element of $H_\CC$:
\[
	\sigma=h+ik,\quad h,k\in H
\]
Following the idea of \cite{Linearity}, we view this as a linear path in the real Cartan subalgebra $\gamma_\sigma:\RR\to H$ given by
\[
	\gamma_\sigma(t)=-h+tk.
\]
In the case of $A_n$, $H_\CC$ is the set of diagonal $(n+1)\times (n+1)$ complex matrices $h+ik$. The condition that $\alpha(h+ik)$ is above the real axis is equivalent to the condition that the entries of $k$ are decreasing
\[
	k_1>k_2>\cdots>k_{n+1}
\]
Therefore the path $\gamma(t)=-h+tk$ converges to the region in the middle triangle of the picture where $h_1>h_2>\cdots>h_{n+1}$. See Figure \ref{Figure01}. In Figure \ref{fig: relative picture b}, we see that $\gamma_0,\gamma_1,\gamma_2$ are such linear paths. $\gamma_3$ is ``nonlinear'' but it is given by a ``relative Bridgeland stability condition''.
}

{\color{blue}
\begin{thm}\label{thm: gamma passes through D(M)}
$M$ is $\sigma$-semistable if and only if the path $\gamma_\sigma$ passes through $D(M)$. The corresponding HN-stratification is given by taking these semi-stable modules in the order that $\gamma_\sigma$ passes the walls $D(M)$.
\end{thm}

\begin{proof} Although this holds in general \cite{Linearity}, the notation is easier if we take $M=M_\alpha$ for $\alpha:H_\CC\to \CC$ a positive root.

By definition, $\gamma_\sigma(t)\in D(M_\alpha)$ if the following two conditions are satisfied.

(1) $\alpha(\gamma_\sigma(t))=-\alpha(h)+t\alpha(k)=0$, i.e.,
\[
	t=\frac{\alpha(h)}{\alpha(k)}=\cot \phi(M_\alpha).
\]

(2) $\beta(\gamma_\sigma(t))=-\beta(h)+t\beta(k)\le 0$ for all $M_\beta\subset M_\alpha$, i.e.,
\[
	-\cot\phi(M_\beta)=\frac{-\beta(h)}{\beta(k)}\le -t=-\cot \phi(M_\alpha)
\]
or, equivalently, $\phi(M_\beta)\le \phi(M_\alpha)$. 

Thus $\gamma_\sigma(t)\in D(M_\alpha)$ for some $t$ implies $M_\alpha$ is $\sigma$-semistable. The converse holds by the same calculations. Since $t=\cot\phi(M_\alpha)$, the path $\gamma_\sigma$ passed through the semi-stable walls in the order of the HN-stratification.
\end{proof}
}

{\color{blue}
For torsion classes, there is a relative version of Bridgeland stability: 

\begin{defn}\label{def: relative Bridgeland stability}
A \emph{relative Bridgeland stability condition} for a torsion class $\cG$ is a linear map
\[
	\sigma:K_0(\cG)\to \CC
\] 
so that $\sigma(M)$ is above the real line for $M\in\cG$.
\end{defn}

In our example, such a relative stability condition is give by $h+ik$ where $k_1>k_2>k_3$ and $k_2>k_4$. We allow $k_3-k_4$, the imaginary part of $\alpha_3(h+ik)=\sigma(S_3)$ to be negative. This is what happens for the ``nonlinear'' MGS $S_1,P_1,I_2,S_2,P_2$. This is given by the relative Bridgeland stability condition shown in Figure \ref{fig: relative Bridgeland}. We see that $\sigma(S_3)$ must be below the real axis and $-\sigma(S_3)$ must be between $\sigma(I_2)$ and $\sigma(S_2)$. This can also be seen in Figure \ref{fig: relative picture b} where the corresponding path goes through the ghost $A$ in the wrong direction.
}

{
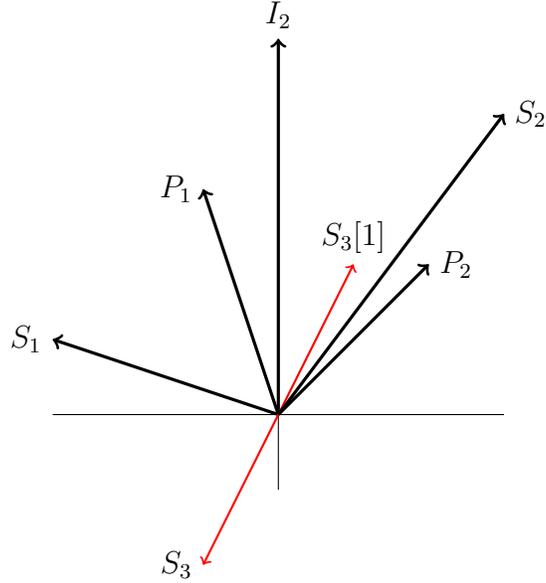
\begin{figure}[htbp]
\begin{center}
\begin{tikzpicture}[black]
\draw (-3,0)--(3,0);
\draw (0,-1)--(0,5);
\draw[very thick,->] (0,0)--(-3,1);
\draw[very thick,->] (0,0)--(-1,3);
\draw[very thick,->] (0,0)--(0,5);
\draw (-3,1)node[left]{$S_1$};
\draw (-1,3)node[left]{$P_1$};
\draw (0,5)node[above]{$I_2$};
\draw[red, thick,->] (0,0)--(-1,-2);
\draw (-1,-2)node[left]{$S_3$};
\draw[thick, red,->] (0,0)--(1,2);
\draw (1,2)node[above]{$S_3[1]$};
\draw[very thick,->] (0,0)--(3,4);
\draw (3,4)node[right]{$S_2$};
\draw[very thick,->] (0,0)--(2,2);
\draw (2,2)node[right]{$P_2$};
\end{tikzpicture}
\color{blue}
\caption{This shows a relative Bridgeland stability condition which gives the nonlinear MGS $S_1,P_1,I_2,S_2,P_2$. The stability condition puts $S_3$ below the real axis. In the augmented picture group this gives $S_1P_1I_2A^{-1}S_2P_2AB$ although the final $AB$ are not given by Bridgeland stability. This is best illustrated by the green path $\gamma_3$ in Figure \ref{fig: relative picture b}.}
\label{fig: relative Bridgeland}
\end{center}
\end{figure}
}

{\color{blue}
Turning to the ghosts $b=z_{2,34}^{v,w}$ and $a=z_{1,34}^{uv,uw}$ in the right hand figure in Figure \ref{fig: relative picture b}, we will see that these can be interpreted as the generalized Grassmann invariant. Recaall that each region in a picture for the Steinberg group can be labeled with an element of the Steinberg group We project to the general linear group $GL(\ZZ[\pi])$ to get matrices $r$ in each region.

Figure \ref{fig: relative picture b} gives the following three terms where the first term (1) can be taken as the definition of the operation $z_{i,jk}^{v,w}$. Recall that $\left<x,y\right>\in \ZZ_2[\pi]$ is the intersection of $x$ and $y$ in $\ZZ_2[\pi]$ considered as the set of finite subsets of $\pi$.
\begin{enumerate}
\item $z_{2,34}^{v,w}(r)=\sum_p r_{p2}\left< vs_{3p},ws_{4p}\right>$
\item $z_{2,34}^{v,w}(r')=\sum_p r'_{p2}\left< vs'_{3p},ws'_{4p}\right>$ where $r'=re_{12}^u$, $s'=e_{12}^{-u}s$.
\item $z_{1,34}^{uv,uw}(r'')=\sum_p r''_{p1}\left< uvs''_{3p},uws''_{4p}\right>$ where $r''=r e_{24}^w e_{23}^v$, $s''=e_{23}^{-v}e_{24}^{-w}s$.
\end{enumerate}
We see that $s_{3p}=s'_{3p}=s''_{3p}$ and $s_{4p}=s'_{4p}=s''_{4p}$. Also $r''_{p1}=r_{p1}$ and 
\[
	r'_{p2}=r_{p2}+r_{p1}u.
\]
This implies $(2)-(1)=(3)$ by distributivity since
\[
	r_{p1}u\left< vs'_{3p},ws'_{4p}\right>=r_{p1}\left< uvs'_{3p},uws'_{4p}\right>.
\]
Thus $(1),(2),(3)$ add up to zero in $\ZZ_2[\pi]$.  The identity $AS_2P_2=S_2P_2A$, which translates into $e_{24}^we_{23}^va=ae_{24}^we_{23}^v$ is reflected in the fact that
\[
	z_{1,34}^{uv,uw}(r'')=z_{1,34}^{uv,uw}(r)
\]
which is used in the above proof.

{
The identity $(1)+(2)+(3)=0$ proved above is the handle addition side of the last relation of Definition \ref{def: augmented picture monoid}: $(8): BS_1=S_1AB$. Converting to row operations: $re_{12}^{u}$ becomes $e_{12}^{-u}s$ and we get:
\[
	e_{12}^{-u}b(s)=bae_{12}^{-u}(s).
\]
This is an equation between two operations on the set $St(\ZZ[\pi])\times \ZZ_2[\pi]$ indicated as follows.
\[
	b(s,0)=(s,z_{2,34}^{v,w}(r))=\left(
	s, \sum_p r_{p2}\left< vs_{3p},ws_{4p}\right>
	\right)
\]
\[
	e_{12}^{-u}b(s,0)=(e_{12}^{-u}s,z_{2,34}^{v,w}(r))=\left(
	s', \sum_p r_{p2}\left< vs_{3p},ws_{4p}\right>
	\right)\tag{$\ast$}
\]
On the other side we have:
\[
	e_{12}^{-u}(s,0)=(e_{12}^{-u}s,0)=(s',0)
\]
\[
	ae_{12}^{-u}(s,0)=(s',z_{1,34}^{uv,uw}(r'))=\left(
	s', \sum_p r'_{p1}\left< uvs'_{3p},uws'_{4p}\right>
	\right)
\]
\[
	bae_{12}^{-u}(s,0)=(s',z_{1,34}^{uv,uw}(r')+z_{2,34}^{v,w}(r'))=\left(
	s', \sum_p r'_{p1}\left< uvs'_{3p},uws'_{4p}\right>
	+
	\sum_p r'_{p2}\left< vs'_{3p},ws'_{4p}\right>
	\right)
\]
which is equal to $(\ast)$ by the previous calculation.
}

This is a graphical interpretation of the calculation that the generalized Grassmann invariant $\chi$ is zero on the second order Steinberg relation given in Figure \ref{def: 2.3}(3). In the sequel \cite{MoreGhosts} we will show how ghosts appear and can be used to show that the generalized Grassmann invariant $\chi$ is zero on all second order Steinberg relation.
}

{
\begin{figure}[htbp]
\begin{center}

\begin{tikzpicture}[scale=1.3,black] 
%
{
\begin{scope} 
\begin{scope}
\clip rectangle (-2,-1.3) rectangle (0,1.3);
\draw[thick] (0,0) ellipse [x radius=1.3cm,y radius=1.21cm];
\end{scope}
\begin{scope}[xshift=.87mm]
		\draw[thick] (-1.15,.7) node[left]{\small$I_2$};
\end{scope}
\begin{scope}[xshift=-3mm,yshift=-8mm]
		\draw[thick] (-1.15,.7) node[left]{\small$P_1$};
\end{scope}
\begin{scope}[xshift=-3.35mm,yshift=-6.9mm]
\begin{scope}[rotate=297]
\clip rectangle (-2,-1.3) rectangle (0,1.3);
\draw[thick] (0,0) ellipse [x radius=1.3cm,y radius=1.175cm];
\end{scope}
\end{scope}
\begin{scope}[xshift=-.7cm]
	\draw[thick] (0,0) circle[radius=1.4cm];
		\draw (-1,1.1) node[left]{\small$S_1$};
\end{scope}
\begin{scope}[xshift=.7cm]
	\draw[thick] (0,0) circle[radius=1.4cm];
		\draw (1,1.1) node[right]{\small$S_2$};
\end{scope}
\begin{scope}[yshift=-1.35cm]
	\draw[thick] (0,0) circle[radius=1.4cm];
	\draw (1.37,-.36) node[right]{\small$P_2$};
\end{scope}
\draw[red,dashed] (-1.28,.18)--(1.39,-1.22);
\draw[red] (0.02,-.33) node[rotate=-25]{\small$S_3$};
\end{scope} 
}
{
\begin{scope}[xshift=5.5cm,scale=1.2]
%
\begin{scope}
\clip rectangle (-2,-1.3) rectangle (0,1.3);
\draw[thick] (0,0) ellipse [x radius=1.3cm,y radius=1.21cm];
\end{scope}
\begin{scope}[xshift=-3mm,yshift=-8mm]
	\draw[thick] (-1.34,.5) node[rotate=90]{\tiny$h_1=h_4\ge h_2$};
\end{scope}
\begin{scope}[xshift=-3.35mm,yshift=-6.9mm]
\begin{scope}[rotate=297]
\clip rectangle (-2,-1.3) rectangle (0,1.3);
\draw[thick] (0,0) ellipse [x radius=1.3cm,y radius=1.175cm];
\end{scope}
\end{scope}
	\draw[thick] (-1.05,.9) node[rotate=50]{\tiny$h_1=h_3\ge h_2$};
\begin{scope}[xshift=-.7cm]
	\draw[thick] (0,0) circle[radius=1.4cm];
		\draw (-1.1,1.1) node[rotate=45]{\tiny$h_1=h_2$};
\end{scope}
\begin{scope}[xshift=.7cm]
	\draw[thick] (0,0) circle[radius=1.4cm];
		\draw (1.1,1.1) node[rotate=-45]{\tiny$h_2=h_3$};
\end{scope}
\begin{scope}[yshift=-1.35cm]
	\draw[thick] (0,0) circle[radius=1.4cm];
	\draw (0,-1.2) node{\tiny$h_2=h_4$};
\end{scope}
\draw[red,dashed] (-1.28,.18)--(1.39,-1.22);
\draw[red] (0,-.35) node[rotate=-25]{\tiny$h_3=h_4$};
\end{scope}
} 
\end{tikzpicture}
\color{blue}
\caption{Relative picture for the torsion class $\cG=Gen(P_1\oplus P_2)$ in the module category of $A_3:1\to 2\to 3$.  This has 5 walls corresponding to the 5 indecomposable object of the torsion class: $S_1,S_2,P_1,P_2,I_2$. The circles are the domains of the minimal objects $S_1,S_2,P_2$. On the right are the coordinates of these sets in the Cartan subalgebra $H$. The ``ghost'' of the missing module $S_3$ is shown on the left, and is shown in $H$ on the right.}
\label{fig: relative picture}
\end{center}
\end{figure}
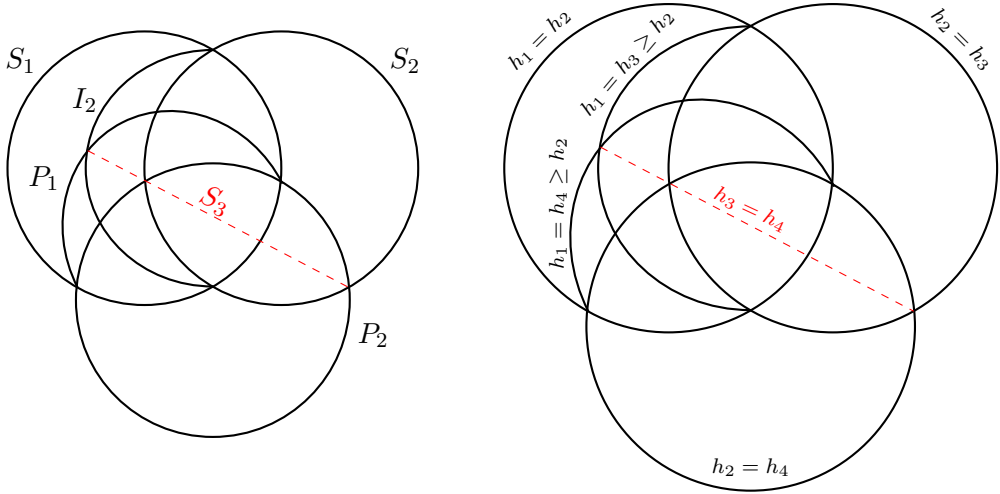
}

{
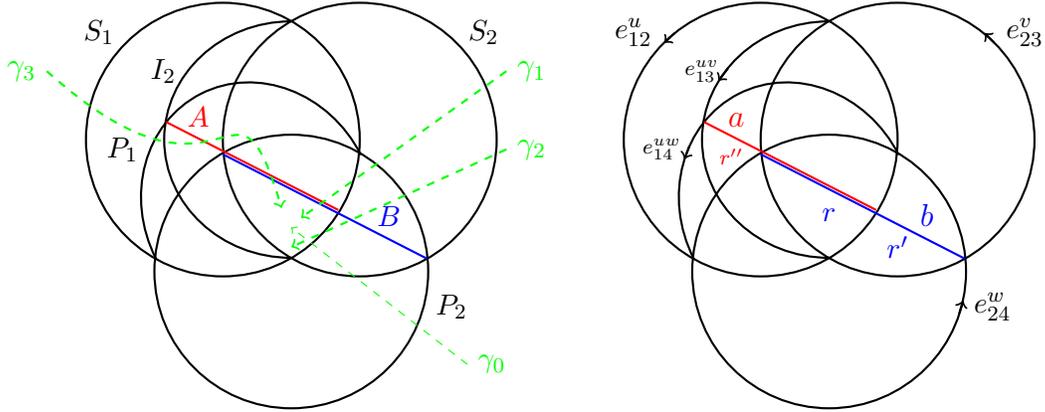
\begin{figure}[htbp]
\begin{center}

\begin{tikzpicture}[scale=1.3,black] 
%
{
\begin{scope} 
\begin{scope}
\clip rectangle (-2,-1.3) rectangle (0,1.3);
\draw[thick] (0,0) ellipse [x radius=1.3cm,y radius=1.21cm];
\end{scope}
\begin{scope}[xshift=.87mm]
		\draw[thick] (-1.15,.7) node[left]{\small$I_2$};
\end{scope}
\begin{scope}[xshift=-3mm,yshift=-8mm]
		\draw[thick] (-1.15,.7) node[left]{\small$P_1$};
\end{scope}
\begin{scope}[xshift=-3.35mm,yshift=-6.9mm]
\begin{scope}[rotate=297]
\clip rectangle (-2,-1.3) rectangle (0,1.3);
\draw[thick] (0,0) ellipse [x radius=1.3cm,y radius=1.175cm];
\end{scope}
\end{scope}
\begin{scope}[xshift=-.7cm]
	\draw[thick] (0,0) circle[radius=1.4cm];
		\draw (-1,1.1) node[left]{\small$S_1$};
\end{scope}
\begin{scope}[xshift=.7cm]
	\draw[thick] (0,0) circle[radius=1.4cm];
		\draw (1,1.1) node[right]{\small$S_2$};
\end{scope}
\begin{scope}[yshift=-1.35cm]
	\draw[thick] (0,0) circle[radius=1.4cm];
	\draw (1.37,-.36) node[right]{\small$P_2$};
\end{scope}
\draw[blue] (1,-1) node[above]{\small$B$};
\draw[red] (-.95,0.02) node[above]{\small$A$};
\draw[red,thick] (-1.28,.18)--(.48,-.72);
\draw[thick,blue] (-.7,-.15)--(1.39,-1.22);
\draw[thick,dashed,green,->] (-2.5,0.7)..controls (-.6,-1) and (-.6,1)..(-.1,-.7);
\draw[thick,dashed,green,->] (2.2,.7)--(0.1,-.8);
\draw[thick,dashed,green,->] (2.2,-.1)--(0,-1.1);

\draw[dashed,green,->] (1.8,-2.3)--(0,-.9);

\draw[green] (2.2,-.1) node[right]{$\small\gamma_2$};

\draw[green] (1.8,-2.3) node[right]{\small$\gamma_0$};
\draw[green] (2.2,.7) node[right]{$\small\gamma_1$};
\draw[green] (-2.5,0.7) node[left]{$\small\gamma_3$};

\end{scope} 
}
{
\begin{scope}[xshift=5.5cm]
%
\begin{scope}
\clip rectangle (-2,-1.3) rectangle (0,1.3);
\draw[thick] (0,0) ellipse [x radius=1.3cm,y radius=1.21cm];
\end{scope}
\begin{scope}[xshift=.87mm]
		\draw[thick,<-] (-1.22,.6)--(-1.12,.73); 
		\draw[thick] (-1.1,.7) node[left]{\tiny$e_{13}^{uv}$};
\end{scope}
\begin{scope}[xshift=-3mm,yshift=-8mm]
		\draw[thick,<-] (-1.18,.6)--(-1.13,.73); 
		\draw[thick] (-1.1,.7) node[left]{\tiny$e_{14}^{uw}$};
\end{scope}
\begin{scope}[xshift=-3.35mm,yshift=-6.9mm]
\begin{scope}[rotate=297]
\clip rectangle (-2,-1.3) rectangle (0,1.3);
\draw[thick] (0,0) ellipse [x radius=1.3cm,y radius=1.175cm];
\end{scope}
\end{scope}
\begin{scope}[xshift=-.7cm]
	\draw[thick] (0,0) circle[radius=1.4cm];
		\draw[thick,<-] (-.985,.99)--(-.88,1.09); 
		\draw (-1,1.1) node[left]{\small$e_{12}^u$};
\end{scope}
\begin{scope}[xshift=.7cm]
	\draw[thick] (0,0) circle[radius=1.4cm];
	\draw[thick,->] (.99,.99)--(.89,1.09); 
		\draw (1,1.1) node[right]{\small$e_{23}^v$};
\end{scope}
\begin{scope}[yshift=-1.35cm]
	\draw[thick] (0,0) circle[radius=1.4cm];
	\draw[thick,->] (1.33,-.43)--(1.37,-.3); 
	\draw (1.37,-.36) node[right]{\small$e_{24}^w$};
\end{scope}
\draw[blue] (1,-1) node[above]{$b$};
\draw[red] (-.95,0.02) node[above]{$a$};
\draw[red,thick] (-1.28,.18)--(.48,-.72);
\draw[thick,blue] (-.7,-.15)--(1.39,-1.22);
\draw[blue] (0,-.6) node[below]{\small$r$};
\draw[blue] (.7,-1.1) node{\small$r'$};
\draw[red] (-1,0.02) node[below]{\tiny$r''$};

\end{scope}
} 
\end{tikzpicture}
\color{blue}
\caption{Relative picture for the torsion class $\cG=Gen(P_1\oplus P_2)$ with two ``ghost modules'' added on the left, $A$ in red and $B$ in blue. The corresponding edges in the handle slide diagram are shown on the right. The labels are $a=z_{1,34}^{uv,uw}$ and $b=z_{2,34}^{v,w}$. The dashed green paths are maximal green sequence or, equivalently, Harder-Narasimhan stratifications of the torsion class.}
\label{fig: relative picture b}
\end{center}
\end{figure}
}

{\color{blue}
\subsection{Torsion-free classes and duality}

\begin{defn}\label{def: 8.14}
    Dually to Definition \ref{def: 8.12}, we define the \emph{corelative picture} of a torsion-free class $\cF$ as the union of \emph{corelative walls} $D^\cF(M)$ consisting of all stability conditions $\theta$ so that $\theta(\undim M)=0$ and $\theta(\undim M'')\ge 0$ for all quotients $M''$ of $M$ which lie in $\cF$. For example, $D^\cF(P_2)$ in Figure \ref{fig: corelative picture} is given by $h_2=h_4$ ($\theta_h(\undim P_2)=0$) and $h_2\ge h_3$ since $S_2\in\cF$ is a quotient of $P_2$.
\end{defn}
}

{
\begin{figure}[htbp]
\begin{center}

\begin{tikzpicture}[scale=1.3,black] 
%
{
\begin{scope}
\begin{scope}
\clip rectangle (-2,-1.3) rectangle (0,1.3);
\draw[thick] (0,0) ellipse [x radius=1.3cm,y radius=1.21cm];
\end{scope}
\begin{scope}[xshift=.87mm]
		\draw[thick] (-1.1,.7) node[left]{\small$P_1$};
\end{scope}
\begin{scope}[xshift=-3mm,yshift=-8mm]
		\draw[thick] (.5,-1.13) node[below]{\small$P_2$};
\end{scope}
\begin{scope}[xshift=3.5mm,yshift=-6.9mm]
\begin{scope}[rotate=62]
\clip rectangle (-2,-1.3) rectangle (0,1.3);
\draw[thick] (0,0) ellipse [x radius=1.3cm,y radius=1.175cm];
\end{scope}
\end{scope}
\begin{scope}[xshift=-.7cm]
	\draw[thick] (0,0) circle[radius=1.4cm];
		\draw (-1,1.1) node[left]{\small$I_2$};
\end{scope}
\begin{scope}[xshift=.7cm]
	\draw[thick] (0,0) circle[radius=1.4cm];
		\draw (1,1.1) node[right]{\small$S_3$};
\end{scope}
\begin{scope}[yshift=-1.35cm]
	\draw[thick] (0,0) circle[radius=1.4cm];
	\draw (1.15,-.8) node[right]{\small$S_2$};
\end{scope}
\begin{scope}[yshift=-1cm]
\end{scope}
\begin{scope}[xshift=-18.6mm,yshift=-19.3mm]
\draw[blue] (0,.2) node{\small$B$};
\draw[blue,thick] (-.6,0.17)--(.48,.72);
\draw[red,thick] (-.6,0.2)--(1,1);
\draw[red] (0,.5) node[above]{\small$A$};
\end{scope}
\begin{scope}[xshift=14mm,yshift=-2.8mm]
\draw[thick,blue] (-.7,.15)--(1.39,1.22);
\draw[blue] (1,1) node[below]{\small$B$};
\draw[thick,red] (.58,.835)--(1.39,1.25);
\draw[red] (1,1.05) node[above]{\small$A$}; 
\end{scope}

\end{scope} 
}

{
\begin{scope}[xshift=5.5cm] 
\begin{scope}
\clip rectangle (-2,-1.3) rectangle (0,1.3);
\draw[thick] (0,0) ellipse [x radius=1.3cm,y radius=1.21cm];
\end{scope}
\begin{scope}[xshift=.87mm]
		\draw[thick,<-] (-1.22,.6)--(-1.12,.73); 
		\draw[thick] (-1.1,.7) node[left]{\tiny$e_{14}^{uv}$};
\end{scope}
\begin{scope}[xshift=-3mm,yshift=-8mm]
		\draw[thick,->] (0.5,-1.164)--(0.6,-1.164); 
		\draw[thick] (.5,-1.2) node[below]{\tiny$e_{24}^{wv}$};
\end{scope}
\begin{scope}[xshift=3.5mm,yshift=-6.9mm]
\begin{scope}[rotate=62]
\clip rectangle (-2,-1.3) rectangle (0,1.3);
\draw[thick] (0,0) ellipse [x radius=1.3cm,y radius=1.175cm];
\end{scope}
\end{scope}
\begin{scope}[xshift=-.7cm]
	\draw[thick] (0,0) circle[radius=1.4cm];
		\draw[thick,<-] (-.985,.99)--(-.88,1.09); 
		\draw (-1,1.1) node[left]{\small$e_{13}^u$};
\end{scope}
\begin{scope}[xshift=.7cm]
	\draw[thick] (0,0) circle[radius=1.4cm];
	\draw[thick,->] (.99,.99)--(.89,1.09); 
		\draw (1,1.1) node[right]{\small$e_{34}^v$};
\end{scope}
\begin{scope}[yshift=-1.35cm]
	\draw[thick] (0,0) circle[radius=1.4cm];
	\draw[thick,->] (1.33,-.43)--(1.37,-.3); 
	\draw (1.37,-.36) node[right]{\small$e_{23}^w$};
\end{scope}
%
\begin{scope}[xshift=-18.6mm,yshift=-19.3mm]
\draw[blue] (0,.2) node{\small$b$}; 
\draw[blue,thick] (-.6,0.17)--(.48,.72);
\draw[red,thick] (-.6,0.2)--(1,1);
\draw[red] (0,.5) node[above]{\small$a$};
\draw[red] (.7,1) node{\small$r$};
\draw[blue] (-.3,.3) node[above]{\small$r'$}; 

\end{scope}
\begin{scope}[xshift=14mm,yshift=-2.8mm]
\draw[thick,blue] (-.7,.15)--(1.39,1.22);
\draw[blue] (1,1) node[below]{\small$b$};
\draw[thick,red] (.58,.835)--(1.39,1.25);
\draw[red] (1,1.05) node[above]{\small$a$}; 
\draw[blue] (.1,.5) node[above]{\small$r''$}; 

\end{scope}
\end{scope} 
} 
\end{tikzpicture}
\color{blue}
\caption{Co-relative picture for the torsion-free class $\cF=S_1^\perp$ in the module category of $A_3:1\to2\to3$. This has 5 walls for the indecomposable objects $S_2,S_3,P_1,P_2,I_2$ plus two ``ghost walls'' for $A,B$ which are the ghosts of $S_1$. The objects without quotients: $S_2,S_3,I_2$ form the three circles. On the right is the Morse picture from Figure \ref{def: 2.3}(3) plus the two ``ghosts'' $b=z_{12,3}^{u,w}$ and $a=z_{12,4}^{uv,wv}$ which we interpret as giving the ``dual generalized Grassmann invariant.}
\label{fig: corelative picture}
\end{center}
\end{figure}
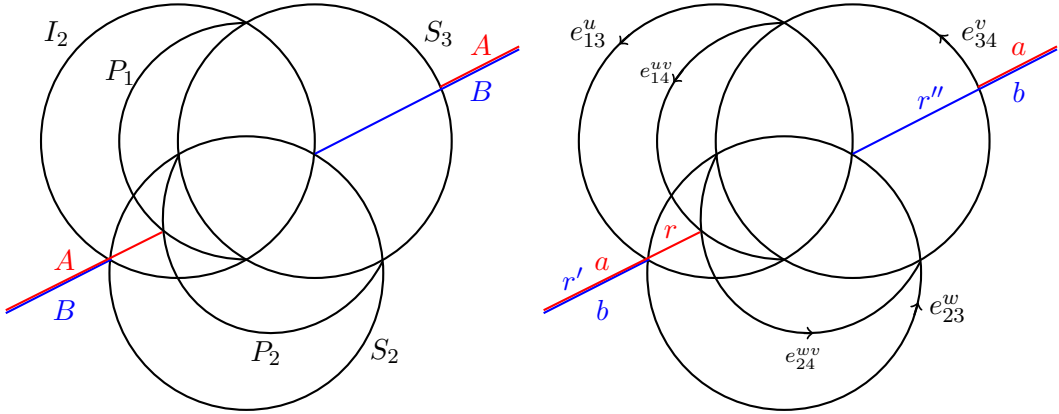
}

{\color{blue}
For the torsion-free case: We give a streamlined explanation for the augmented picture monoid and MGS's for both $\cF$ and $\widetilde\cF$.

\begin{defn}\label{def: augmented picture monoid}
The \emph{augmented picture monoid} $M(\widetilde \cF)$ is the cancellation monoid with generators $P_1,P_2,S_2,S_3,I_2,A,B$ with the relation that all of the MGS's are equal. These are
\begin{enumerate}
\item $S_3I_2S_2$
\item $I_2P_1S_3S_2$
\item $I_2P_1S_2P_2S_3$
\item $I_2S_2P_1P_2S_3$
\item $I_2S_2 AP_2P_1 S_3$
\item $ABS_2I_2 P_2P_1 S_3$
\item $ABS_2P_2I_2 P_1 S_3$
\item $ABS_2P_2 S_3I_2$
\item $AB S_3 S_2I_2$
\item $S_3B S_2I_2$
\end{enumerate}
The minimal relations can be seen by comparing two consecutive MGS's. For example, $(6)=(7)$ means $I_2P_2=P_2I_2$ and $(9)=(10)$ implies $ABS_3=S_3B$.
\end{defn}
}

{\color{blue}
The picture monoid $M(\cF)$ is the quotient of $M(\widetilde\cF)$ given by setting $A=B=1$. This gives 9 MGS's for $\cF$ since (9) and (10) give the same MGS for $\cF$: $S_3S_2I_2$. There is one more (nonlinear) MGS for $\cF$:

$(10')\ S_2I_2P_1P_2S_3$.

This is equivalent to the reduced (6): $S_2I_2 P_2P_1 S_3$ since $P_1$, $P_2$ commute in $M(\cF)$ by the relation $(4)=(5)$. If we attempt to lift the MGS $(10')$ up to $M(\widetilde\cF)$ we would get
\[
	AB S_2I_2A^{-1}P_1P_2S_3.
\]
Also, the relative Bridgeland stability condition for $\cF$ which makes $(10')$ stable puts $S^1$ below the real axis, resulting in the insertion of $A^{-1}$ for $M(\widetilde\cF)$ (similar to Figure \ref{fig: relative Bridgeland}).

The relative forward Hom-orthogonality condition is slightly different for torsion-free classes. The forbidden morphisms are those whose quotients lie in $\cF$. Ghost modules are also defined differently. A module $X\notin \cF$ has one ghost for every object of $\cF$ which maps onto $X$ \cite{MoreGhosts}. For example, $A$ is the ghost of $S_1$ corresponding to $P_1$ and $B$ corresponds to $I_2\onto S_1$. 
}

{\color{blue}
Finally, we come to the dual generalized Grassmann invariant. This comes from the right part of Figure \ref{fig: corelative picture}. The two ghost arcs $b=z_{12,3}^{u,w}$ and $a=z_{12,4}^{uv,wv}$ indicate the operations on the incidence matrices $r$ with $s=r^{-1}$ and 
\[
	r'= r e_{13}^u e_{23}^w,\quad s'= (r')^{-1}= e_{23}^{-w} e_{13}^{-u} s
\]
\[
	r''=r'e_{34}^{-v},\quad s''=e_{34}^vs'
\]
given by the following formulas.
\begin{enumerate}
\item $z_{12,3}^{u,w}(r')=\sum_p \left< r'_{p1}u,r'_{p2}w\right> s'_{3p}$
\item $z_{12,3}^{u,w}(r'')=\sum_p \left< r''_{p1}u,r''_{p2}w\right> s''_{3p}$
\item $z_{12,4}^{uv,wv}(r)=\sum_p \left< r_{p1}uv,r_{p2}wv\right> s_{4p}$
\end{enumerate}
We see that $r_{p1}=r'_{p1}=r''_{p1}$ and $r_{p2}=r'_{p2}=r''_{p2}$. The other terms are related by
\[
	s''_{3p}=s'_{3p}+vs'_{4p}=s'_{3p}+vs_{4p}.
\]
Therefore (2) is equal to (1) plus
\[
	\sum_p \left< r_{p1}u,r_{p2}w\right> vs_{4p}=(3).
\]
This can also be seen as an equation of right operations:
\[
    (r)e_{34}^{-v}ba=(r)be_{34}^{-v}
\]
coming from the relation $ABS_3=S_3B$ in Definition \ref{def: augmented picture monoid} applied to $r'$ where
\[
    (r')be_{34}^{-v}=(r',z_{12,3}^{u,w}(r'))e_{34}^{-v}=(r'e_{34}^{-u},z_{12,3}^{u,w}(r'))=(r'',z_{12,3}^{u,w}(r'))
\]
\[
    (r')e_{34}^{-v}ba)=(r'')ba=(r'',z_{12,3}^{u,w}(r'')+ z_{12,4}^{uv,wv}(r''))
\]
Thus the dual generalized Grassmann invariant given by summing up the expressions
\[
	z_{ij,k}^{u,v}(r)\in H_0(\pi;\ZZ_2[\pi])
\]
whenever curves $e_{ik}^u$ and $e_{jk}^v$ cross with incidence matrix $r$ at the inner corner of the crossing, is zero on this second order Steinberg relation and clearly vanishes on the other second order Steinberg relation. Therefore it induces a well-defined homomorphism
\[
	\chi':K_3(\ZZ[\pi])\to H_0(\pi;\ZZ_2[\pi]).
\]

The dual generalized Grassmann invariant has a nice algebraic interpretation: In the equation 
\[
	re_{ik}^ue_{jk}^w=re_{jk}^we_{ik}^u,
\]
we are using commutativity of addition: We are adding $u$ times the $i$th column of $r$ and $w$ times the $j$th column of $r$ to the $k$th column of $r$. When we do these in the other order, we are using commutativity of addition. The expression 
\[
	\left< r_{pi}u,r_{pj}v\right>
\]
counts the number of times that the same element of $\pi$ is commuted with itself in the $pk$ entry of $re_{ik}^ue_{jk}^w$. We multiply by $s_{kp}$ and sum over all $p$ to get an invariant. See \cite{A-infty} for more details about this interpretation of the generalized Grassmann invariant.

The hope is that ``ghost modules'' and ``augmented picture monoids'' can be defined more generally (see \cite{MoreGhosts} for one generalization) and that this will result in higher degree version of the generalized Grassmann invariant, giving formulas for the transgression (in the spectral sequence for $A(B\pi)$, Waldhausen's $A$-theory \cite{W} of the classifying space of $\pi$):
\[
	K_{n+2}(\ZZ[\pi])\to H_0(\pi;S[\pi])
\]
where $S=\pi_n^s(S^0)$, the $n$-th higher homotopy group of spheres.
}

{\color{blue}
\section*{Acknowledgements} The author thanks Gordana Todorov for many good years of working together on picture groups, picture spaces, maximal green sequences and, more recently, on ``ghost modules'' and what they might mean. The author is also grateful for the current support of the Simons Foundation, Grant \#686616. Work for the original paper ``The generalized Grassmann invariant'' was supported by NSF Grant No. MSC 76-08804.
}

\end{document}